\documentclass[11pt, a4paper]{amsart}
\pdfoutput=1 
\usepackage[dvips]{graphicx}
\usepackage{amssymb, amstext, amscd, amsmath, color}
\usepackage{epstopdf}
\usepackage{tikz,mathdots,enumerate}
\usepackage{pgfplots}
\usepackage{url}
\usepackage{tikz-cd}
\usetikzlibrary{arrows}

\definecolor{cof}{RGB}{219,144,71}
\definecolor{pur}{RGB}{186,146,162}
\definecolor{greeo}{RGB}{91,173,69}
\definecolor{greet}{RGB}{52,111,72}

\usepackage{hyperref}
\usepackage{hhline}

\begin{document}

\newtheorem{theorem}{Theorem}[section]
\newtheorem{corollary}[theorem]{Corollary}
\newtheorem{proposition}[theorem]{Proposition}
\newtheorem{lemma}[theorem]{Lemma}

\theoremstyle{definition}
\newtheorem{remark}[theorem]{Remark}
\newtheorem{definition}[theorem]{Definition}
\newtheorem{example}[theorem]{Example}
\newtheorem{conjecture}[theorem]{Conjecture}

\newtheorem{claim}[theorem]{Claim}

\newcommand{\FFock}{\mathcal{F}}
\newcommand{\kil}{\mathsf{k}}
\newcommand{\Hil}{\mathsf{H}}
\newcommand{\hil}{\mathsf{h}}
\newcommand{\Kil}{\mathsf{K}}
\newcommand{\Real}{\mathbb{R}}
\newcommand{\Rplus}{\Real_+}

\newcommand{\bC}{{\mathbb{C}}}
\newcommand{\bD}{{\mathbb{D}}}
\newcommand{\bN}{{\mathbb{N}}}
\newcommand{\bQ}{{\mathbb{Q}}}
\newcommand{\bR}{{\mathbb{R}}}
\newcommand{\bT}{{\mathbb{T}}}
\newcommand{\bX}{{\mathbb{X}}}
\newcommand{\bZ}{{\mathbb{Z}}}
\newcommand{\bH}{{\mathbb{H}}}
\newcommand{\BH}{{\B(\H)}}
\newcommand{\bsl}{\setminus}
\newcommand{\ca}{\mathrm{C}^*}
\newcommand{\cstar}{\mathrm{C}^*}
\newcommand{\cenv}{\mathrm{C}^*_{\text{env}}}
\newcommand{\rip}{\rangle}
\newcommand{\ol}{\overline}
\newcommand{\td}{\widetilde}
\newcommand{\wh}{\widehat}
\newcommand{\sot}{\textsc{sot}}
\newcommand{\wot}{\textsc{wot}}
\newcommand{\wotclos}[1]{\ol{#1}^{\textsc{wot}}}
 \newcommand{\A}{{\mathcal{A}}}
 \newcommand{\B}{{\mathcal{B}}}
 \newcommand{\C}{{\mathcal{C}}}
 \newcommand{\D}{{\mathcal{D}}}
 \newcommand{\E}{{\mathcal{E}}}
 \newcommand{\F}{{\mathcal{F}}}
 \newcommand{\G}{{\mathcal{G}}}
\renewcommand{\H}{{\mathcal{H}}}
 \newcommand{\I}{{\mathcal{I}}}
 \newcommand{\J}{{\mathcal{J}}}
 \newcommand{\K}{{\mathcal{K}}}
\renewcommand{\L}{{\mathcal{L}}}
 \newcommand{\M}{{\mathcal{M}}}
 \newcommand{\N}{{\mathcal{N}}}
\renewcommand{\O}{{\mathcal{O}}}
\renewcommand{\P}{{\mathcal{P}}}
 \newcommand{\Q}{{\mathcal{Q}}}
 \newcommand{\R}{{\mathcal{R}}}
\renewcommand{\S}{{\mathcal{S}}}
 \newcommand{\T}{{\mathcal{T}}}
 \newcommand{\U}{{\mathcal{U}}}
 \newcommand{\V}{{\mathcal{V}}}
 \newcommand{\W}{{\mathcal{W}}}
 \newcommand{\X}{{\mathcal{X}}}
 \newcommand{\Y}{{\mathcal{Y}}}
 \newcommand{\Z}{{\mathcal{Z}}}

\newcommand{\supp}{\operatorname{supp}}
\newcommand{\conv}{\operatorname{conv}}
\newcommand{\cone}{\operatorname{cone}}
\newcommand{\vspan}{\operatorname{span}}
\newcommand{\proj}{\operatorname{proj}}
\newcommand{\sgn}{\operatorname{sgn}}
\newcommand{\rank}{\operatorname{rank}}
\newcommand{\Isom}{\operatorname{Isom}}
\newcommand{\qIsom}{\operatorname{q-Isom}}
\newcommand{\Cknet}{{\mathcal{C}_{\text{knet}}}}
\newcommand{\Ckag}{{\mathcal{C}_{\text{kag}}}}
\newcommand{\rind}{\operatorname{r-ind}}
\newcommand{\lind}{\operatorname{r-ind}}
\newcommand{\ind}{\operatorname{ind}}
\newcommand{\coker}{\operatorname{coker}}
\newcommand{\ran}{\operatorname{ran}}
\newcommand{\Aut}{\operatorname{Aut}}
\newcommand{\Hom}{\operatorname{Hom}}
\newcommand{\GL}{\operatorname{GL}}
\newcommand{\tr}{\operatorname{tr}}
\newcommand{\comp}{\operatorname{comp}}

\newcommand{\te}{\tilde{e}}
\newcommand{\tf}{\tilde{f}}
\newcommand{\tu}{\tilde{u}}
\newcommand{\tv}{\tilde{v}}
\newcommand{\ta}{\tilde{a}}
\newcommand{\tb}{\tilde{b}}
\newcommand{\tc}{\tilde{c}}
\newcommand{\tx}{\tilde{x}}
\newcommand{\ty}{\tilde{y}}
\newcommand{\tz}{\tilde{z}}
\newcommand{\tw}{\tilde{w}}

\newcommand{\eqnwithbr}[2]{%
\refstepcounter{equation}
\begin{trivlist}
\item[]#1 \hfill $\displaystyle #2$ \hfill (\theequation)
\end{trivlist}}

\setcounter{tocdepth}{1}

\title[Rigidity of symmetric frameworks in normed spaces]{Rigidity of symmetric frameworks in normed spaces}

\author[D. Kitson]{Derek Kitson}
\email{derek.kitson@mic.ul.ie}
\address{Dept.\ Math.\ Comp.  \\ Mary Immaculate College\\ Thurles\\ Co.~Tipperary\\ Ireland.}
\address{Dept.\ Math.\ Stats.\\ Lancaster University\\Lancaster, LA1 4YF \\U.K.}
\thanks{D.K. supported by the Engineering and Physical Sciences Research Council [grant numbers EP/P01108X/1 and EP/S00940X/1].}

\author[A. Nixon]{Anthony Nixon}
\author[B. Schulze]{Bernd Schulze}
\email{a.nixon@lancaster.ac.uk, b.schulze@lancaster.ac.uk}
\address{Dept.\ Math.\ Stats.\\ Lancaster University\\ Lancaster LA1 4YF \\U.K. }


\begin{abstract}
We develop a combinatorial rigidity theory for symmetric bar-joint frameworks in a general finite dimensional normed space. In the case of rotational symmetry, matroidal Maxwell-type sparsity counts are identified for a large class of $d$-dimensional normed spaces (including all $\ell^p$ spaces with $p\not=2$). Complete combinatorial characterisations are obtained for half-turn rotation in the $\ell^1$ and $\ell^\infty$-plane. As a key tool, a new Henneberg-type inductive construction is developed for the matroidal class of $(2,2,0)$-gain-tight graphs. 
\end{abstract}

\maketitle
\tableofcontents


\section{Introduction}

The determination of rigidity and flexibility of bar-joint frameworks consisting of rigid bars connected at their ends by idealised joints is a highly active research area in discrete geometry with a long and rich history dating back to considerations of linkages, trusses and polyhedral structures by Maxwell, Cauchy and Euler, among others. Since bar-joint frameworks are suitable models for a variety of both man-made and natural structures (buildings, linkages, molecules, crystals, etc.), rigidity theory has a broad range of modern practical applications in fields such as engineering, robotics, CAD and materials science. (See \cite{HoDCG, HoGCP, W1} e.g.). This transfer of knowledge between fundamental and applied researchers is one of the motivations for exploring constraint systems in new geometric contexts, such as the normed spaces considered in this article (see also \cite{dam-lib,FHJV} for related problems). Another strong motivation comes from the potential for developing combinatorial Laman-type characterisations (\cite{Lamanbib}) of rigid graphs in any dimension, due to the amenability of the matroidal sparsity counts arising in some of these contexts.  

In this article, we consider first-order rigidity for bar-joint frameworks with a finite abelian symmetry group, developing both a general linear theory as well as sharp combinatorial results in the case of half-turn rotational symmetry in the $\ell^1$ and $\ell^\infty$-plane. This complements and extends work of Schulze \cite{BS2}, Jord\'{a}n,  Kaszanitzky and Tanigawa \cite{jkt}, Malestein and Theran  \cite{MaTh}, and Schulze and Tanigawa \cite{schtan} on symmetric frameworks in Euclidean space, and work of Kitson and Power \cite{kit-pow} and Kitson and Schulze \cite{kit-sch,kit-sch2} on infinitesimal rigidity in normed spaces.

In Sect.~\ref{Sect:SymmetricNormedSpaces}, we introduce the natural notion of a {\em framework complex} and develop several key tools for analysing frameworks with a finite abelian symmetry group acting freely on the vertex set. These include a decomposition theorem for the framework complex (which incorporates a block decomposition for the rigidity operator) and counting criteria for the accompanying group-labelled quotient graph (called a {\em gain graph}). For a large class of $d$-dimensional normed spaces, with $d\geq 2$, this leads to the identification of $(d,d,m)$-gain-tight gain graphs, with $m\in\{0,1,2,d-2\}$, as the underlying structure graphs for phase-symmetrically isostatic frameworks with rotational symmetry (see Corollary~\ref{cor:neccounts}). In contrast to Euclidean contexts, these classes of graphs are matroidal for all dimensions $d$, and are computationally accessible through associated pebble-game algorithms \cite{LS}. 

In Sect.~\ref{Sect:Inductive}, a new inductive construction is obtained for the class of $(2,2,0)$-gain-tight gain graphs (Theorem \ref{thm:recurse}). Previous recursive characterisations of $(2,2,m)$-gain-tight graphs, with $m\in\{1,2\}$, can be found in \cite{NS}. The construction presented here is necessarily more involved due to a step change in the possible minimum degree when $m=0$. Recursive constructions of classes of graphs are of fundamental importance in rigidity theory, occurring for example in Laman's landmark characterisation of rigidity in the Euclidean plane \cite{Lamanbib}. Of particular relevance are previous characterisations of classes of gain graphs \cite{jkt,NS} and characterisations where graph simplicity is required to be preserved \cite{NO14,NOP12}.

In Sect.~\ref{Sect:Gridlike}, geometric and combinatorial characterisations are obtained for the rigidity of two-dimensional frameworks with half-turn rotational symmetry in the $\ell^1$ and $\ell^\infty$-plane.
The geometric results (Theorem \ref{thm:C2gridgeom}) use an edge-colouring technique which expresses the gain graph of a phase-symmetrically isostatic framework as an edge-disjoint union of either two unbalanced spanning map graphs (defined in Sect.~\ref{Sect:FrameworkColours}), or two spanning trees. Combinatorial characterisations  are then obtained for graphs which admit a placement as a phase-symmetrically isostatic framework with half-turn rotational symmetry (Theorems \ref{thm:geom} and \ref{thm:geom2}) by combining these geometric results with the construction scheme from Sect.~\ref{Sect:Inductive}.  
The analogous problem for frameworks with reflectional symmetry requires different methods and was settled in \cite{kit-sch2}. 

\section{Symmetric frameworks and gain sparsity}
\label{Sect:SymmetricNormedSpaces}
The aim of this section is to derive necessary gain-graph counting conditions for  symmetrically isostatic bar-joint frameworks in normed spaces.  Throughout this article, $X$ denotes a finite dimensional real vector space with a norm $\|\cdot\|$ and dimension $d\geq 2$. The group of linear isometries of $X$ is denoted $\Isom(X,\|\cdot\|)$, or simply $\Isom (X)$. The complexification $\bC \otimes_\bR X$ is denoted $X_\bC$ and, for convenience, elementary tensor products of the form $\lambda\otimes x$ are denoted by $\lambda x$. 
Also,  $\Gamma$ will denote a finite abelian group with identity element $1$ and $\hat{\Gamma}$ will denote the dual group of characters $\chi:\Gamma\to \{z\in\bC:|z|=1\}$. 

\subsection{Bar-joint frameworks}
Let $G=(V,E)$ be a finite simple undirected graph and let $p=(p_v)_{v\in V}\in X^{V}$. If the components of $p$ are distinct vectors in $X$ then the pair $(G,p)$ is called a {\em bar-joint framework} in $X$. If $H$ is a subgraph of $G$ and $p_H=(p_v)_{v\in V(H)}$ then the pair $(H,p_H)$ is called a {\em subframework} of $(G,p)$.  Define,
\[f_G:X^{V}\to \bR^{E}, \,\,\,\,\,\,\,\,\, (x_v)_{v\in V}\mapsto (\|x_v-x_w\|)_{vw\in E}.\]
If $f_G$ is differentiable at $p$ then the bar-joint framework $(G,p)$ is said to be {\em well-positioned} in $X$.

\begin{lemma}{\cite[Proposition 6]{kit-sch}}
\label{lem:differential}
If $(G,p)$ is well-positioned in $X$ then the differential of $f_{G}$ at $p$ satisfies, \[df_{G}(p):X^V\to \bR^E, \,\,\,\,\,
(u_v)_{v\in V} \mapsto (\varphi_{v,w}(u_v-u_w))_{vw\in E},\] 
where, for each edge $vw\in E$, $\varphi_{v,w}:X\to \bR$ is the linear functional,
\[\varphi_{v,w}(x) = \|p_v-p_w\|\,\left(\lim_{t\to 0} \, \frac{1}{t}\left(\|p_v-p_w+tx\|-\|p_v-p_w\|\right)\right), \,\,\,\forall\,x\in X.\]
\end{lemma}

A {\em rigid motion} of $X$ is a family of continuous paths $\{\alpha_x:[-1,1]\to X\}_{x\in X}$
such that $\alpha_x(0)=x$ and $\|\alpha_x(t)-\alpha_y(t)\| =\|x-y\|$ for all pairs $x,y\in X$ and all $t\in [-1,1]$.
An {\em infinitesimal rigid motion} of $X$ is a vector field $\eta:X\to X$ with the property that $\eta(x)=\alpha'_x(0)$ for all $x\in X$, for some rigid motion $\{\alpha_x\}_{x\in X}$.  
The collection of all infinitesimal rigid motions of $X$ is a vector subspace of $X^X$, denoted $\T(X)$. 

Let $(G,p)$ be a bar-joint framework in $X$ and define $\rho_{(G,p)}: \T(X)\to X^V$, $\eta\mapsto(\eta(p_v))_{v\in V}$.
Note that if $(G,p)$ is well-positioned, then $df_G(p)\circ \rho_{(G,p)}=0$ (see \cite[Lemma 2.1]{kit-pow}).
The {\em framework complex} for $(G,p)$, denoted $\comp(G,p)$, is the chain complex,
\begin{eqnarray}
\label{chaincomplex1}
\begin{CD}
0 @>>> \T(X) @>\rho_{(G,p)}>> X^V @>df_G(p)>> \bR^E @>>> 0. \\
\end{CD}
\end{eqnarray}
The kernel of $df_G(p)$, denoted $\F(G,p)$, is referred to as the space of {\em infinitesimal flexes} of $(G,p)$, while the image of $\rho_{(G,p)}$, denoted $\T(G,p)$, is referred to as the space of {\em trivial} infinitesimal flexes of $(G,p)$.

\begin{definition}
A well-positioned bar-joint framework $(G,p)$ in a normed space $X$ is, 
\begin{enumerate}[(a)]
\item \emph{full} if the framework  complex $\comp(G,p)$ is exact at $\T(X)$.
\item \emph{infinitesimally rigid} if the framework  complex $\comp(G,p)$ is exact at $X^V$.
\item \emph{independent} if the framework  complex $\comp(G,p)$ is exact at $\bR^E$.
\end{enumerate}
\end{definition}
 
A  well-positioned  bar-joint framework is {\em isostatic} if it is both  infinitesimally rigid and  independent. Note that $\comp(G,p)$ is a short exact sequence if and only if $(G,p)$ is both full and isostatic. 

\subsection{Symmetric graphs}
A {\em $\Gamma$-symmetric graph} is a pair $(G,\theta)$ where $G$ is a finite simple undirected graph with automorphism group $\Aut(G)$ and $\theta:\Gamma\to\textrm{Aut}(G)$ is a group homomorphism. It is assumed throughout this article that $\theta$ acts freely on the vertex set of $G$. Thus $v\not=\theta(\gamma)v$ for all $v\in V$ and for all $\gamma\in \Gamma$ with $\gamma\not=1$. For convenience, we suppress $\theta$ and denote $\theta(\gamma)$ by  $\gamma$ for each group element $\gamma\in\Gamma$. Also, for each edge $e=vw\in E$ we denote by $\gamma e$, or $\gamma(vw)$, the edge in $E$ which joins the vertices $\gamma v$ and $\gamma w$. 

\begin{proposition} 
\label{lem:decomp}
Let $(G,\theta)$ be a $\Gamma$-symmetric graph and let $\tau:\Gamma\to\Isom(X)$ be a group representation.
\begin{enumerate}[(i)]
\item
$(X_\bC)^{V}=\bigoplus_{\chi\in\hat{\Gamma}} X_{\chi}$ where,
\[X_\chi = \{x=(x_v)_{v\in V}\in (X_\bC)^{V}: x_{\gamma v} = \chi(\gamma)\tau(\gamma)x_v,\,\,\forall\,\,\gamma\in\Gamma,\,\,\forall\,\,v\in V\}.\]

\item $\bC^{E}= \bigoplus_{\chi\in\hat{\Gamma}} Y_\chi$ where, 
\[Y_\chi = \{y=(y_e)_{e\in E}\in \bC^{E}: y_{\gamma e}=\chi(\gamma)y_e,\,\,\forall\,\,\gamma\in\Gamma,\,\,\forall\,\,e\in E\}.\]
\end{enumerate}
\end{proposition}

\proof
Each $x=(x_v)_{v\in V}\in (X_\bC)^{V}$ may be expressed 
as a sum $x=\sum_{\chi\in\hat{\Gamma}} x_\chi$ where $x_\chi=(x_{\chi,v})_{v\in V}\in (X_\bC)^V$ has components, 
\[x_{\chi,v} = \frac{1}{|\Gamma|}\left(\sum_{\gamma\in\Gamma} 
\overline{\chi(\gamma)}\tau(\gamma^{-1})x_{\gamma v}\right).\]
Similarly, each 
$y=(y_e)_{e\in E}\in \bC^{E}$ may be expressed as
a sum $\sum_{\chi\in\hat{\Gamma}} y_\chi$ where 
$y_\chi=(y_{\chi,e})_{e\in E}\in \bC^{E}$ has components,
\[y_{\chi,e} = \frac{1}{|\Gamma|}\left(\sum_{\gamma\in\Gamma} \overline{\chi(\gamma)}y_{\gamma e}\right).\]
We use here the standard identity,
\[\sum_{\chi\in\hat{\Gamma}} \chi(\gamma) = \left\{\begin{array}{ll}
|\Gamma| & \mbox{ if } \gamma=1,\\
0 & \mbox{ otherwise.}
\end{array}\right.\]
Note that $x_\chi\in X_\chi$ and $y_\chi\in Y_\chi$ for each $\chi\in\hat{\Gamma}$. 

Let $S\subseteq \hat{\Gamma}$ be a maximal subset which gives rise to a direct sum $\oplus_{\chi\in S} X_\chi$ and suppose there exists $\tilde{\chi}\in \hat{\Gamma}\backslash S$. If $x\in X_{\tilde{\chi}}\cap (\oplus_{\chi\in S} X_\chi)$ then
$x_v=\overline{\tilde{\chi}(\gamma)}\tau(\gamma^{-1})x_{\gamma v}$ for all $v\in V$ and all $\gamma\in \Gamma$. Moreover, $x=\sum_{\chi\in S} z_\chi$ for some unique $z_\chi\in X_\chi$.
It follows that, for all $v\in V$ and all $\gamma\in \Gamma$, 
\begin{eqnarray*}
x_v &=& \overline{\tilde{\chi}(\gamma)}\tau(\gamma^{-1})\left(\sum_{\chi\in S}  z_{\chi,\gamma v}\right)\\
&=&\overline{\tilde{\chi}(\gamma)}\tau(\gamma^{-1})\left(\sum_{\chi\in S}  \chi(\gamma)\tau(\gamma)z_{\chi,v}\right)\\
&=&\overline{\tilde{\chi}(\gamma)}\left(\sum_{\chi\in S}  \chi(\gamma)z_{\chi,v}\right).
\end{eqnarray*}
Thus $x = \overline{\tilde{\chi}(\gamma)}(\sum_{\chi\in S}  \chi(\gamma)z_{\chi})$ for all $\gamma\in \Gamma$.
Since the sum $x=\sum_{\chi\in S} z_\chi$ is direct, if $x\not=0$ then $\tilde{\chi} = \chi$ for some $\chi\in S$. This is a contradiction and so $X_{\tilde{\chi}}\cap (\oplus_{\chi\in S} X_\chi)=\{0\}$. However, this contradicts the maximality of $S$ and so it follows that $S=\hat{\Gamma}$.  This establishes the direct sum $\oplus_{\chi\in \hat{\Gamma}} X_\chi$ and a similar argument can be applied for $\oplus_{\chi\in \hat{\Gamma}} Y_\chi$.
\endproof

\subsection{Symmetric frameworks}
A {\em $\Gamma$-symmetric bar-joint framework} is a tuple $\G=(G,p,\theta,\tau)$ where $(G,p)$ is a bar-joint framework, $(G,\theta)$ is a $\Gamma$-symmetric graph and
$\tau:\Gamma\rightarrow \Isom(X)$ is a group representation which satisfies $\tau(\gamma) (p_v)=p_{\gamma v}$ for all $\gamma\in \Gamma$ and all  $v\in V$.

\begin{lemma}
\label{lem:SupportFunctionals}
Let $\G=(G,p,\theta,\tau)$ be a well-positioned and $\Gamma$-symmetric bar-joint framework in $X$ and let $vw\in E$.
Then $\varphi_{\gamma v,\gamma w}=\varphi_{v,w}\circ \tau(\gamma^{-1})$  for all $\gamma\in \Gamma$.
\end{lemma}

\proof
Let $p_0=p_v-p_w$. Then $\tau(\gamma)(p_0)=p_{\gamma v} - p_{\gamma w}$ and so for each $x\in X$,
\begin{eqnarray*}
\varphi_{\gamma v,\gamma w}(x)
&=& \|\tau(\gamma)(p_0)\|\,\left(\lim_{t\to 0}\,\frac{1}{t}\left(\|\tau(\gamma)(p_0+t\,\tau(\gamma^{-1})x)\|-\|\tau(\gamma)p_0\|\right)\right) \\
&=& \|p_0\|\,\left(\lim_{t\to 0}\,\frac{1}{t}\left(\|p_0+t\,\tau(\gamma^{-1})x\|-\|p_0\|\right) \right)= \varphi_{v,w}(\tau(\gamma^{-1})x).
\end{eqnarray*}
\endproof

In the following, the same symbol will be used to denote a real affine transformation $T:Y\to Z$ between two real linear spaces $Y$ and $Z$ and its complex extension $T:Y_\bC\to Z_\bC$. In particular, we consider the complex linear functionals $\varphi_{v,w}:X_\bC\to\bC$, the complex linear transformations $\tau(\gamma):X_\bC\to X_\bC$ and the complex differential $df_G(p):(X_\bC)^V \to \bC^E$ associated to a   bar-joint framework $(G,p)$ in $X$.

\begin{proposition} 
\label{prop:block}
Let $\G=(G,p,\theta,\tau)$ be a well-positioned and $\Gamma$-symmetric bar-joint framework in $X$.
With respect to the direct sum decompositions obtained in Proposition \ref{lem:decomp},
\[(X_\bC)^{V}=\bigoplus_{\chi\in\hat{\Gamma}} X_{\chi} \,\,\,\,  
\mbox{ and } \,\,\,\,\bC^{E}= \bigoplus_{\chi\in\hat{\Gamma}} Y_\chi,\] 
the (complex) differential $df_G(p)$ may be expressed as a direct sum of linear transformations,
\[df_G(p)=\bigoplus_{\chi\in\hat{\Gamma}} R_\chi(\G),\]
where $R_\chi(\G):X_{\chi}\to Y_{\chi}$ for each character $\chi\in\hat{\Gamma}$.
\end{proposition}

\proof
Let $vw\in E$. If $(x_v)_{v\in V}\in X_\chi$ then, using Lemma \ref{lem:SupportFunctionals}, 
\[\varphi_{v,w}(x_v-x_w) = \varphi_{\gamma v,\gamma w}(\tau(\gamma)(x_v-x_w))
=\varphi_{\gamma v,\gamma w}(\overline{\chi(\gamma)}(x_{\gamma v}-x_{\gamma w}))
=\overline{\chi(\gamma)}\varphi_{\gamma v,\gamma w}(x_{\gamma v}-x_{\gamma w}),\]
for each $\chi\in\hat{\Gamma}$ and each $\gamma\in \Gamma$. 
Thus, by Lemma \ref{lem:differential}, $df_G(p)(X_\chi)\subseteq Y_\chi$ and the result follows.
\endproof

\subsection{Infinitesimal rigid motions}
Denote by $\T(X;\bC)$ the complex vector space spanned by vector fields $\eta_\bC:X\to X_\bC$, $x\mapsto 1\otimes\eta(x)$ where $\eta\in\T(X)$. For convenience, $\eta_\bC$ will simply be denoted $\eta$.

\begin{proposition} 
\label{lem:decomp2}
Let $\tau:\Gamma\to\Isom(X)$ be a group representation.
Then,
$\T(X;\bC)=\bigoplus_{\chi\in\hat{\Gamma}} \T_{\chi}(X)$ where,
\[\T_\chi(X) = \{\eta\in \T(X;\bC): \eta(\tau(\gamma)x) = \chi(\gamma)\tau(\gamma)\eta(x),\,\,\forall\,\,\gamma\in\Gamma,\,\,\forall\,\,x\in X\}.\]
\end{proposition}

\proof
Applying an argument similar to Lemma \ref{lem:decomp}, each $\eta\in \T(X;\bC)$ may be expressed as a sum $\eta=\sum_{\chi\in\hat{\Gamma}} \eta_\chi$ where $\eta_\chi:X\to X_\bC$ is the vector field, 
\[\eta_{\chi}(x) = \frac{1}{|\Gamma|}\left(\sum_{\gamma\in\Gamma} 
\overline{\chi(\gamma)}\tau(\gamma^{-1})\eta(\tau(\gamma)x)\right).\]
Note that $\eta_\chi\in \T_{\chi}(X)$  for each $\chi\in\hat{\Gamma}$.

Let $S\subseteq\hat{\Gamma}$ be a maximal subset which gives rise to a direct sum $\oplus_{\chi\in S} \T_{\chi}(X)$ and suppose there exists $\tilde{\chi}\in\hat{\Gamma}\backslash S$. If $\eta\in \T_{\tilde{\chi}}(X)\cap(\oplus_{\chi\in S} \T_{\chi}(X))$ then
$\eta(x) = \overline{\tilde{\chi}(\gamma)}\tau(\gamma^{-1})\eta(\tau(\gamma)x)$ for all $x\in X$ and all $\gamma\in\Gamma$. Moreover, $\eta = \sum_{\chi\in S} \delta_\chi$ for some unique $\delta_\chi\in\T_\chi(X)$. Thus, for all $x\in X$ and all $\gamma\in\Gamma$,
\begin{eqnarray*}
\eta(x) &=& \overline{\tilde{\chi}(\gamma)}\tau(\gamma^{-1})\left(\sum_{\chi\in S}\delta_\chi(\tau(\gamma)x)\right) \\
&=& \overline{\tilde{\chi}(\gamma)}\tau(\gamma^{-1})\left(\sum_{\chi\in S}\chi(\gamma)\tau(\gamma)\delta_\chi(x)\right) \\
&=&  \overline{\tilde{\chi}(\gamma)}\left(\sum_{\chi\in S}\chi(\gamma)\delta_\chi(x) \right).
\end{eqnarray*}
Thus $\eta = \overline{\tilde{\chi}(\gamma)}\left(\sum_{\chi\in S}\chi(\gamma)\delta_\chi \right)$ for all $\gamma\in\Gamma$. Since the sum $\eta = \sum_{\chi\in S} \delta_\chi$ is direct, if $\eta\not=0$ then it follows that $\tilde{\chi}=\chi$ for some $\chi\in S$. This is a contradiction and so $\T_{\tilde{\chi}}(X)\cap(\oplus_{\chi\in S} \T_{\chi}(X))=\{0\}$.
However, this contradicts the maximality of $S$ and so $S=\hat{\Gamma}$.
\endproof

We now consider the complex restriction map $\rho_{(G,p)}: \T(X;\bC)\to (X_\bC)^V$, $\eta\mapsto (\eta(p_v))_{v\in V}$ for a given bar-joint framework $(G,p)$ in $X$.

\begin{proposition} 
\label{prop:block2}
Let $\G=(G,p,\theta,\tau)$ be a well-positioned and $\Gamma$-symmetric bar-joint framework in $X$.
With respect to the direct sum decompositions obtained in Propositions \ref{lem:decomp2} and \ref{lem:decomp},
\[ \T(X;\bC)=\bigoplus_{\chi\in\hat{\Gamma}} \T_{\chi}(X), \,\,\,\,  
\mbox{ and, }\,\,\,\,(X_\bC)^{V}=\bigoplus_{\chi\in\hat{\Gamma}} X_{\chi},\] 
the (complex) restriction map $\rho_{(G,p)}$ may be expressed as a direct sum of linear transformations,
\[\rho_{(G,p)}=\bigoplus_{\chi\in\hat{\Gamma}} \rho_\chi(\G),\]
where $\rho_\chi(\G):\T_\chi(X)\to X_{\chi}$ for each character $\chi\in\hat{\Gamma}$.
\end{proposition}

\proof
Let $\chi\in\hat{\Gamma}$. If $\eta\in \T_{\chi}(X)$ then, for each $\gamma\in \Gamma$ and all $v\in V$, 
\[\eta(p_{\gamma v}) = \eta(\tau(\gamma)p_v) 
=\chi(\gamma)\tau(\gamma)\eta(p_v).\]
Thus, $\rho_{(G,p)}(\T_{\chi}(X))\subseteq X_{\chi}$ and the result follows.
\endproof

\subsection{Decomposing the framework complex}
Denote by $\comp_\bC(G,p)$ the {\em complexified} framework complex for a bar-joint framework $(G,p)$,
\begin{eqnarray}
\label{chaincomplex2}
\begin{CD}
0 @>>> \T(X;\bC) @>\rho_{(G,p)}>> (X_\bC)^V @>df_G(p)>> \bC^E @>>> 0. \\
\end{CD}
\end{eqnarray}
If $\G=(G,p,\theta,\tau)$ is a well-positioned and $\Gamma$-symmetric bar-joint framework in $X$ then, recalling the decompositions $df_G(p)=\bigoplus_{\chi\in \hat{\Gamma}} R_\chi(\G)$ and $\rho_{(G,p)} = \bigoplus_{\chi\in \hat{\Gamma}} \rho_\chi(\G)$ from Propositions \ref{prop:block} and \ref{prop:block2}, we have $R_\chi(\G)\circ\rho_\chi(\G)=0$ for all $\chi\in \hat{\Gamma}$.
The {\em $\chi$-symmetric framework complex} for $\G$, denoted $\comp_\chi(\G)$, is the chain complex,
\begin{eqnarray}
\label{chaincomplex3}
\begin{CD}
0 @>>> \T_\chi(X) @>\rho_\chi(\G)>> X_\chi @>R_\chi(\G)>> Y_\chi @>>> 0. \\
\end{CD}
\end{eqnarray}
The kernel of $R_\chi(\G)$, denoted $\F_\chi(\G)$, is referred to as the space of {\em $\chi$-symmetric infinitesimal flexes} of $\G$.
The image of $\rho_\chi$, denoted $\T_\chi(\G)$, is referred to as the space of {\em trivial} $\chi$-symmetric infinitesimal flexes of $\G$.

\begin{theorem}
Let $\G=(G,p,\theta,\tau)$ be a well-positioned and $\Gamma$-symmetric bar-joint framework in $X$. Then,
\[\comp_\bC(G,p) = \bigoplus_{\chi\in\hat{\Gamma}} \,\comp_\chi(\G).\]
\end{theorem}

\proof
The result follows from Propositions \ref{prop:block}, \ref{lem:decomp2} and \ref{prop:block2}.
\endproof

\begin{definition}
A well-positioned and $\Gamma$-symmetric bar-joint framework $\G=(G,p,\theta,\tau)$ in a normed space $X$ is said to be,
\begin{enumerate}[(a)]
\item \emph{$\chi$-symmetrically full} if $\comp_\chi(\G)$ is exact at $\T_\chi(X)$.
\item {\em $\chi$-symmetrically infinitesimally rigid} if $\comp_\chi(\G)$ is exact at $X_\chi$.
\item \emph{$\chi$-symmetrically  independent} if $\comp_\chi(\G)$ is exact at $Y_\chi$.
\end{enumerate}
\end{definition}

A $\Gamma$-symmetric bar-joint framework is {\em $\chi$-symmetrically isostatic} if it is both $\chi$-symmetrically infinitesimally rigid and $\chi$-symmetrically independent.

\begin{corollary}
\label{lem:SymRigid}
Let $\G=(G,p,\theta,\tau)$ be a well-positioned and $\Gamma$-symmetric bar-joint framework in $X$.
The following statements are equivalent.
\begin{enumerate}[(i)]
\item $(G,p)$ is full (respectively, infinitesimally rigid, independent or isostatic).
\item $\G$ is $\chi$-symmetrically full (respectively, $\chi$-symmetrically infinitesimally rigid, $\chi$-symmetrically independent or $\chi$-symmetrically isostatic) for each $\chi\in \hat{\Gamma}$.
\end{enumerate}
\end{corollary}

Let $\G=(G,p,\theta,\tau)$ be a  $\Gamma$-symmetric bar-joint framework in $X$. A {\em $\Gamma$-symmetric subframework} of $\G$ is a $\Gamma$-symmetric framework $\H=(H,p_H,\theta_H,\tau_H)$ where $(H,p_H)$ is a subframework of $(G,p)$, $\theta_H:\Gamma\to \Aut(H)$ is the group homomorphism induced by $\theta$ and $\tau_H=\tau$.

\begin{lemma} \label{lem:blocksurjSym}
Let $\G=(G,p,\theta,\tau)$ be a well-positioned and $\Gamma$-symmetric bar-joint framework in $X$ and let $\chi\in\hat{\Gamma}$.
If $\G$ is $\chi$-symmetrically independent then every $\Gamma$-symmetric subframework of $\G$ is $\chi$-symmetrically independent.
\end{lemma}

\proof
Let $\H=(H,p_H,\theta_H,\tau_H)$ be a $\Gamma$-symmetric subframework of $\G$ and consider the direct sum decompositions,
\[X_\bC^{V(G)} = X_\bC^{V(H)}\oplus X_\bC^{V(G)\backslash V(H)}, \,\,\,
\mbox{ and, }\,\,\,
\bC^{E(G)} = \bC^{E(H)}\oplus \bC^{E(G)\backslash E(H)}\]
Note that $X_\chi  = X_\chi^{\H}\oplus X_\chi^{\G\backslash \H}$ where $X_\chi^{\H} = X_\bC^{V(H)}\cap X_\chi$ and $X_\chi^{\G\backslash\H} = X_\bC^{V(G)\backslash V(H)}\cap X_\chi$. Similarly, $Y_\chi = Y_\chi^{\H}\oplus Y_\chi^{\G\backslash\H}$ where $Y_\chi^{\H} = \bC^{E(H)}\cap Y_\chi$ and $Y_\chi^{\G\backslash\H} = \bC^{E(G)\backslash E(H)}\cap Y_\chi$.
With respect to these decompositions,
$R_\chi(\G)$ admits a block decomposition of the form,
\[R_\chi(\G) = \left(\begin{array}{cc}
R_\chi(\H) & 0 \\
C & D
\end{array}
\right).\]
Thus, if $R_\chi(\G)$ is surjective then so too is $R_\chi(\H)$.
\endproof

\subsection{Quotient graphs} \label{subsec:quotientgraphs}
Let $(G,\theta)$ be a $\Gamma$-symmetric graph and suppose $\theta$ acts freely on the vertices and edges of $G$.
The orbit of a vertex $v\in V$ (respectively an edge $e\in E$) is denoted by $[v]$ (respectively $[e]$).
Thus $[v]=\{\gamma v:\gamma\in \Gamma\}$ and 
$[e]=\{\gamma e:\gamma\in \Gamma\}$. The collection of all vertex orbits (respectively, edge orbits) is denoted $V_0$ (respectively, $E_0$).
The {\em quotient graph} $G_0=G/\Gamma$ is a multigraph with vertex set $V_0$, edge set $E_0$ and incidence relation satisfying $[e] = [v][w]$ if some (equivalently, every) edge in $[e]$ is incident with a vertex in $[v]$ and a vertex in $[w]$.

\begin{proposition}
\label{prop:counts}
Let $\G=(G,p,\theta,\tau)$ be a well-positioned and $\Gamma$-symmetric bar-joint framework in $X$. Let $\chi\in\hat{\Gamma}$ and suppose $\G$ is $\chi$-symmetrically full.  
\begin{enumerate}[(i)]
\item If $\G$ is $\chi$-symmetrically infinitesimally rigid then,
\[|E_0| \geq (\dim_\bR X)|V_0| - \dim_\bC \T_\chi(X).\]
\item If $\G$ is $\chi$-symmetrically independent then,
\[|E_0| \leq (\dim_\bR X)|V_0| - \dim_\bC \T_\chi(X).\]
\item
If $\G$ is $\chi$-symmetrically isostatic then, 
\[|E_0| = (\dim_\bR X)|V_0| - \dim_\bC \T_\chi(X).\]
\end{enumerate}
\end{proposition}

\proof
Applying Proposition \ref{prop:block},
$$
|E_0| = \dim_\bC Y_\chi \geq \rank R_\chi  = \dim_\bC X_\chi - \dim_\bC \ker R_\chi 
= (\dim_\bC X_\bC)|V_0| - \dim_\bC \F_\chi(\G).
$$
If $(i)$ holds then $\F_\chi(\G) = \T_\chi(\G)$ and $\dim_\bC\T_\chi(\G) = \dim_\bC\T_\chi(X)$. If $(ii)$ holds then $\dim_\bC Y_\chi = \rank R_\chi$ and $\dim_\bC \F_\chi(\G)\geq \dim_\bC \T_\chi(\G)= \dim_\bC\T_\chi(X)$.
If $(iii)$ holds then the result follows from $(i)$ and $(ii)$.
\endproof

\subsection{Norms with a minimal space of infinitesimal rigid motions}
The space $\T(X)$ of infinitesimal rigid motions of a normed space $X$ is {\em minimal} if given any $\eta\in\T(X)$ there exists $x_0\in X$ such that $\eta(x)=x_0$ for all $x\in X$.
This class includes all $\ell^p$-spaces, with $p\not=2$, and all normed spaces with a polyhedral unit ball (see \cite[Lemma 2.5]{kit-pow}). If $\dim_\bR X=2$, then this class includes all norms not derived from an inner product. 
In the following, the identity map on $X$ is denoted $I_X$, or simply $I$.

\begin{lemma}
\label{lem:dimension}
Let $\tau:\Gamma\to\Isom(X)$ be a group representation and let $\chi\in\hat{\Gamma}$. If $\T(X)$ is minimal then,
\[\dim_\bC \T_\chi(X)=\dim_\bC \left(\bigcap_{\gamma\in\Gamma} \ker(\tau(\gamma)-\overline{\chi(\gamma)}I)\right).\] 
\end{lemma}

\proof
Let $\eta\in \T(X)$. Since $\T(X)$ is minimal, there exists $x_0\in X$ such that $\eta(x)=x_0$ for all $x\in X$.
Note in particular that $x_0 = \eta(\tau(\gamma)(x_0))$ for each $\gamma\in \Gamma$. Thus $\eta_\bC \in \T_\chi(X)$ if and only if $x_0 = \chi(\gamma)\tau(\gamma)(x_0)$ for each $\gamma\in \Gamma$.
The result now follows.
\endproof

Let $\omega = e^{2\pi i/n}$, where $n\in\bN$ and $n\geq 2$, and consider the multiplicative cyclic group $\bZ_n=\{\omega^k:k=0,1,\ldots,n-1\}$. 
Recall that the dual group for $\bZ_n$ consists of characters 
$\chi_0,\chi_1,\ldots,\chi_{n-1}$ where $\chi_j(\omega)=\omega^{j}$ for each $j=0,1,\ldots,n-1$. 

\begin{lemma}
\label{lem:dimensionCyclic}
Let $\tau:\bZ_n\to\Isom(X)$ be a group representation  where $n\geq 2$. If $\T(X)$ is minimal then, for each $j=0,1,\ldots,n-1$,
\[\dim_\bC \T_{\chi_j}(X) = \dim_\bC \ker(\tau(\omega)-\overline{\omega}^jI).\]
\end{lemma}

\proof
Let $j\in\{0,1,\ldots,n-1\}$. 
Note that $\ker(\tau(\omega)-\overline{\omega}^j I)\subseteq \ker(\tau(\omega^k)-\overline{\omega}^{jk}I)$ for $k=0,1,\ldots,n-1$.  Thus, by Lemma \ref{lem:dimension},  
\[\dim_\bC \T_{\chi_j}(X)
= \dim_\bC \left(\bigcap_{k=0}^{n-1} \ker(\tau(\omega^k)-\overline{\omega}^{jk} I)\right)
= \dim_\bC \ker(\tau(\omega)-\overline{\omega}^jI).\]
\endproof

In the following, an {\em $n$-fold rotation} ($n \geq 2$) of a two-dimensional real vector space $X$ is a linear operator $S:X\to X$ with matrix  $\tiny{\left(\begin{array}{cc} \cos(2\pi /n) &  -\sin(2\pi /n) \\  \sin(2\pi /n) &  \cos(2\pi /n)\end{array}\right)}$
with respect to some basis for $X$.
If $\dim X\geq 3$ then a linear operator $T:X\to X$ is an {\em $n$-fold rotation} if there exists a direct sum decomposition $X=Y\oplus Z$, where $Y$ is a two-dimensional subspace of $X$, with respect to which $T=S\oplus I_Z$, $S$ is an $n$-fold rotation of $Y$ and $I_Z$ is the identity operator on $Z$.

\begin{lemma}
\label{lem:dimensionCn}
Let $\tau:\bZ_n\to\Isom(X)$ be a group representation where $\tau(\omega)$ is an $n$-fold rotation of $X$  and $n\geq 2$.
Suppose, in addition, that $\T(X)$ is minimal. 

\begin{enumerate}[(i)]
\item
If $n=2$ then,
\[\dim_\bC \T_{\chi_j}(X)=
\left\{\begin{array}{ll}
\dim_\bR X-2 & \mbox{ if } j=0, \\
2 & \mbox{ if } j=1.
\end{array}\right.\]
\item 
If $n\geq 3$ then,
\[\dim_\bC \T_{\chi_j}(X)=
\left\{\begin{array}{ll}
\dim_\bR X-2 & \mbox{ if } j=0, \\
1 & \mbox{ if } j\in\{1,n-1\},  \\
0 & \mbox{ otherwise.}
\end{array}\right.\]
\end{enumerate}
\end{lemma}

\proof
Write $X=Y\oplus Z$ and $\tau(\omega)=S\oplus I_Z$ where $\dim Y=2$ and $S$ is an $n$-fold rotation of $Y$.
Then $X_\bC = Y_\bC\oplus Z_\bC$.
Note that $\tau(\omega)-\overline{\chi_0(\omega)}I
= (S-I_Y)\oplus 0$.
Also, $S-I_Y$ is invertible and so, by Lemma \ref{lem:dimensionCyclic},  
\[\dim_\bC \T_{\chi_0}(X)=
\dim_\bC \ker((S-I_Y)\oplus 0)=\dim_\bR X-2.\]
Now let $j\in\{1,\ldots,n-1\}$. 
If $n=2$, then $\omega =-1$ and $S=-I_Y$. Note that $\tau(\omega)-\overline{\omega} I
=0\oplus 2I_Z$ and so, by Lemma \ref{lem:dimensionCyclic},
\[\dim_\bC \T_{\chi_1}(X)=
\dim_\bC \ker(0\oplus 2I_Z) = \dim_\bR Y = 2.\]
If $n\geq 3$, then $S$ has eigenvalues of multiplicity $1$ at $\omega$ and $\overline{\omega}$. Note that $\tau(\omega)-\overline{\omega}^j I
=(S-\overline{\omega}^jI_Y)\oplus (1-\overline{\omega}^j)I_Z$ and so, by Lemma \ref{lem:dimensionCyclic},
\[\dim_\bC \T_{\chi_j}(X) = \dim_\bC \ker(\tau(\omega)-\overline{\omega}^jI) 
= \dim_\bC \ker(S-\overline{\omega}^jI_Y)
=\left\{\begin{array}{ll}
1 & \mbox{ if } j\in\{1,n-1\},  \\
0 & \mbox{ otherwise.}
\end{array}\right.\]
\endproof

\subsection{Gain graphs} \label{subsec:gaingraphs}
Let $(G,\theta)$ be a  $\Gamma$-symmetric graph and fix an orientation on the edges of the quotient graph $G_0 = (V_0,E_0)$.
For each vertex orbit $[v]\in V_0$, choose a representative vertex $\tilde{v}\in[v]$ and denote the set of all such representatives by $\tilde{V}_0$.
For each directed edge $[e]=([v],[w])$  in the directed multigraph $G_0$ there exists a unique $\gamma\in\Gamma$, referred to as the {\em gain} on $[e]$, such that $\tilde{v}(\gamma\tilde{w})\in [e]$.
This gain assignment $\psi:E_0\to \Gamma$, $[e]\mapsto \psi_{[e]}$, is well-defined and the pair $(G_0,\psi)$ is referred to as a {\em (quotient) gain graph} for $(G,\theta)$.
The graph $G$ is also called the \emph{covering graph} of $(G_0,\psi)$. 

Note that a gain assignment $\psi$ is dependent on the choice of representative vertices $\tilde{V}_0$ and also on the choice of orientation for each edge of $G_0$. We may {\em switch} the gain assignment on the directed multigraph $G_0$ by choosing a different set of vertex orbit representatives. We regard two gain assignments on the directed multigraph $G_0$ as equivalent if one can be obtained from the other by such a switching operation. Note that if the orientation of an edge $[e]$ in $G_0$ is reversed then the induced gain $\psi_{[e]}$ is replaced with $\psi_{[e]}^{-1}$.

In general, we refer to a group-labelled directed multigraph $(G_0,\psi)$ with $\psi:E_0\to \Gamma$ as a \emph{$\Gamma$-gain graph} if it is a quotient gain graph for a $\Gamma$-symmetric graph $(G,\theta)$. Note that, since $G$ is assumed to be simple, $(G_0,\psi)$ has no parallel edges with the same gain when oriented in the same direction and no loops with a trivial gain. We regard two $\Gamma$-gain graphs as equivalent if they are derived from the same $\Gamma$-symmetric graph $(G,\theta)$. 
 For more on  gain graphs we refer the reader to \cite{jkt,zas}. 

\begin{example}
Figure \ref{fig:gain_cov_graphs} illustrates several examples of $\bZ_2$-symmetric graphs together with accompanying quotient gain graphs. These gain graphs will form base graphs for the inductive construction presented in Section \ref{Sect:Inductive}. Note that in the case of $\bZ_2$-symmetric graphs, gain assignments are independent of the chosen edge orientation. Thus edge orientations have been omitted from Figure \ref{fig:gain_cov_graphs}. In each gain graph, the indicated gains are determined by the set of representative vertices, labelled  by $p$, in its covering graph. Note that each covering graph is presented as a two-dimensional bar-joint framework with half-turn rotational symmetry. Moreover, it can be shown that these bar-joint frameworks are $\chi_0$-symmetrically isostatic with respect to the $\ell^\infty$ norm. The reasons for this, and the significance of the edge colourings, are explained in Section \ref{Sect:FrameworkColours}.
\end{example}

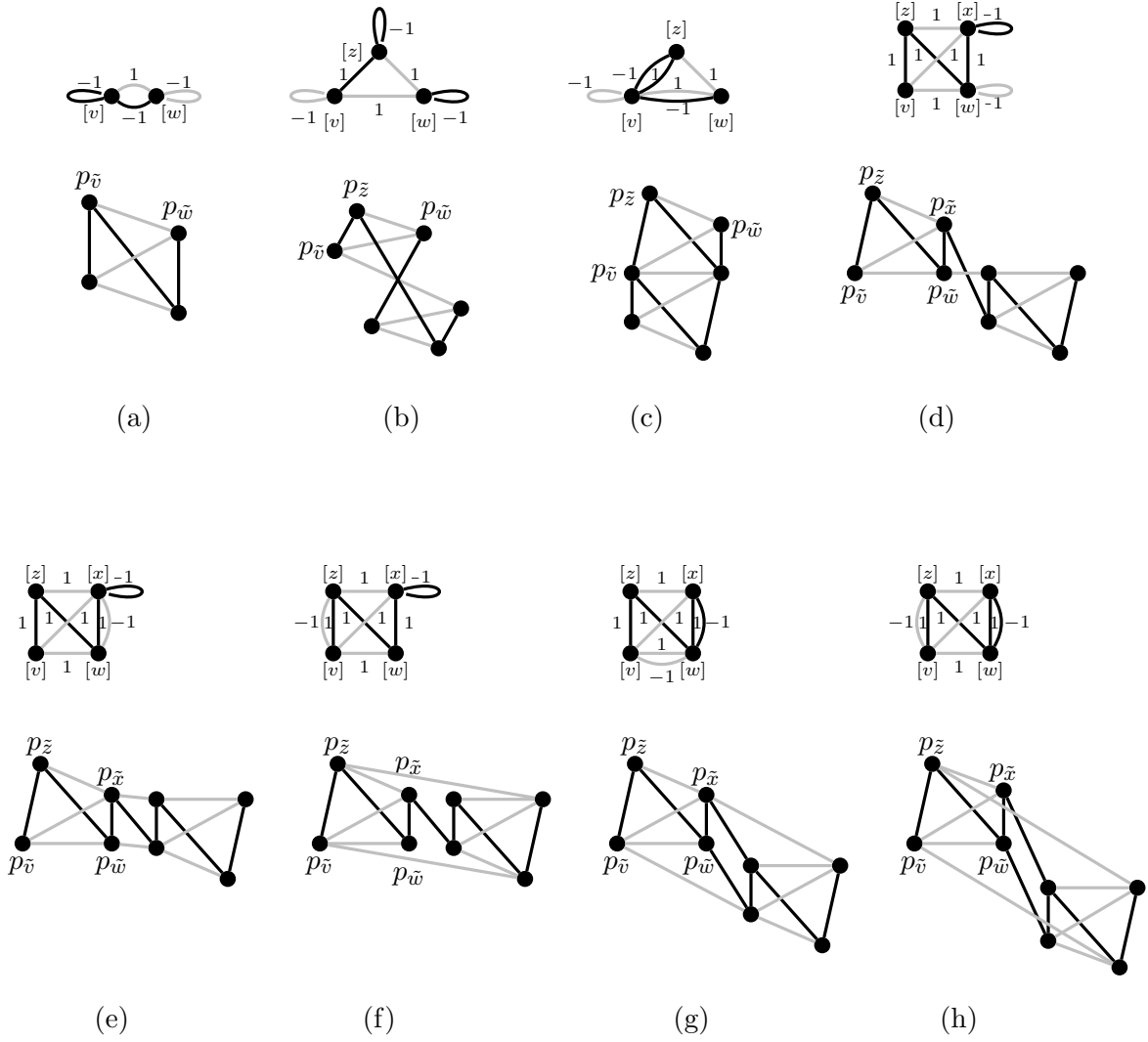
\begin{figure}[htp]
\begin{center}
\begin{tikzpicture}[very thick,scale=0.6]
\tikzstyle{every node}=[circle, draw=black, fill=black, inner sep=0pt, minimum width=5pt];

\node (p1) at (-1,0) {};
\node (p2) at (0,0) {};

\node [draw=white, fill=white] (a) at (-1.4,-0.4)  {\tiny $[v]$};
\node [draw=white, fill=white] (a) at (0.4,-0.4)  {\tiny $[w]$};

\node [draw=white, fill=white] (a) at (-0.5,0.5)  {\tiny $1$};
\node [draw=white, fill=white] (a) at (-0.5,-0.5)  {\tiny $-1$};
\node [draw=white, fill=white] (a) at (-1.5,0.3)  {\tiny $-1$};
\node [draw=white, fill=white] (a) at (0.5,0.3)  {\tiny $-1$};

\path(p1) edge [bend right=40] (p2);
\path(p1) edge [bend left=40, lightgray] (p2);
\path(p1) edge [loop left, >=stealth,shorten >=1pt,looseness=30] (p1);
\path(p2) edge [lightgray,loop right, >=stealth,shorten >=1pt,looseness=30] (p2);

\node [draw=white, fill=white] (a) at (-1.5,-1.9)  { $p_{\tv}$};
\node [draw=white, fill=white] (a) at (0.5,-2.6)  { $p_{\tw}$};

\node (p1) at (-1.5,-4.2) {};
\node (p2) at (0.5,-4.9) {};
\node (p3) at (0.5,-3.1) {};
\node (p4) at (-1.5,-2.4) {};

\draw[lightgray](p1)--(p2);
\draw(p3)--(p2);
\draw[lightgray](p3)--(p4);
\draw(p1)--(p4);
\draw(p4)--(p2);
\draw[lightgray](p1)--(p3);

\node [draw=white, fill=white] (a) at (-0.5,-7.3)  {(a)};


\node (pa) at (4,0) {};
\node (pb) at (6,0) {};
\node (pc) at (5,1) {};

\node [draw=white, fill=white] (a) at (4,-0.6)  {\tiny $[v]$};
\node [draw=white, fill=white] (a) at (6,-0.6)  {\tiny $[w]$};
\node [draw=white, fill=white] (a) at (4.4,1)  {\tiny $[z]$};

\node [draw=white, fill=white] (a) at (5,-0.3)  {\tiny $1$};
\node [draw=white, fill=white] (a) at (4.2,0.5)  {\tiny $1$};
\node [draw=white, fill=white] (a) at (5.8,0.5)  {\tiny $1$};
\node [draw=white, fill=white] (a) at (6.7,-0.5)  {\tiny $-1$};
\node [draw=white, fill=white] (a) at (3.3,-0.5)  {\tiny $-1$};
\node [draw=white, fill=white] (a) at (5.5,1.5)  {\tiny $-1$};

\path(pa) edge [lightgray,loop left, >=stealth,shorten >=1pt,looseness=30] (pa);
\path(pb) edge [loop right, >=stealth,shorten >=1pt,looseness=30] (pb);
\path(pc) edge [loop above, >=stealth,shorten >=1pt,looseness=30] (pc);
\draw[lightgray](pa)--(pb);
\draw[lightgray](pc)--(pb);
\draw(pa)--(pc);

\node [draw=white, fill=white, yshift=-3cm] (a) at (3.5,1.5)  { $p_{\tv}$};
\node [draw=white, fill=white,yshift=-3cm] (a) at (6.3,2.4)  { $p_{\tw}$};
\node [draw=white, fill=white,yshift=-3cm] (a) at (4.5,2.9)  { $p_{\tz}$};

\node[yshift=-3cm] (pa) at (4,1.5) {};
\node[yshift=-3cm] (pb) at (6,1.9) {};
\node[yshift=-3cm] (pc) at (4.5,2.4) {};
\node[xshift=0.5cm,yshift=-3cm] (paa) at (6,0.2) {};
\node[xshift=0.5cm,yshift=-3cm] (pbb) at (4,-0.2) {};
\node[xshift=0.5cm,yshift=-3cm] (pcc) at (5.5,-0.7) {};

\draw[lightgray](pa)--(pb);
\draw[lightgray](pb)--(pc);
\draw(pc)--(pa);
\draw[lightgray](paa)--(pbb);
\draw[lightgray](pbb)--(pcc);
\draw(paa)--(pcc);
\draw[lightgray](pa)--(paa);
\draw(pb)--(pbb);
\draw(pc)--(pcc);

\node [draw=white, fill=white] (a) at (5.5,-7.3)  {(b)};


\node [draw=white, fill=white] (a) at (11.7,0.3)  {\tiny $1$};
\node [draw=white, fill=white] (a) at (11.7,-0.3)  {\tiny $-1$};
\node [draw=white, fill=white] (a) at (10.5,0.5)  {\tiny $-1$};
\node [draw=white, fill=white] (a) at (11.2,0.5)  {\tiny $1$};
\node [draw=white, fill=white] (a) at (12.5,0.5)  {\tiny $1$};
\node [draw=white, fill=white] (a) at (9.5,0.3)  {\tiny $-1$};

\node [draw=white, fill=white, xshift=4cm] (a) at (4,-0.6)  {\tiny $[v]$};
\node [draw=white, fill=white, xshift=4cm] (a) at (6,-0.6)  {\tiny $[w]$};
\node [draw=white, fill=white, xshift=4cm] (a) at (5,1.5)  {\tiny $[z]$};

\node[xshift=4cm] (qa) at (4,0) {};
\node[xshift=4cm] (qb) at (6,0) {};
\node[xshift=4cm] (qc) at (5,1) {};
\path(qa) edge [bend left=10, lightgray] (qb);
\path(qa) edge [bend right=10] (qb);
\draw[lightgray](qb)-- (qc);
\path(qc) edge [bend left=20] (qa);
\path(qc) edge [bend right=20] (qa);
\path(qa) edge [lightgray,loop left, >=stealth,shorten >=1pt,looseness=30] (qa);

\node [draw=white, fill=white,xshift=7cm,yshift=-2.4cm] (a) at (-1.6,0)  { $p_{\tv}$};
\node [draw=white, fill=white,xshift=7cm,yshift=-2.4cm] (a) at (1.6,1)  { $p_{\tw}$};
\node [draw=white, fill=white,xshift=7cm,yshift=-2.4cm] (a) at (-1.2,1.7)  { $p_{\tz}$};

\node[xshift=7cm,yshift=-2.4cm] (qaa) at (-1,0) {};
\node[xshift=7cm,yshift=-2.4cm] (qbb) at (1,0) {};
\node[xshift=7cm,yshift=-2.4cm] (qcc) at (-0.6,1.8) {};
\node[xshift=7cm,yshift=-2.4cm] (qdd) at (1,1.1) {};
\node[xshift=7cm,yshift=-2.4cm] (qee) at (0.6,-1.8) {};
\node[xshift=7cm,yshift=-2.4cm] (qff) at (-1,-1.1) {};

\draw[lightgray](qaa)--(qbb);
\draw(qcc)--(qbb);
\draw[lightgray](qcc)--(qdd);
\draw[lightgray](qaa)--(qdd);
\draw(qaa)--(qcc);
\draw(qbb)--(qdd);
\draw(qaa)--(qee);
\draw[lightgray](qbb)--(qff);
\draw[lightgray](qee)--(qff);
\draw(qaa)--(qff);
\draw(qbb)--(qee);

\node [draw=white, fill=white] (a) at (11,-7.3)  {(c)};


\node [draw=white, fill=white,xshift=10.5cm,yshift=0.5cm] (a) at (0,-1)  {\tiny $1$};
\node [draw=white, fill=white,xshift=10.5cm,yshift=0.5cm] (a) at (1,0)  {\tiny $1$};
\node [draw=white, fill=white,xshift=10.5cm,yshift=0.5cm] (a) at (-1,0)  {\tiny $1$};
\node [draw=white, fill=white,xshift=10.5cm,yshift=0.5cm] (a) at (0,1)  {\tiny $1$};
\node [draw=white, fill=white,xshift=10.5cm,yshift=0.5cm] (a) at (0.4,0.1)  {\tiny $1$};
\node [draw=white, fill=white,xshift=10.5cm,yshift=0.5cm] (a) at (-0.4,0.1)  {\tiny $1$};
\node [draw=white, fill=white,xshift=10.5cm,yshift=0.5cm] (a) at (1.2,1)  {\tiny $-1$};
\node [draw=white, fill=white,xshift=10.5cm,yshift=0.5cm] (a) at (1.2,-1)  {\tiny $-1$};

\node [draw=white, fill=white, xshift=10.5cm,yshift=0.5cm] (a) at (-0.7,-1.1)  {\tiny $[v]$};
\node [draw=white, fill=white, xshift=10.5cm,yshift=0.5cm] (a) at (0.7,-1.1)  {\tiny $[w]$};
\node [draw=white, fill=white, xshift=10.5cm,yshift=0.5cm] (a) at (0.7,1.1)  {\tiny $[x]$};
\node [draw=white, fill=white, xshift=10.5cm,yshift=0.5cm] (a) at (-0.7,1.1)  {\tiny $[z]$};

\node[xshift=10.5cm,yshift=0.5cm] (pp1) at (-0.7,-0.7) {};
\node[xshift=10.5cm,yshift=0.5cm] (pp2) at (0.7,-0.7) {};
\node[xshift=10.5cm,yshift=0.5cm] (pp3) at (0.7,0.7) {};
\node[xshift=10.5cm,yshift=0.5cm] (pp4) at (-0.7,0.7) {};

\draw[lightgray](pp1)--(pp2);
\draw(pp3)--(pp2);
\draw[lightgray](pp3)--(pp4);
\draw(pp1)--(pp4);
\draw(pp4)--(pp2);
\draw[lightgray](pp1)--(pp3);
\path(pp2) edge [lightgray,loop right, >=stealth,shorten >=1pt,looseness=30] (pp2);
\path(pp3) edge [loop right, >=stealth,shorten >=1pt,looseness=30] (pp3);

\node [draw=white, fill=white,xshift=10cm,yshift=-2.4cm] (a) at (-1,-0.5)  { $p_{\tv}$};
\node [draw=white, fill=white,xshift=10cm,yshift=-2.4cm] (a) at (1,-0.5)  { $p_{\tw}$};
\node [draw=white, fill=white,xshift=10cm,yshift=-2.4cm] (a) at (-0.6,2.2)  { $p_{\tz}$};
\node [draw=white, fill=white,xshift=10cm,yshift=-2.4cm] (a) at (1,1.5)  { $p_{\tx}$};

\node[xshift=10cm,yshift=-2.4cm] (qaa) at (-1,0) {};
\node[xshift=10cm,yshift=-2.4cm] (qbb) at (1,0) {};
\node[xshift=10cm,yshift=-2.4cm] (qcc) at (-0.6,1.8) {};
\node[xshift=10cm,yshift=-2.4cm] (qdd) at (1,1.1) {};
\draw[lightgray](qaa)--(qbb);
\draw(qcc)--(qbb);
\draw[lightgray](qcc)--(qdd);
\draw[lightgray](qaa)--(qdd);
\draw(qaa)--(qcc);
\draw(qbb)--(qdd);

\node[xshift=10cm,yshift=-2.4cm] (qaaa) at (4,0) {};
\node[xshift=10cm,yshift=-2.4cm] (qbbb) at (2,0) {};
\node[xshift=10cm,yshift=-2.4cm] (qccc) at (3.6,-1.8) {};
\node[xshift=10cm,yshift=-2.4cm] (qddd) at (2,-1.1) {};
\draw[lightgray](qaaa)--(qbbb);
\draw(qccc)--(qbbb);
\draw[lightgray](qccc)--(qddd);
\draw[lightgray](qaaa)--(qddd);
\draw(qaaa)--(qccc);
\draw(qbbb)--(qddd);

\draw(qdd)--(qddd);
\draw[lightgray](qbb)--(qbbb);

\node [draw=white, fill=white] (a) at (17.5,-7.3)  {(d)};

\end{tikzpicture}
\vspace{1.5cm}

\begin{tikzpicture}[very thick,scale=0.6]
\tikzstyle{every node}=[circle, draw=black, fill=black, inner sep=0pt, minimum width=5pt];

\node [draw=white, fill=white] (a) at (0,-1)  {\tiny $1$};
\node [draw=white, fill=white] (a) at (0.8,0)  {\tiny $1$};
\node [draw=white, fill=white] (a) at (-1,0)  {\tiny $1$};
\node [draw=white, fill=white] (a) at (0,1)  {\tiny $1$};
\node [draw=white, fill=white] (a) at (0.4,0.1)  {\tiny $1$};
\node [draw=white, fill=white] (a) at (-0.4,0.1)  {\tiny $1$};
\node [draw=white, fill=white] (a) at (1.25,0)  {\tiny $-1$};
\node [draw=white, fill=white] (a) at (1.2,1)  {\tiny $-1$};

\node [draw=white, fill=white] (a) at (-0.7,-1.1)  {\tiny $[v]$};
\node [draw=white, fill=white] (a) at (0.7,-1.1)  {\tiny $[w]$};
\node [draw=white, fill=white] (a) at (0.7,1.1)  {\tiny $[x]$};
\node [draw=white, fill=white] (a) at (-0.7,1.1)  {\tiny $[z]$};

\node (p1) at (-0.7,-0.7) {};
\node (p2) at (0.7,-0.7) {};
\node (p3) at (0.7,0.7) {};
\node (p4) at (-0.7,0.7) {};

\draw[lightgray](p1)--(p2);
\draw(p3)--(p2);
\draw[lightgray](p3)--(p4);
\draw(p1)--(p4);
\draw(p4)--(p2);
\draw[lightgray](p1)--(p3);
\path(p3) edge [loop right, >=stealth,shorten >=1pt,looseness=30] (p3);
\path(p2) edge [bend right=30,lightgray] (p3);

\node [draw=white, fill=white,yshift=-3cm] (a) at (-1,-0.5)  { $p_{\tv}$};
\node [draw=white, fill=white,yshift=-3cm] (a) at (1,-0.5)  { $p_{\tw}$};
\node [draw=white, fill=white,yshift=-3cm] (a) at (-0.6,2.2)  { $p_{\tz}$};
\node [draw=white, fill=white,yshift=-3cm] (a) at (1,1.5)  { $p_{\tx}$};

\node[yshift=-3cm] (qaa) at (-1,0) {};
\node[yshift=-3cm] (qbb) at (1,0) {};
\node[yshift=-3cm] (qcc) at (-0.6,1.8) {};
\node[yshift=-3cm] (qdd) at (1,1.1) {};
\draw[lightgray](qaa)--(qbb);
\draw(qcc)--(qbb);
\draw[lightgray](qcc)--(qdd);
\draw[lightgray](qaa)--(qdd);
\draw(qaa)--(qcc);
\draw(qbb)--(qdd);

\node[yshift=-3cm] (qaaa) at (4,1) {};
\node[yshift=-3cm] (qbbb) at (2,1) {};
\node[yshift=-3cm] (qccc) at (3.6,-0.8) {};
\node[yshift=-3cm] (qddd) at (2,-0.1) {};
\draw[lightgray](qaaa)--(qbbb);
\draw(qccc)--(qbbb);
\draw[lightgray](qccc)--(qddd);
\draw[lightgray](qaaa)--(qddd);
\draw(qaaa)--(qccc);
\draw(qbbb)--(qddd);

\draw(qdd)--(qddd);
\draw[lightgray](qbb)--(qddd);
\draw[lightgray](qdd)--(qbbb);

\node [draw=white, fill=white] (a) at (1,-9)  {(e)};


\node [draw=white, fill=white, xshift=4cm] (a) at (0,-1)  {\tiny $1$};
\node [draw=white, fill=white, xshift=4cm] (a) at (1,0)  {\tiny $1$};
\node [draw=white, fill=white, xshift=4cm] (a) at (-0.8,0)  {\tiny $1$};
\node [draw=white, fill=white, xshift=4cm] (a) at (0,1)  {\tiny $1$};
\node [draw=white, fill=white, xshift=4cm] (a) at (0.4,0.1)  {\tiny $1$};
\node [draw=white, fill=white, xshift=4cm] (a) at (-0.4,0.1)  {\tiny $1$};
\node [draw=white, fill=white, xshift=4cm] (a) at (-1.3,0)  {\tiny $-1$};
\node [draw=white, fill=white, xshift=4cm] (a) at (1.2,1)  {\tiny $-1$};

\node [draw=white, fill=white,xshift=4cm] (a) at (-0.7,-1.1)  {\tiny $[v]$};
\node [draw=white, fill=white,xshift=4cm] (a) at (0.7,-1.1)  {\tiny $[w]$};
\node [draw=white, fill=white,xshift=4cm] (a) at (0.7,1.1)  {\tiny $[x]$};
\node [draw=white, fill=white,xshift=4cm] (a) at (-0.7,1.1)  {\tiny $[z]$};

\node[xshift=4cm] (p1) at (-0.7,-0.7) {};
\node[xshift=4cm] (p2) at (0.7,-0.7) {};
\node[xshift=4cm] (p3) at (0.7,0.7) {};
\node[xshift=4cm] (p4) at (-0.7,0.7) {};

\draw[lightgray](p1)--(p2);
\draw(p3)--(p2);
\draw[lightgray](p3)--(p4);
\draw(p1)--(p4);
\draw(p4)--(p2);
\draw[lightgray](p1)--(p3);
\path(p3) edge [loop right, >=stealth,shorten >=1pt,looseness=30] (p3);
\path(p1) edge [bend left=30,lightgray] (p4);

\node [draw=white, fill=white,xshift=4cm,yshift=-3cm] (a) at (-1,-0.5)  { $p_{\tv}$};
\node [draw=white, fill=white,xshift=4cm,yshift=-3cm] (a) at (1,-0.75)  { $p_{\tw}$};
\node [draw=white, fill=white,xshift=4cm,yshift=-3cm] (a) at (-0.6,2.2)  { $p_{\tz}$};
\node [draw=white, fill=white,xshift=4cm,yshift=-3cm] (a) at (1,1.75)  { $p_{\tx}$};

\node[xshift=4cm,yshift=-3cm] (qaa) at (-1,0) {};
\node[xshift=4cm,yshift=-3cm] (qbb) at (1,0) {};
\node[xshift=4cm,yshift=-3cm] (qcc) at (-0.6,1.8) {};
\node[xshift=4cm,yshift=-3cm] (qdd) at (1,1.1) {};
\draw[lightgray](qaa)--(qbb);
\draw(qcc)--(qbb);
\draw[lightgray](qcc)--(qdd);
\draw[lightgray](qaa)--(qdd);
\draw(qaa)--(qcc);
\draw(qbb)--(qdd);

\node[xshift=4cm,yshift=-3cm] (qaaa) at (4,1) {};
\node[xshift=4cm,yshift=-3cm] (qbbb) at (2,1) {};
\node[xshift=4cm,yshift=-3cm] (qccc) at (3.6,-0.8) {};
\node[xshift=4cm,yshift=-3cm] (qddd) at (2,-0.1) {};
\draw[lightgray](qaaa)--(qbbb);
\draw(qccc)--(qbbb);
\draw[lightgray](qccc)--(qddd);
\draw[lightgray](qaaa)--(qddd);
\draw(qaaa)--(qccc);
\draw(qbbb)--(qddd);

\draw(qdd)--(qddd);
\draw[lightgray](qaa)--(qccc);
\draw[lightgray](qcc)--(qaaa);

\node [draw=white, fill=white] (a) at (7,-9)  {(f)};


\node [draw=white, fill=white, xshift=8cm] (a) at (0,-0.5)  {\tiny $1$};
\node [draw=white, fill=white, xshift=8cm] (a) at (0.8,0)  {\tiny $1$};
\node [draw=white, fill=white, xshift=8cm] (a) at (-1,0)  {\tiny $1$};
\node [draw=white, fill=white, xshift=8cm] (a) at (0,1)  {\tiny $1$};
\node [draw=white, fill=white, xshift=8cm] (a) at (0.4,0.1)  {\tiny $1$};
\node [draw=white, fill=white, xshift=8cm] (a) at (-0.4,0.1)  {\tiny $1$};
\node [draw=white, fill=white, xshift=8cm] (a) at (1.25,0)  {\tiny $-1$};
\node [draw=white, fill=white, xshift=8cm] (a) at (0,-1.25)  {\tiny $-1$};

\node [draw=white, fill=white,xshift=8cm] (a) at (-0.7,-1.1)  {\tiny $[v]$};
\node [draw=white, fill=white,xshift=8cm] (a) at (0.7,-1.1)  {\tiny $[w]$};
\node [draw=white, fill=white,xshift=8cm] (a) at (0.7,1.1)  {\tiny $[x]$};
\node [draw=white, fill=white,xshift=8cm] (a) at (-0.7,1.1)  {\tiny $[z]$};

\node[xshift=8cm] (p1) at (-0.7,-0.7) {};
\node[xshift=8cm] (p2) at (0.7,-0.7) {};
\node[xshift=8cm] (p3) at (0.7,0.7) {};
\node[xshift=8cm] (p4) at (-0.7,0.7) {};

\draw[lightgray](p1)--(p2);
\draw(p3)--(p2);
\draw[lightgray](p3)--(p4);
\draw(p1)--(p4);
\draw(p4)--(p2);
\draw[lightgray](p1)--(p3);
\path(p1) edge [bend right=30,lightgray] (p2);
\path(p2) edge [bend right=30] (p3);

\node [draw=white, fill=white,xshift=8cm,yshift=-3cm] (a) at (-1,-0.5)  { $p_{\tv}$};
\node [draw=white, fill=white,xshift=8cm,yshift=-3cm] (a) at (0.85,-0.5)  { $p_{\tw}$};
\node [draw=white, fill=white,xshift=8cm,yshift=-3cm] (a) at (-0.6,2.2)  { $p_{\tz}$};
\node [draw=white, fill=white,xshift=8cm,yshift=-3cm] (a) at (1,1.5)  { $p_{\tx}$};

\node[xshift=8cm,yshift=-3cm] (qaa) at (-1,0) {};
\node[xshift=8cm,yshift=-3cm] (qbb) at (1,0) {};
\node[xshift=8cm,yshift=-3cm] (qcc) at (-0.6,1.8) {};
\node[xshift=8cm,yshift=-3cm] (qdd) at (1,1.1) {};
\draw[lightgray](qaa)--(qbb);
\draw(qcc)--(qbb);
\draw[lightgray](qcc)--(qdd);
\draw[lightgray](qaa)--(qdd);
\draw(qaa)--(qcc);
\draw(qbb)--(qdd);

\node[xshift=8cm,yshift=-3.9cm] (qaaa) at (4,1) {};
\node[xshift=8cm,yshift=-3.9cm] (qbbb) at (2,1) {};
\node[xshift=8cm,yshift=-3.9cm] (qccc) at (3.6,-0.8) {};
\node[xshift=8cm,yshift=-3.9cm] (qddd) at (2,-0.1) {};
\draw[lightgray](qaaa)--(qbbb);
\draw(qccc)--(qbbb);
\draw[lightgray](qccc)--(qddd);
\draw[lightgray](qaaa)--(qddd);
\draw(qaaa)--(qccc);
\draw(qbbb)--(qddd);

\draw(qdd)--(qbbb);
\draw(qbb)--(qddd);
\draw[lightgray](qaa)--(qddd);
\draw[lightgray](qdd)--(qaaa);

\node [draw=white, fill=white] (a) at (14,-9)  {(g)};


\node [draw=white, fill=white, xshift=12cm] (a) at (0,-1)  {\tiny $1$};
\node [draw=white, fill=white, xshift=12cm] (a) at (0.8,0)  {\tiny $1$};
\node [draw=white, fill=white, xshift=12cm] (a) at (-0.8,0)  {\tiny $1$};
\node [draw=white, fill=white, xshift=12cm] (a) at (0,1)  {\tiny $1$};
\node [draw=white, fill=white, xshift=12cm] (a) at (0.4,0.1)  {\tiny $1$};
\node [draw=white, fill=white, xshift=12cm] (a) at (-0.4,0.1)  {\tiny $1$};
\node [draw=white, fill=white, xshift=12cm] (a) at (-1.3,0)  {\tiny $-1$};
\node [draw=white, fill=white, xshift=12cm] (a) at (1.3,0)  {\tiny $-1$};

\node [draw=white, fill=white,xshift=12cm] (a) at (-0.7,-1.1)  {\tiny $[v]$};
\node [draw=white, fill=white,xshift=12cm] (a) at (0.7,-1.1)  {\tiny $[w]$};
\node [draw=white, fill=white,xshift=12cm] (a) at (0.7,1.1)  {\tiny $[x]$};
\node [draw=white, fill=white,xshift=12cm] (a) at (-0.7,1.1)  {\tiny $[z]$};

\node[xshift=12cm] (p1) at (-0.7,-0.7) {};
\node[xshift=12cm] (p2) at (0.7,-0.7) {};
\node[xshift=12cm] (p3) at (0.7,0.7) {};
\node[xshift=12cm] (p4) at (-0.7,0.7) {};

\draw[lightgray](p1)--(p2);
\draw(p3)--(p2);
\draw[lightgray](p3)--(p4);
\draw(p1)--(p4);
\draw(p4)--(p2);
\draw[lightgray](p1)--(p3);
\path(p1) edge [bend left=30,lightgray] (p4);
\path(p2) edge [bend right=30] (p3);

\node [draw=white, fill=white,xshift=12cm,yshift=-3cm] (a) at (-1,-0.5)  { $p_{\tv}$};
\node [draw=white, fill=white,xshift=12cm,yshift=-3cm] (a) at (0.8,-0.5)  { $p_{\tw}$};
\node [draw=white, fill=white,xshift=12cm,yshift=-3cm] (a) at (-0.6,2.2)  { $p_{\tz}$};
\node [draw=white, fill=white,xshift=12cm,yshift=-3cm] (a) at (1,1.6)  { $p_{\tx}$};

\node[xshift=12cm,yshift=-3cm] (qaa) at (-1,0) {};
\node[xshift=12cm,yshift=-3cm] (qbb) at (1,0) {};
\node[xshift=12cm,yshift=-3cm] (qcc) at (-0.6,1.8) {};
\node[xshift=12cm,yshift=-3cm] (qdd) at (1,1.2) {};
\draw[lightgray](qaa)--(qbb);
\draw(qcc)--(qbb);
\draw[lightgray](qcc)--(qdd);
\draw[lightgray](qaa)--(qdd);
\draw(qaa)--(qcc);
\draw(qbb)--(qdd);

\node[xshift=12cm,yshift=-4.2cm] (qaaa) at (4,1) {};
\node[xshift=12cm,yshift=-4.2cm] (qbbb) at (2,1) {};
\node[xshift=12cm,yshift=-4.2cm] (qccc) at (3.6,-0.8) {};
\node[xshift=12cm,yshift=-4.2cm] (qddd) at (2,-0.2) {};
\draw[lightgray](qaaa)--(qbbb);
\draw(qccc)--(qbbb);
\draw[lightgray](qccc)--(qddd);
\draw[lightgray](qaaa)--(qddd);
\draw(qaaa)--(qccc);
\draw(qbbb)--(qddd);

\draw(qbb)--(qddd);
\draw(qdd)--(qbbb);
\draw[lightgray](qaa)--(qccc);
\draw[lightgray](qcc)--(qaaa);

\node [draw=white, fill=white] (a) at (20,-9)  {(h)};


\end{tikzpicture}
\end{center}
\vspace{-0.2cm}
\caption{Examples of $\mathbb{Z}_2$-gain graphs and their covering graphs. These are precisely the base $\mathbb{Z}_2$-gain graphs for the $(2,2,0)$-gain-tight inductive construction described in Section~\ref{Sect:Inductive}.
The bottom rows illustrate $\chi_0$-symmetrically isostatic realisations for the $\ell^\infty$-plane under half-turn rotational symmetry. The monochrome subgraphs induced by these realisations (described in Section~\ref{Sect:Gridlike}) are indicated in black and grey.}
\label{fig:gain_cov_graphs}
\end{figure}

The {\em gain} of a path of directed edges $F=[v_1],[e_1],[v_2], \ldots, [e_k],[v_k]$ in a gain graph $(G_0,\psi)$ (where $[v_1]$ may be equal to $[v_k]$) is defined as the product $\psi(F) =\Pi_{i=1}^k \,\psi([e_i])^{\textrm{sign}([e_i])}$, where $\textrm{sign}([e_i])=1$ if $[e_i]$ is directed from $[v_i]$ to $[v_{i+1}]$ and $\textrm{sign}([e_i])=-1$ if $[e_i]$ is directed from $[v_{i+1}]$ to $[v_i]$.
A set of edges $F$ is {\em balanced} if it does not contain a cycle of edges, or,  has the property that every cycle of edges in $F$ has gain $1$. A subgraph of $G_0$ is {\em balanced} in $(G_0,\psi)$  if its edge set is balanced; otherwise, the subgraph is {\em unbalanced}. 

\begin{lemma}[\cite{jkt,zas}] \label{switch}
Let $G_0$ be a quotient graph and fix an orientation on the edges of $G_0$.  If a subgraph $H_0$ is balanced for some gain assignment on the directed   quotient graph $G_0$ then,
\begin{enumerate}[(i)]
\item $H_0$ is balanced for every equivalent gain assignment on the directed quotient graph $G_0$.
\item
there exists an equivalent gain assignment $\psi$ on the directed  quotient graph $G_0$ which satisfies $\psi([e])=1$ for all $[e]\in E(H_0)$.
\end{enumerate}
\end{lemma}

\subsection{A special case}
\label{s:specialcase}
Let $\G=(G,p,\theta,\tau)$ be a well-positioned and $\Gamma$-symmetric bar-joint framework in $X$ and suppose the $\Gamma$-symmetric graph $(G,\theta)$ has an associated gain graph which is balanced.  By Lemma \ref{switch}, there exists a choice of vertex orbit representatives $\tilde{V}_0$ such that the induced gain assignment satisfies $\psi([e])=1$ for all $[e]\in E(G_0)$. It follows that $G_0$ is a simple graph. Consider the well-positioned bar-joint framework $(G_0,\tilde{p})$ in $X$ where $\tilde{p}_{[v]}=p_{\tilde{v}}$ for each vertex orbit $[v]\in V_0$ and vertex orbit representative $\tilde{v}\in \tilde{V}_0$ with $\tilde{v}\in[v]$. The following lemma shows the relationship between the differential $df_{G_0}(\tilde{p})$ and the components of the block decomposition of $df_G(p)$ described in Proposition \ref{prop:block}.

\begin{lemma}
\label{lem:commutativediagram}
Let $\chi\in\hat{\Gamma}$ and
define a pair of linear transformations,
\[S_\chi(\G):(X_\bC)^{V_0}\to X_\chi, \,\,\,\,\,\, 
(x_{[v]})_{[v]\in V_0}\mapsto 
( \chi(\gamma) \tau(\gamma)x_{[v]})_{v\in V}\]
where $v=\gamma \tilde{v}$ for some unique $\tilde{v}\in\tilde{V}_0$ and some unique $\gamma\in\Gamma$,
and,
\[T_\chi(\G):\bC^{E_0}\to Y_\chi, \,\,\,\,\,\, 
(x_{[e]})_{[e]\in E_0}\mapsto 
(\chi(\gamma) x_{[e]})_{e\in E}\]
where $e=\gamma(\tilde{v}\tilde{w})$ for some unique $\tilde{v},\tilde{w}\in \tilde{V}_0$ and some unique $\gamma\in\Gamma$.

Then the following diagram commutes.
\[
\begin{tikzcd}
(X_\bC)^{V_0} \arrow{r}{df_{G_0}(\tilde{p})} \arrow[swap]{d}{S_\chi(\G)} & \bC^{E_0} \arrow{d}{T_\chi(\G)}  \\
X_\chi  \arrow{r}{R_\chi(\G)} & Y_\chi
\end{tikzcd}
\]
In particular, $df_{G_0}(\tilde{p})$ and $R_\chi(\G)$ are similar linear transformations.
\end{lemma}

\proof
Let $u=(u_{[v]})_{[v]\in V_0}\in (X_\bC)^{V_0}$ and let $e=vw\in E$. Then $v=\gamma\tilde{v}$ and $w=\gamma\tilde{w}$ for some unique vertex orbit representatives $\tilde{v},\tilde{w}\in \tilde{V}_0$ and some unique $\gamma\in\Gamma$.
Recall from Proposition  \ref{prop:block} that $R_\chi(\G)$ is the restriction of $df_G(p)$ to the subspace $X_\chi\subset (X_\bC)^V$. Thus, by Lemma \ref{lem:differential}, the $e$-component of $(R_\chi(\G)\circ S_\chi(\G))(u)$ is given by $\varphi_{v,w}( \chi(\gamma) \tau(\gamma)(u_{[v]}-u_{[w]}))$.
Also, by applying Lemma \ref{lem:differential} to the bar-joint framework $(G_0,\tilde{p})$, and using Lemma \ref{lem:SupportFunctionals}, we see that  
the $e$-component of $(T_\chi(\G) \circ df_{G_0}(\tilde{p}))(u)$ is given by 
$\chi(\gamma)\varphi_{\tilde{v},\tilde{w}}(u_{[v]} - u_{[w]})
= \varphi_{v,w}(\chi(\gamma)\tau(\gamma)(u_{[v]} - u_{[w]}))$.
\endproof

\subsection{Gain-sparsity}
Let $k\in\bN$, let $l\in\{0,1,\ldots, 2k-1\}$ and let $m\in \{0,1,\ldots, l\}$.

\begin{definition}
A gain graph $(G_0,\psi)$ is \emph{$(k,l,m)$-gain-sparse} if 
\begin{enumerate}[(a)]
\item $|F|\leq k|V(F)|-l$ for any nonempty balanced $F\subseteq E(G_0)$, and,
\item $|F|\leq k|V(F)|-m$ for all $F\subseteq E(G_0)$.
\end{enumerate}
Moreover, $(G_0,\psi)$ is \emph{$(k,l,m)$-gain-tight} if $|E(G_0)|=k|V(G_0)|-m$ 
and $(G_0,\psi)$ is $(k,l,m)$-gain-sparse.
\end{definition}

Consider again the multiplicative cyclic group 
$\bZ_n=\{\omega^k:k=0,1,\ldots,n-1\}$ with characters $\chi_j(\omega)=\omega^j$ for $j=0,1,\ldots, n-1$.
A $\bZ_n$-symmetric bar-joint framework $\G=(G,p,\theta,\tau)$ in $X$ is said to be {\em $\C_n$-symmetric} if $\tau(\omega)$ is an $n$-fold rotation of $X$.

\begin{corollary} \label{cor:neccounts}
Let $\G=(G,p,\theta,\tau)$ be a well-positioned and $\C_n$-symmetric bar-joint framework in $X$,  where $n\geq 2$, and let $d=\dim_\bR X$. Suppose, in addition, that 
$\T(X)$ is minimal and $\G$ is $\chi_j$-symmetrically isostatic.

\begin{enumerate}[(i)]
\item
Suppose $n=2$.
\begin{enumerate}[(a)]
\item If $j=0$ then $(G_0,\psi)$ is $(d,d,d-2)$-gain-tight.
\item If $j=1$ then, $(G_0,\psi)$ is $(d,d,2)$-gain-tight.
\end{enumerate}
\item 
Suppose $n\geq 3$. 
\begin{enumerate}[(a)]
\item If $j=0$ then $(G_0,\psi)$ is $(d,d,d-2)$-gain-tight.
\item If $j\in \{1,n-1\}$ then, $(G_0,\psi)$ is $(d,d,1)$-gain-tight.
\item If $j\notin \{0,1,n-1\}$ then, $(G_0,\psi)$ is $(d,d,0)$-gain-tight.
\end{enumerate}
\end{enumerate}
\end{corollary}

\proof
Let $\chi=\chi_j$.
Note that since $\T(X)$ is minimal, every bar-joint framework in $X$ is full.  By Lemma \ref{lem:SymRigid}, $\G$, and every $\C_n$-symmetric subframework of $\G$, is $\chi$-symmetrically full. 
Also note that $\dim X_\chi = (\dim_\bR X)|V_0|$ and $\dim Y_\chi = |E_0|$.
Let $F\subseteq E(G_0)$, let $H_0$ be the subgraph of $G_0$ spanned by the edges in $F$ and let $H$ be the covering graph for $H_0$ in $G$.
By Lemma \ref{lem:blocksurjSym}, the $\C_n$-symmetric subframework $\H=(H,p_H,\theta_H,\tau_H)$ is $\chi$-symmetrically independent. Thus, by Proposition \ref{prop:counts}$(ii)$, \[|E(H_0)|\leq (\dim_\bR X)|V(H_0)| - \dim_\bC \T_\chi(X).\]
If $H_0$ is a balanced subgraph of $G_0$ then we may consider an associated bar-joint framework $(H_0,\tilde{p}_H)$, as described in Section \ref{s:specialcase}.
By Lemma \ref{lem:commutativediagram}, $df_{H_0}(\tilde{p}_{H})$ and $R_\chi(\H)$ are similar linear transformations. It follows that $(H_0,\tilde{p}_{H})$ is an independent subframework of $(G,p)$ and so, 
\begin{eqnarray*}
|E(H_0)|= \rank df_{H_0}(\tilde{p}_{H}) &=&(\dim_\bR X)|V(H_0)| - \dim_\bR \F(H_0,\tilde{p}_{H})\\
&\leq&(\dim_\bR X)|V(H_0)| - \dim_\bR \T(X).
\end{eqnarray*} 
Since $\T(X)$ is minimal, $\dim_\bR \T(X) = \dim_\bR X$. Thus the results now follow from Lemma \ref{lem:dimensionCn} and Proposition \ref{prop:counts}$(iii)$.
\endproof

\begin{remark}
Note that, by the above corollary, for two-dimensional $\chi_j$-symmetrically isostatic bar-joint frameworks with rotational symmetry, the associated gain graph must be either $(2,2,0)$-gain tight, $(2,2,1)$-gain-tight or $(2,2,2)$-gain-tight.
Inductive constructions for $(2,2,1)$- and $(2,2,2)$-gain-tight gain graphs are presented in \cite{NS} (see also \cite{NOP12,NO14}). In the next section we present an inductive construction for $(2,2,0)$-gain-tight gain graphs.
Also note that, in any dimension, the $(k,l,m)$-gain tight counts given by Corollary \ref{cor:neccounts} are the bases of a matroid as was observed in \cite{NS}. (Note however that this matroidal property does not hold for arbitrary triples $k,l,m\in \mathbb{N}$. Indeed it fails in some rigidity contexts \cite{iktan}.)
\end{remark}


\section{An inductive construction of $(2,2,0)$-gain tight $\mathbb{Z}_2$-gain graphs}
\label{Sect:Inductive}
Let $(G_0,\psi)$ be a $\mathbb{Z}_2$-gain graph with covering graph $G$. For simplicity, we will omit the square brackets in the notation of vertices and edges of $(G_0,\psi)$ in this section, and simply write $v$ for the vertex $[v]$, and $(uv,\alpha)$ for the edge $([u],[v])$ with gain $\alpha$. Note that the orientation of the edges of $(G_0,\psi)$ does not matter, since $(G_0,\psi)$ is a $\mathbb{Z}_2$-gain graph and $\mathbb{Z}_2$ is of order 2. 
For the remainder of this article we will only consider $\mathbb{Z}_2$-gain graphs and so from now on the term {\em gain graph} will be used to mean {\em $\mathbb{Z}_2$-gain graph}.

\subsection{Base graphs}

Let $\B$ denote the family of $(2,2,0)$-gain-tight \emph{base graphs} presented in Figure \ref{fig:gain_cov_graphs}.
It will be convenient to assign names to elements of $\B$.
Let $iK_j^\ell$ denote the complete graph on $j$ vertices, with $i$ copies of each edge and $\ell$ loops on each vertex. Then $2K_2^1$ and $K_3^1$ with a balanced $K_3$ are the gain graphs in Figures~\ref{fig:gain_cov_graphs}(a) and (b), respectively.
The graph formed from $2K_3$ by adding a loop and deleting an edge not incident with the vertex with the loop will be denoted by $R$. (See Figure~\ref{fig:gain_cov_graphs}(c).)
We denote by $K_4^+$ any $\mathbb{Z}_2$-gain graph formed from a balanced copy of $K_4$ by adding a single edge (subject to $(2,2,0)$-gain-sparsity).
We shall also use $K_4^{++}$ to denote any one of the non-isomorphic $(2,2,0)$-gain-tight gain graphs formed from $K_4$ by adding two edges. (See Figures \ref{fig:gain_cov_graphs}(d)-(h).)

\subsection{Preliminaries}

We first record two preliminary lemmas about gain graphs which go back to Zaslavsky \cite{zas}.

\begin{lemma}\label{lem:switchto-1}
Let $G_0$ be a (simple) cycle.
A $\mathbb{Z}_2$-gain graph $(G_0,\psi)$ is unbalanced if and only if the vertices in $V_0$ can be switched so that any one edge has non-identity gain and every other edge in the resulting $\mathbb{Z}_2$-gain graph $(G_0,\psi')$ has identity gain.
\end{lemma}

\begin{lemma}\label{lem:balancedunion}
Let $(G_0,\psi)$ be a $\mathbb{Z}_2$-gain graph and let $A$ and $B$ be subgraphs of $(G_0,\psi)$. Suppose that $A\cap B$ is connected. If $A$ and $B$ are balanced then $A\cup B$ is also balanced.
\end{lemma}

We will also need some elementary results about sparse graphs which we record here for convenience. Let $f(G_0)=2|V_0|-|E_0|$. So, for example, any $(2,2,0)$-gain-tight gain graph $G_0$ satisfies $f(G_0)=0$ while any balanced subgraph $G_0'$ satisfies $f(G_0')\geq2$.

\begin{lemma}\label{lem:4regno0}
Let $G_0$ be connected and 4-regular. Then $f(G_0')\geq 1$ for any proper subgraph $G_0'\subset G_0$.
\end{lemma}

\begin{proof}
Suppose $G_0$ contains a subgraph $G_0'$ with $f(G_0')=0$. Then $G_0'$ has average degree 4, so $G_0'$ must be 4-regular (by the 4-regularity of $G_0$). Since $G_0$ is connected it follows that $G_0'=G_0$.
\end{proof}

For two disjoint vertex sets $A,B\subset V(G_0)$, we denote by $d(A,B)$ the number of edges between $A$ and $B$.

\begin{lemma}
\label{l:claim}
Let $H_0=(V_0',E_0')$ be a subgraph of $G_0$. 
If the degree of $v$ in $G_0$ is at least 4  for all $v\in V_0'$ then $d(V_0',V_0-V_0')\geq 2f(H_0)$.
\end{lemma}

\proof
Since $|E_0'|=2|V_0'|-f(H_0)$  and every vertex in $V_0'$ has degree at least 4 in $G_0$ we have, 
\[4|V_0'|\leq \sum_{v\in V_0'} \deg_{G_0}(v) = 2|E_0'|+d(V_0',V_0-V_0')
= 4|V_0'| - 2f(H_0) +d(V_0',V_0-V_0').\]
\endproof

\begin{lemma}\label{lem:23con}
Let $(G_0,\psi)$ be $(2,2,0)$-gain-sparse, and let $G_0'$ be a balanced subgraph of $(G_0,\psi)$ with $f(G_0')\in \{2,3\}$. Then $G_0'$ is connected.
\end{lemma}

\begin{proof}
Suppose $G_0'$ is disconnected. Let $A$ be a connected component of $G_0'$ and let $B=G_0'-A$. Since any subgraph of a balanced gain graph is also balanced, we have $f(A)\geq 2$ and $f(B)\geq 2$. Hence $f(G_0')=f(A)+f(B)\geq 4$, contradicting the hypothesis of the lemma.
\end{proof}

\begin{lemma}\label{lem:unionint}
Let $(G_0,\psi)$ be $(2,2,0)$-gain-tight.
Let $H_0'$ and $H_0''$ be balanced subgraphs of $(G_0,\psi)$ with $V(H_0')\cap V(H_0'')\neq \emptyset$ and $f(H_0')=2=f(H_0'')$. Then either
\begin{enumerate}[(i)]
\item $f(H_0'\cap H_0'')=4$ and $f(H_0'\cup H_0'')=0$, or,
\item $f(H_0'\cup H_0'')=f(H_0'\cap H_0'')=2$.
\end{enumerate}
Moreover, $(ii)$ holds if and only if $H_0'\cup H_0''$ is balanced.
\end{lemma}

\begin{proof}
As $H_0'\cap H_0'' \subset H_0'$ we have $f(H_0'\cap H_0'')\geq 2$. If $f(H_0'\cap H_0'')\geq 4$ then $$0\leq  f(H_0'\cup H_0'')=f(H_0')+f(H_0'')-f(H_0'\cap H_0'')=
4-f(H_0'\cap H_0'')\leq 0$$ and so $(i)$ holds. If $f(H_0'\cap H_0'') \in \{2,3\}$ then $H_0'\cap H_0''$ is connected by Lemma \ref{lem:23con}. It follows that $H_0'\cup H_0''$ is balanced by Lemma \ref{lem:balancedunion} and hence $f(H_0'\cup H_0'')\geq 2$. Thus, $$2\leq f(H_0'\cup H_0'')=2+2-f(H_0'\cap H_0'')\leq 2$$ 
and so $(ii)$ holds.
\end{proof}

\subsection{Henneberg-type operations}

Now we define operations on $\mathbb{Z}_2$-gain graphs. The H1 operation (or Henneberg 1 move, or 0-extension) adds a new vertex of 3 possible types. In type 1a the new vertex has degree 2 and two distinct neighbours; in type 1b the new vertex has degree 2 and one neighbour with two parallel edges; and in type 1c the new vertex has degree 3 with one neighbour and a loop. (See Figure \ref{fig:2}.) The gains on the new edges are arbitrary subject to the condition that the covering graph is simple, i.e. parallel edges have different gains and a loop has gain $-1$.

\begin{figure}[ht]
\begin{center}
\begin{tikzpicture}[scale=0.6]

\draw (-6,0.1) circle (40pt);

\filldraw (-6.8,.75) circle (3pt);
\filldraw (-5.2,.75) circle (3pt);
\filldraw (-6,2.5) circle (3pt);

\draw[black]
(-6.8,.75) -- (-6,2.5) -- (-5.2,.75);


\draw (0,.1) circle (40pt);

\filldraw (0,.75) circle (3pt);
\filldraw (0,2.5) circle (3pt);

\draw[thick] plot[smooth, tension=1] coordinates{(0,.75) (-.25,1.6) (0,2.5)};
\draw[thick] plot[smooth, tension=1] coordinates{(0,.75) (.25,1.6) (0,2.5)};

\filldraw (-.25,1.6) circle (0pt) node[anchor=east]{};
\filldraw (.25,1.6) circle (0pt) node[anchor=west]{};


\draw (6,.1) circle (40pt);

\filldraw (6,.75) circle (3pt);
\filldraw (6,2.5) circle (3pt);

\draw[black]
(6,.75) -- (6,2.5);

\draw[thick] plot[smooth, tension=1] coordinates{(6,2.5) (5.7,3) (6.3,3) (6,2.5)};

\filldraw (6,3) circle (0pt) node[anchor=south]{};

\end{tikzpicture}
\end{center}
\caption{H1 a, b, c operations on gain graphs. Gain labels are omitted. }
\label{fig:2}
\end{figure}
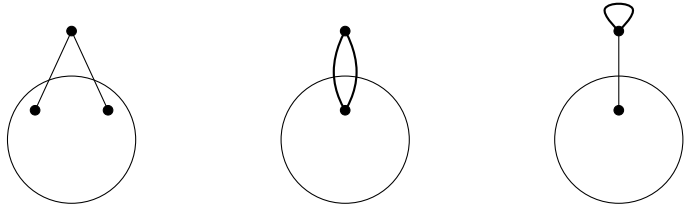

The H2 operation (or Henneberg 2 move, or 1-extension) deletes one edge $(xy,\alpha)$ and adds a new vertex $v$ adjacent to $x,y$ of five possible types. In type 2a, $v$ has degree 3 and 3 distinct neighbours with edges $(xv,\beta)$ and $(yv,\gamma)$ satisfying $\beta\gamma=\alpha$; 
in type 2b, $v$ has degree 3 and exactly 2 neighbours with edges $(xv,1), (xv,-1)$ and $(yv,\delta)$ with $\delta=\pm 1$; 
in type 2c, the deleted edge $xy$ is a loop $(xx,-1)$ and $v$ has degree 3 and exactly 2 neighbours with edges $(xv,1), (xv,-1)$ and $(yv,\delta)$ with $\delta=\pm 1$;
in type 2d, $v$ has degree 4 and exactly 2 neighbours with edges $(xv,\beta), (yv,\gamma)$ and $(vv,-1)$ satisfying $\alpha=\beta\gamma$; 
finally, in type 2e, the deleted edge $xy$ is a loop $(xx,-1)$, $v$ has degree 4 and exactly 1 neighbour with edges $(xv,1), (xv,-1)$ and $(vv,-1)$. (See Figure~\ref{fig:3b}.)

\begin{figure}[ht]
\begin{center}
\begin{tikzpicture}[scale=0.6]

\draw (-6,0) circle (40pt);

\filldraw (-6.8,.75) circle (3pt);
\filldraw (-5.2,.75) circle (3pt);
\filldraw (-6,2.5) circle (3pt);
\filldraw (-6,.1) circle (3pt);

\draw[black]
(-6.8,.75) -- (-6,2.5) -- (-5.2,.75);

\draw[black]
(-6,.1) -- (-6,2.5);

\draw[dashed]
(-6.8,.75) -- (-6,.1);
\end{tikzpicture}
\hspace{0.7cm}
\begin{tikzpicture}[scale=0.6]


\draw (0,0) circle (40pt);

\filldraw (-.8,.75) circle (3pt);
\filldraw (0,2.5) circle (3pt);
\filldraw (0.8,.75) circle (3pt);

\draw[thick] plot[smooth, tension=1] coordinates{(-.8,.75) (-.2,1.6) (0,2.5)};
\draw[thick] plot[smooth, tension=1] coordinates{(-.8,.75) (-.7,1.6) (0,2.5)};

\filldraw (-.5,1.6) circle (0pt) node[anchor=east]{};
\filldraw (-.4,1.6) circle (0pt) node[anchor=west]{};

\draw[black]
(.8,.75) -- (0,2.5);

\draw[dashed]
(-.8,.75) -- (0.8,.75);

\end{tikzpicture}
\hspace{0.7cm}
\begin{tikzpicture}[scale=0.6]

\draw (6,0) circle (40pt);

\filldraw (5.2,.75) circle (3pt);
\filldraw (6,2.5) circle (3pt);
\filldraw (6.8,.75) circle (3pt);

\draw[thick] plot[smooth, tension=1] coordinates{(5.2,.75) (5.8,1.6) (6,2.5)};
\draw[thick] plot[smooth, tension=1] coordinates{(5.2,.75) (5.3,1.6) (6,2.5)};

\filldraw (5.5,1.6) circle (0pt) node[anchor=east]{};
\filldraw (5.7,1.6) circle (0pt) node[anchor=west]{};

\draw[black]
(6.8,.75) -- (6,2.5);

\draw[dashed] plot[smooth, tension=1] coordinates{(5.2,.75) (4.9,.25) (5.5,.25) (5.2,.75)};

\end{tikzpicture}
\hspace{0.7cm}
\begin{tikzpicture}[scale=0.6]

\draw (0,0) circle (40pt);

\filldraw (-.8,.75) circle (3pt);
\filldraw (0,2.5) circle (3pt);
\filldraw (0.8,.75) circle (3pt);

\filldraw (-.5,1.6) circle (0pt) node[anchor=east]{};
\filldraw (0,1.6) circle (0pt) node[anchor=west]{};

\draw[black]
(.8,.75) -- (0,2.5) -- (-.8,.75);

\draw[dashed]
(-.8,.75) -- (0.8,.75);

\draw[thick] plot[smooth, tension=1] coordinates{(0,2.5) (-.3,2.9) (.3,2.9) (0,2.5)};

\filldraw (0,2.8) circle (0pt) node[anchor=south]{};
\end{tikzpicture}
\hspace{0.7cm}
\begin{tikzpicture}[scale=0.6]

\draw (6,0) circle (40pt);

\filldraw (6,.75) circle (3pt);
\filldraw (6,2.5) circle (3pt);

\draw[thick] plot[smooth, tension=1] coordinates{(6,.75) (5.8,1.6) (6,2.5)};
\draw[thick] plot[smooth, tension=1] coordinates{(6,.75) (6.2,1.6) (6,2.5)};

\filldraw (6,1.6) circle (0pt) node[anchor=east]{};
\filldraw (6,1.6) circle (0pt) node[anchor=west]{};
\filldraw (6,2.8) circle (0pt) node[anchor=south]{};

\draw[dashed] plot[smooth, tension=1] coordinates{(6,.75) (5.7,.25) (6.3,.25) (6,.75)};
\draw[thick] plot[smooth, tension=1] coordinates{(6,2.5) (5.7,2.9) (6.3,2.9) (6,2.5)};

\end{tikzpicture}
\end{center}
\caption{H2 a, b, c, d, e operations on gain graphs. Gain labels are omitted.}
\label{fig:3b}
\end{figure}
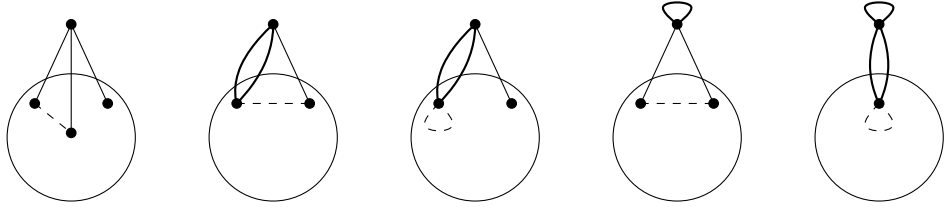

The H3 operation (or X-replacement, or 2-extension) deletes two edges $(xy,\alpha),(zw,\beta)$ and adds a new degree 4 vertex $v$ adjacent to $x,y,z,w$ of five possible types. 
In type 3a, $v$ has 4 distinct neighbours and edges $(xv,\gamma),(yv,\delta),(zv,\epsilon),(wv,\zeta)$ where $\alpha=\gamma\delta$ and $\beta=\epsilon\zeta$; 
in type 3b, $v$ has 3 distinct neighbours, $y=z$ and there are two parallel edges between $v$ and $y$, with edges $(xv,\gamma),(yv,1),(yv,-1),(wv,\zeta)$ where $\alpha=\gamma$ and $\beta=-\zeta$; 
in type 3c, $v$ has 3 distinct neighbours, $x=y$ so $\alpha=-1$ and there are two parallel edges between $v$ and $x$ with edges $(xv,-1),(xv,1),(zv,\epsilon),(wv,\zeta)$ and $\beta=\epsilon\zeta$; 
in type 3d, $v$ has 2 distinct neighbours, $x=y$ and $z=w$ so $\alpha=\beta=-1$ and there are two parallel edges between $v$ and $x$ and between $v$ and $z$ with edges $(xv,1),(xv,-1),(zv,1),(zv,-1)$.    
(See Figure~\ref{fig:3d}.)

\begin{figure}[ht]
\begin{center}
\begin{tikzpicture}[scale=0.6]

\draw (-6,0) circle (40pt);

\filldraw (-6.9,.75) circle (3pt);
\filldraw (-5.1,.75) circle (3pt);
\filldraw (-6,2.5) circle (3pt);
\filldraw (-6.3,.1) circle (3pt);
\filldraw (-5.7,.1) circle (3pt);

\draw[black]
(-6.9,.75) -- (-6,2.5) -- (-5.1,.75);

\draw[black]
(-6.3,.1) -- (-6,2.5);

\draw[black]
(-5.7,.1) -- (-6,2.5);

\draw[dashed]
(-6.8,.75) -- (-6.3,.1);

\draw[dashed]
(-5.1,.75) -- (-5.7,.1);
\end{tikzpicture}
\hspace{0.9cm}
\begin{tikzpicture}[scale=0.6]

\draw (0,0) circle (40pt);

\filldraw (-.8,.75) circle (3pt);
\filldraw (0,2.5) circle (3pt);
\filldraw (0.8,.75) circle (3pt);
\filldraw (0,.1) circle (3pt);

\draw[thick] plot[smooth, tension=1] coordinates{(-.8,.75) (-.2,1.6) (0,2.5)};
\draw[thick] plot[smooth, tension=1] coordinates{(-.8,.75) (-.7,1.6) (0,2.5)};

\filldraw (-.5,1.6) circle (0pt) node[anchor=east]{};
\filldraw (-.4,1.6) circle (0pt) node[anchor=west]{};

\draw[black]
(.8,.75) -- (0,2.5);

\draw[black]
(0,.1) -- (0,2.5);

\draw[dashed]
(-.8,.75) -- (0.8,.75);

\draw[dashed]
(-.8,.75) -- (0,.1);
\end{tikzpicture}
\hspace{0.9cm}
\begin{tikzpicture}[scale=0.6]


\draw (6,0) circle (40pt);

\filldraw (5.2,.75) circle (3pt);
\filldraw (6,2.5) circle (3pt);
\filldraw (6.8,.75) circle (3pt);
\filldraw (6,.1) circle (3pt);

\draw[thick] plot[smooth, tension=1] coordinates{(5.2,.75) (5.8,1.6) (6,2.5)};
\draw[thick] plot[smooth, tension=1] coordinates{(5.2,.75) (5.3,1.6) (6,2.5)};

\filldraw (5.5,1.6) circle (0pt) node[anchor=east]{};
\filldraw (5.7,1.6) circle (0pt) node[anchor=west]{};

\draw[black]
(6.8,.75) -- (6,2.5);

\draw[black]
(6,.1) -- (6,2.5);

\draw[dashed]
(6,.1) -- (6.8,.75);

\draw[dashed] plot[smooth, tension=1] coordinates{(5.2,.75) (4.9,.25) (5.5,.25) (5.2,.75)};

\end{tikzpicture}
\hspace{0.9cm}
\begin{tikzpicture}[scale=0.6]

\draw (0,0) circle (40pt);

\filldraw (-.8,.75) circle (3pt);
\filldraw (0,2.5) circle (3pt);
\filldraw (0.8,.75) circle (3pt);

\filldraw (-.5,1.6) circle (0pt) node[anchor=east]{};
\filldraw (-.4,1.6) circle (0pt) node[anchor=west]{};

\filldraw (.7,1.6) circle (0pt) node[anchor=east]{};
\filldraw (.7,1.6) circle (0pt) node[anchor=west]{};

\draw[thick] plot[smooth, tension=1] coordinates{(-.8,.75) (-.7,1.7) (0,2.5)};
\draw[thick] plot[smooth, tension=1] coordinates{(-.8,.75) (-.2,1.7) (0,2.5)};

\draw[thick] plot[smooth, tension=1] coordinates{(.8,.75) (.7,1.7) (0,2.5)};
\draw[thick] plot[smooth, tension=1] coordinates{(.8,.75) (.2,1.7) (0,2.5)};

\draw[dashed] plot[smooth, tension=1] coordinates{(.8,.75) (1.1,.4) (.5,.4) (.8,.75)};
\draw[dashed] plot[smooth, tension=1] coordinates{(-.8,.75) (-1.1,.4) (-.5,.4) (-.8,.75)};

\end{tikzpicture}
\end{center}
\caption{H3 a, b, c, d operations on gain graphs. Gain labels are omitted.}
\label{fig:3d}
\end{figure}
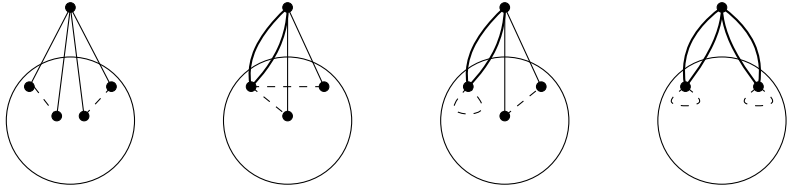

A \emph{vertex-to-$K_4$} operation removes a vertex $v$ (of arbitrary degree) and 
all the edges incident with $v$, and adds in a copy of $K_4$ with only trivial gains. 
Each removed edge $(xv,\gamma)$, where $x\not= v$, is replaced by an edge $(xy,\gamma)$ for some $y$ in the new $K_4$. If $x=v$ then the removed edge $(vv,-1)$ is replaced by an edge $(wz,-1)$ where $w,z$ are vertices of the new $K_4$. Note $w$ and $z$ need not be distinct. (See Figure \ref{fig:vk4}.)

\begin{center}
\begin{figure}[ht]
\centering
\begin{tikzpicture}[scale=0.9]
\draw (0,0) circle (27pt);
\draw (5,0) circle (27pt);
 
\filldraw (0,1.5) circle (3pt);
\filldraw (.2,0.5) circle (3pt);
\filldraw (-.2,0.5) circle (3pt);
\filldraw (0.6,0.35) circle (3pt);
\filldraw (-0.6,0.35) circle (3pt);

\filldraw (4.65,1.35) circle (3pt);
\filldraw (5.35,1.35) circle (3pt);
\filldraw (4.65,1.8) circle (3pt);
\filldraw (5.35,1.8) circle (3pt);

\filldraw (5.2,0.5) circle (3pt);
\filldraw (4.8,0.5) circle (3pt);
\filldraw (5.6,0.35) circle (3pt);
\filldraw (4.4,0.35) circle (3pt);

\draw[black]
(-.6,0.35) -- (0,1.5)  -- (-.2,.5) -- (0,1.5) -- (.2,.5);
\draw[black]
(.6,.35) -- (0,1.5);

\draw[black]
(2,0) -- (2.85,0);

\draw[black]
(2.85,0.15) -- (2.85,-.15) -- (3,0) -- (2.85,.15); 

\draw[black]
(4.4,0.35) -- (4.65,1.35) -- (5.35,1.35) -- (4.65,1.8) -- (5.35,1.8) -- (5.35,1.35) -- (5.6,.35);

\draw[black]
(5.2,0.5) -- (4.65,1.8) -- (4.65,1.35) -- (4.8,.5);

\draw[black]
(4.65,1.35) -- (5.35,1.8);

\end{tikzpicture}
\hspace{1cm}
\begin{tikzpicture}[scale=0.9]
\draw (0,0) circle (27pt);
\draw (5,0) circle (27pt);

\filldraw (-.6,-.1) circle (3pt);
\filldraw (.6,.1) circle (3pt);

\draw[black]
(-.8,-.1) -- (-.6,-.1) -- (.6,.1) -- (.8,.1);

\draw[black]
(-.6,-.3) -- (-.6,-.1) -- (-.6,.1);

\draw[black]
(.6,.3) -- (.6,.1) -- (.6,-.1);

\filldraw (4.4,-.1) circle (3pt);
\filldraw (5.6,-.2) circle (3pt);
\filldraw (5.6,.4) circle (3pt);

\draw[black]
(4.2,-.1) -- (4.4,-.1) -- (5.6,-.2) -- (5.8,-.2);

\draw[black]
(4.4,-.3) -- (4.4,-.1) -- (4.4,.1);

\draw[black]
(4.4,-.1) -- (5.6,.4) -- (5.6,-.2) -- (5.6,-.4);

\draw[black]
(5.6,.4) -- (5.6,.6);

\draw[black]
(2,0) -- (2.85,0);

\draw[black]
(2.85,0.15) -- (2.85,-.15) -- (3,0) -- (2.85,.15); 

\end{tikzpicture}
\caption{The vertex-to-$K_4$ operation and the vertex splitting operation. Gain labels are omitted.}
\label{fig:vk4}
\end{figure}
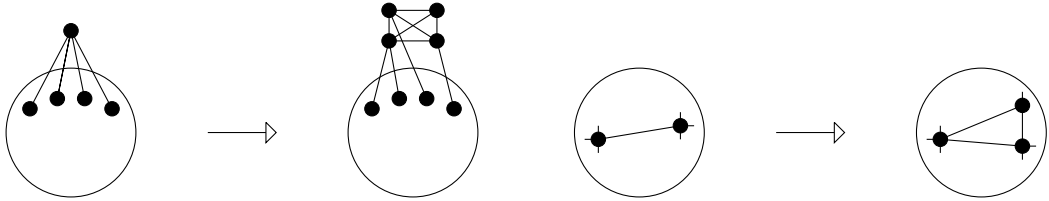
\end{center}

A \emph{vertex splitting} operation first chooses a vertex $v_1$, a neighbour $v_2$ of $v_1$, and a
partition $N_1,N_2$ of the remaining neighbours of $v_1$; it then deletes the edges from $v_1$ to vertices in $N_1$, adds a new vertex $v_0$ joined to vertices in $N_1$ and finally adds two new edges $v_0v_1,v_0v_2$.
If there is a loop at $v_1$ then it is either left unchanged or replaced by a loop at $v_0$. 
We specify that $v_0v_1$ is given gain 1 and $v_0v_2$ is given the same gain as $v_1v_2$. 
(See Figure \ref{fig:vk4}.) 

By construction we have the following.

\begin{lemma}\label{lem:moves}
Applying any of the above operations to a $(2,2,0)$-gain-tight gain graph results in a $(2,2,0)$-gain-tight gain graph.
\end{lemma}

\begin{proof}
When the operation is a H1, H2 or H3 operation then we may employ similar arguments to those in \cite[Lemma 4.1 and 7.6]{jkt}.

Suppose $(G,\psi)$ is obtained from $(H,\psi')$ by a vertex-to-$K_4$ operation at $v$. If $(G,\psi)$ is not $(2,2,0)$-gain-tight then there exists a vertex-induced subgraph $G_1$ of $(G,\psi)$ such that $f(G_1)<0$, or, $(G_1,\psi|_{G_1})$ is balanced and $f(G_1)\in \{0,1\}$.
Consider the subgraph $H_1$ of $H$ corresponding to $G_1$ obtained on contracting $G$ to $H$. Note that in our definition of $H_1$, if there is a loop at $v$ in $H$ then $H_1$ will contain this loop if and only if $G_1$ contains the extra edge $(wz,-1)$ in the copy of $K_4^+$. Note that $(H_1,\psi'|_{H_1})$ is balanced if and only if $(G_1,\psi|_{G_1})$ is balanced. 
There are two possibilities: either $|V(K_4\cap G_1)|\in\{1,4\}$, or, $|V(K_4\cap G_1)|\in \{2,3\}$.
In the first case, $f(H_1)=f(G_1)$ and in the second case $f(H_1)=f(G_1)-1$, contradicting  $(2,2,0)$-gain-sparsity.

Lastly, suppose $(G,\psi)$ is obtained from a $(2,2,0)$-gain-tight gain graph $(H,\psi')$ by a vertex splitting operation at the vertex $v_1$ which adjoins the new vertex  $v_0$. 
Suppose that $(G,\psi)$ is not $(2,2,0)$-gain-tight. Since $f(G)=0$ it follows that  there exists a vertex-induced subgraph $G_1$ of $(G,\psi)$ such that $f(G_1)<0$, or, $(G_1,\psi|_{G_1})$ is balanced and $f(G_1)\in \{0,1\}$. Consider the subgraph $H_1$ of $H$ corresponding to $G_1$ obtained on contracting $G$ to $H$. Note that if $G_1$ contains $v_0$ and $v_1$, then $f(H_1)=f(G_1)-1$. 
Otherwise, $f(H_1)=f(G_1)$. Hence $H_1$ violates $(2,2,0)$-gain-sparsity.
\end{proof}

\subsection{Reducing low-degree vertices via reverse Henneberg-type operations}

Note that if a gain graph $G_0'$ is obtained from a $(2,2,0)$-gain-tight gain graph by reversing any of the above operations then $f(G_0')=0$. Thus $G_0'$ is $(2,2,0)$-gain-tight if and only if each subgraph of $G_0'$ satisfies the $(2,2,0)$-sparsity counts.
A vertex $v$ in a $(2,2,0)$-gain-tight gain graph is \emph{admissible} if there is a reverse $H1$ operation, a reverse $H2$ operation or a reverse $H3$ operation removing $v$ which results in a $(2,2,0)$-gain-tight gain graph. Similarly a balanced subgraph isomorphic to $K_4$ or $K_3$ is \emph{admissible} if there is a $K_4$-contraction (i.e. a reverse vertex-to-$K_4$ operation) or edge contraction (i.e. a reverse vertex splitting operation) which results in a (2,2,0)-gain-tight gain graph.

Our first lemma is trivial and deals with all H1 moves.

\begin{lemma}\label{lem:deg2}
Let $(G_0,\psi)$ be a $(2,2,0)$-gain-tight gain graph. Suppose $v\in V_0$ has degree 2 or is incident to a loop and has degree 3. Then $v$ is admissible.
\end{lemma}

We now work through the H2 moves in turn. 

\begin{lemma}\label{lem:3threei}
Let $(G_0,\psi)$ be a $(2,2,0)$-gain-tight gain graph.
Suppose $v\in V_0$ has degree 3 with exactly three neighbours $a,b,c$. Then $v$ is admissible if and only if it is not contained in a balanced subgraph isomorphic to $K_4$.
\end{lemma}

\begin{proof}
Suppose $v$ is admissible. Then there exists a $(2,2,0)$-gain-tight gain graph $(G_0',\psi')$ which is the result of a reverse H2a operation at $v$. If $v$ is contained in a balanced $K_4$ subgraph then the deleted edge in $(G_0',\psi')$ must be one of two parallel edges with equal gain, contradicting the simplicity of the covering graph for $(G_0',\psi')$. 

For the converse, suppose $v$ is not contained in a balanced subgraph isomorphic to $K_4$.
Then there exists a gain graph $(G_0',\psi')$ which is the result of a reverse H2a operation at $v$. Let $(av,\alpha)$, $(bv,\beta)$ and $(cv,\gamma)$ be in $E_0$.

Suppose that there exists a subgraph $H_{ab}$ of $G_0-v$ which contains $a,b$ with $f(H_{ab})=0$. We claim that $v$ must be admissible. 
To see this first note that $c\notin V(H_{ab})$ since otherwise $f(H_{ab}\cup v)<0$, which contradicts $(2,0)$-sparsity.
If the edges $(ac,\alpha\gamma),(bc,\beta\gamma)$ are in $G_0$ then the union of $H_{ab}$ with $v,c$ and the edges $ac,bc,va,vb,vc$ violates $(2,0)$-sparsity.
If $(ac,\alpha\gamma)\notin E(G_0)$ then there exists a subgraph $H_{ac}$ of $G_0-v$ containing $a,c$ such that $f(H_{ac})\leq 2$. In this case, $f(H_{ab}\cup H_{ac}\cup v)<0$, which is a contradiction.  A similar argument holds for the pairs $a,c$ and $b,c$. Thus we may assume for each pair of vertices $s,t\in \{a,b,c\}$ that there is no  subgraph $H_{st}$ of $G_0-v$ which contains $s,t$ and satisfies $f(H_{st})=0$.

Now assume, without loss of generality, that $(ab,\alpha\beta) \notin E_0$. 
Suppose there does not exist a balanced subgraph $H_{ab}$ of $G_0-v$ which contains $a,b$ with $f(H_{ab})=2$ and all paths from $a$ to $b$ having gain $\alpha\beta$.
Then $v$ is  admissible since adding the edge $(ab,\alpha\beta)$ will not violate $(2,2,0)$-gain-sparsity.

Suppose $G_0-v$ does contain a balanced subgraph $H_{ab}$ which contains $a,b$ with $f(H_{ab})=2$ and all paths from $a$ to $b$ having gain $\alpha\beta$.
We may assume by gain switching (Lemma \ref{switch}) that all edges of $H_{ab}$ have gain 1. In this case, all paths from $a$ to $b$ have gain $1$.
If $\alpha\not=\beta$ then we claim that $v$ is admissible. To see this note that if there exists a balanced subgraph $H_{ab}'$ which contains $a,b$ with $f(H_{ab}')=2$ and all paths from $a$ to $b$ having gain $-1$ then $H_{ab}\cap H_{ab}'$ is not connected. Thus, by Lemma \ref{lem:unionint}, $f(H_{ab}\cup H_{ab}')=0$, contradicting our assumption above. Thus $v$ is admissible since we may add the edge $(ab,-1)$. We may now assume $\alpha=\beta$. We may further assume, by gain switching at $v$, that $\alpha=\beta=1$. 

If $H_{ab}$ contained $c$ then $f(H_{ab}\cup v)=1$. If $vc$ has gain $1$ then we contradict $(2,2,0)$-gain-sparsity and so $\gamma=-1$. 
If $(ac,-1)$ and $(bc,-1)$ are both in $(G_0,\psi)$ then 
the induced subgraph on $V(H_{ab})\cup \{v\}$ violates $(2,0)$-sparsity.
Without loss of generality, suppose $(ac,-1)$ is not in $(G_0,\psi)$.
Then $v$ is  admissible unless there exists a balanced subgraph $H_{ac}$ of $G_0-v$ which contains $a,c$ with $f(H_{ac})=2$ and all paths from $a$ to $c$ having gain $-1$. 

Suppose $G_0-v$ does contain a balanced subgraph $H_{ac}$ which contains $a,c$ with $f(H_{ac})=2$ and all paths from $a$ to $c$ having gain $-1$.
Note that $f(H_{ab}\cup H_{ac})\geq 1$ (for otherwise adding $v$ and its three edges would contradict $(2,2,0)$-gain sparsity).
Thus, by Lemma  \ref{lem:unionint}, $H_{ab}\cup H_{ac}$ is balanced. 
By Lemma \ref{lem:23con}, $H_{ab}$ and $H_{ac}$ are connected and so it follows that $H_{ab}\cup H_{ac}$ contains an unbalanced cycle. This is a contradiction and 
so $v$ is admissible since adding the edge $(ac,-1)$ will not violate $(2,2,0)$-gain-sparsity. 

Now suppose $H_{ab}$ does not contain $c$.
In this case, by gain switching at the vertex $c$ (Lemma \ref{switch}) we can assume $\gamma=1$.
If $(ac,1)$ and $(bc,1)$ are both in $G_0$ then the union of $H_{ab}$ with vertices $v,c$ and the five edges $(ac,1),(bc,1),(va,1),(vb,1),(vc,1)$ is balanced and violates $(2,2)$-sparsity.
This is a contradiction and so, without loss of generality, we may assume $(ac,1)\notin G_0$. 
Then $v$ is  admissible unless there exists a balanced subgraph $H_{ac}$ of $G_0-v$ which contains $a,c$ with $f(H_{ac})=2$ and all paths from $a$ to $c$ having gain $1$.

Suppose there exists a balanced subgraph $H_{ac}$ of $G_0-v$ containing $a,c$ with $f(H_{ac})=2$, with all paths in $H_{ac}$ from $a$ to $c$ having gain $1$. 
By the previous argument we may assume $b$ is not in $H_{ac}$ since otherwise $v$ is admissible. 
Since $f(H_{ab}\cup H_{ac})\geq 1$ (for otherwise adding $v$ and its three edges would contradict $(2,2,0)$-gain sparsity), Lemma \ref{lem:unionint} may be applied, that is $H_{ab}\cup H_{ac}$ is balanced. 
By gain switching, we may assume the edges of $H_{ab}\cup H_{ac}$ all have gain $1$.
Let $\alpha',\beta',\gamma'$ be the resulting gains on the edges $va,vb,vc$ respectively.
If $\alpha'\not=\beta'$ then, by the above argument, $v$ is admissible and we may add the edge $(ab,-1)$. Similarly, $v$ is admissible if $\alpha'\not=\gamma'$. So now suppose $\alpha'=\beta'=\gamma'$. Then $H_{ab}\cup H_{ac}\cup v$ is balanced with $f(H_{ab}\cup H_{ac}\cup v)=1$. 
This contradicts $(2,2,0)$-sparsity. 
\end{proof}

\begin{lemma}\label{lem:3threeii}
Let $(G_0,\psi)$ be a $(2,2,0)$-gain-tight gain graph.
Suppose $v\in V_0$ has degree 3 with exactly two neighbours $a,b$. Then $v$ is admissible if and only if it is not contained in a subgraph isomorphic to $R$ (recall Fig.~\ref{fig:gain_cov_graphs}(c)).
\end{lemma}

\begin{proof}
If $v$ is contained in a subgraph isomorphic to $R$, then $v$ is clearly not admissible. For the converse, suppose that $v$ is not in a subgraph isomorphic to $R$.
Then there exists a gain graph $(G_0',\psi')$ which is the result of either a reverse H2b operation at $v$ or a reverse H2c operation at $v$.
Let $(av,1), (av,-1)$ and $(bv,\alpha)$ be in $E_0$. 

We first consider a reduction move at $v$ which adds an edge between $a$ and $b$.
Observe that any subgraph $H$ of $G_0-v$  containing $a$ and $b$ has $f(H)>0$ (otherwise $f(H\cup v)<0$ would hold) so we need only consider balanced subgraphs.

Suppose there is no edge $ab$. If there is no admissible reverse H2b move then there exist distinct balanced subgraphs $H_1, H_2$ of $G_0-v$ such that $a,b\in V(H_i)$, $f(H_i)=2$ for $i=1,2$ and all paths in $H_i$ from $a$ to $b$ have gain $(-1)^i$. 
Since $f(H_1\cap H_2)\geq 2$ either $H_1\cap H_2$ is not connected and Lemma \ref{lem:unionint} implies that $f(H_1\cap H_2)=0$ and adding $v$ violates $(2,0)$-sparsity or $H_1\cap H_2$ is connected.
Then Lemma \ref{lem:unionint} implies that $H_1\cup H_2$ is balanced. Hence all paths from $a$ to $b$ in $H_1\cap H_2$ have two distinct gains, a contradiction. Thus $v$ is admissible.

Secondly, suppose there is exactly one edge $(ab,\beta)$ in $E_0$. 
If there is no admissible reverse H2b move, then there exists a balanced subgraph $H_{ab}$ of $(G_0,\psi)$ containing $a,b$ but not $v$ with $f(H_{ab})=2$ such that all paths in $H_{ab}$ from $a$ to $b$ have gain $-\beta$. 
Note that $a$ does not have a loop (otherwise $(2,0)$-sparsity would be violated). 
Also, $a$ is not contained in a subgraph $H$ with $f(H)=0$ (otherwise adjoining the edge $(ab,\beta)$ to $H\cup H_{ab}$ will violate $(2,0)$-sparsity). Thus a reverse H2c move can be applied which preserves $(2,2,0)$-gain-sparsity. Thus $v$ is again admissible.

Finally, if both $(ab,1)$ and $(ab,-1)$ are in $E_0$ then the reverse H2c move adding the loop $(aa,-1)$ is non-admissible if and only if there is a subgraph $H$ of $G-v$ containing $a$ which has $f(H)=0$. Note that $H$ does not contain $b$ and so $f(H\cup \{v,b\})=f(H)-1$, giving a contradiction.
\end{proof}

\begin{lemma}\label{lem:4one}
Let $(G_0,\psi)$ be a $(2,2,0)$-gain-tight gain graph. 
 Suppose $v\in V_0$ has degree 4 with exactly one loop at $v$ and one neighbour $a$. 
Then $v$ is admissible if and only if $v$ is not contained in a subgraph isomorphic to $2K_2^1$.
\end{lemma}

\begin{proof}
If $v$ is not contained in a subgraph isomorphic to $2K_2^1$ then it is easy to check that $v$ is admissible for a reverse H2e move adding a loop on $a$. Conversely, if $v$ is contained in a subgraph isomorphic to $2K_2^1$ then $G_0$ contains a loop at $a$ and so $v$ is clearly not admissible.
\end{proof}

We now move on to H3 moves. 
First consider the reverse H3d move.

\begin{lemma}\label{lem:4twob}
Let $(G_0,\psi)$ be a $(2,2,0)$-gain-tight gain graph which is $4$-regular.
Suppose $v\in V_0$ has no loop and exactly two neighbours $a,b$ with a double edge to each. Then $v$ is admissible if and only if $G_0$ does not contain a loop at $a$ and does not contain a loop at $b$.
\end{lemma}

\begin{proof}
Clearly, if $G_0$ contains a loop at $a$ or $b$ then a reverse H3d move cannot be applied and so $v$ is not admissible. For the converse, suppose $G_0$ does not contain any loops at $a$ and $b$. If adding loops on $a$ and $b$ violates $(2,2,0)$-gain-sparsity then either there exists a subgraph $H_0$ of $G_0-v$ containing $a$ and $b$ with $f(H_0)= 1$, or there exists a subgraph $H_0$ of $G_0-v$ containing $a$ (or $b$) with $f(H_0)= 0$. In both cases we may use 4-regularity to get a contradiction. If $f(H_0)=0$ then $a$ must have degree 2 in $H_0$. Note that $H_0$ has average degree 4, giving a vertex $c\in H_0$ with degree greater than $4$ in $H_0$, and hence in $G_0$. This is a contradiction. If $H_0$ contains $a$ and $b$ and $f(H_0)= 1$, then $a$ and $b$ both have degree $2$ in $H_0$.
All other vertices in $H_0$ have degree at most $4$. 
In this case, since $|E(H_0)|=2|V(H_0)|-1$, we obtain the contradiction,
\[\sum_{v\in V(H_0)}\deg_{H_0}(v)
\leq 4(|V(H_0)|-2)+2+2
< 4|V(H_0)|-2=2|E(H_0)|.\]
\end{proof}

\begin{lemma}\label{lem:4threei}
Let $(G_0,\psi)$ be a $(2,2,0)$-gain-tight gain graph which is connected and $4$-regular.
Suppose $v\in V_0$ has exactly three neighbours $a,b,c$ and no loop. Suppose no neighbour of $v$ is admissible, then $v$ is admissible if and only if $v$ is not contained in a subgraph isomorphic to $K_4^+$.
\end{lemma}

\begin{proof}
If $v$ is contained in a subgraph isomorphic to $K_4^+$, then $v$ is clearly not admissible. 
For the converse, suppose $v$ is not contained in a subgraph isomorphic to $K_4^+$.
Let $(av,1)$, $(av,-1)$, $(bv,\beta)$, $(cv,\gamma)$ be the edges incident to $v$.
By Lemma \ref{lem:4regno0}, $(G_0,\psi)$ does not contain any proper subgraph $H$ with $f(H)=0$. Note also that no subgraph $H$ of $G_0-v$ containing $N(v)=\{a,b,c\}$ can have $f(H)=1$. We may suppose there is no loop on $a$ (otherwise we can use Lemma \ref{lem:4one} to see that $a$ is admissible, which is a contradiction). 
Moreover, we may assume that one of the edges $(ac,1)$ and $(ac,-1)$ is not in $G_0$  (otherwise Lemma \ref{lem:4twob} would imply that $a$ is admissible).
Similarly, we may assume that one of the edges $(ab,1)$ and $(ab,-1)$ is not in $G_0$.

First observe that there are three possible reduction moves: a reverse H3c move which adds a loop at $a$ and the edge $(bc,\beta\gamma)$, a reverse H3b move which adds the edges $(ab,\beta),(ac,-\gamma)$ and a reverse H3b move which adds the edges $(ab,-\beta),(ac,\gamma)$.
The reverse H3c move is admissible unless:\\
(a) the edge $(bc,\beta\gamma)$ already exists in $G_0$; or \\
(b) there is a balanced subgraph $H_{bc}$ of $G_0-v$ containing $b,c$ with $f(H_{bc})=2$ in which all paths from $b$ to $c$ have gain $\beta\gamma$.

The reverse H3b move which adds the edges $(ab,\beta),(ac,-\gamma)$ is admissible unless: \\
(c) one of the edges $(ab,\beta),(ac,-\gamma)$ already exists in $G_0$; or\\
(d) there is a balanced subgraph $H_{ab}$ of $G_0-v$ containing $a,b$   with $f(H_{ab})=2$ in which all paths from $a$ to $b$ have gain $\beta$, or, there is a balanced subgraph $H_{ac}$ of $G_0-v$ containing $a,c$   with $f(H_{ac})=2$ in which all paths from $a$ to $c$ have gain $-\gamma$.

The reverse H3b move which adds the edges $(ab,-\beta),(ac,\gamma)$ is admissible unless: \\
(e) one of the edges $(ab,-\beta),(ac,\gamma)$ already exists in $G_0$; or\\
(f) there is a balanced subgraph $H_{ab}$ of $G_0-v$ containing $a,b$   with $f(H_{ab})=2$ in which all paths from $a$ to $b$ have gain $-\beta$, or, there is a balanced subgraph $H_{ac}$ of $G_0-v$ containing $a,c$   with $f(H_{ac})=2$ in which all paths from $a$ to $c$ have gain $\gamma$.

Suppose $(a)$ holds. Note that if $(ab,\beta)$ and $(ac,\gamma)$ both exist or $(ab,-\beta)$ and $(ac,-\gamma)$ both exist then there is a $K_4^+$ containing $v$, which is a contradiction. Thus we may suppose that either $(ab,\beta)$ and $(ac,-\gamma)$ do not exist, or, $(ab,-\beta)$ and $(ac,\gamma)$ do not exist.

Without loss of generality, we suppose that $(ab,\beta)$ and $(ac,-\gamma)$ do not exist. 
If $v$ is not admissible then there exists either a balanced subgraph $H_{ab}$ of $G_0-v$ containing $a,b$ with $f(H_{ab})=2$ and all paths from $a$ to $b$ have gain $\beta$ or a balanced subgraph $H_{ac}$ of $G_0-v$ containing $a,c$ with $f(H_{ac})=2$ and all paths from $a$ to $c$ having gain $-\gamma$. 
If both exist then $H=H_{ab}\cup H_{ac}$ is a proper subgraph of $G_0$ and hence, by Lemma  \ref{lem:4regno0}, $f(H)\geq 1$.
Thus, by Lemma \ref{lem:unionint}, $H$ is balanced and $f(H)=2$.
Note that $(bc,\beta\gamma)\in E(H)$ (otherwise adjoining this edge to $H\cup v$ violates $(2,0)$-sparsity). By switching, we may assume without loss of generality that all gains on $H$ are $1$. In this case note that $\beta=\gamma$. Adjoining $v$ and the edges $(va,\beta),(vb,\beta),(vc,\beta)$ to $H$ results in a balanced subgraph which violates $(2,2)$-sparsity.
This is a contradiction and so either $H_{ab}$ or $H_{ac}$ does not exist.  
Without loss of generality we assume $H_{ab}$ does not exist.
If $H_{ac}$ does not exist then $v$ is admissible. So suppose $H_{ac}$ does exist.

Suppose $(ab,-\beta)\in E(G_0)$. Then $H_{ac}$ contains the vertex $b$. (If not, then the vertex $a$ would have degree at most 1 in $H_{ac}$, a contradiction.)
By switching, we may assume all gains in $H_{ac}$ are $1$. Moreover, $(bc,\beta\gamma)\in E(H_{ac})$ for otherwise $(2,0)$-sparsity is violated.
Thus $\beta=\gamma$. Now adjoining $v$ together with the edges $(va,\beta),(vb,\beta),(vc,\beta)$ to $H_{ac}$ results in a balanced subgraph which violates $(2,2)$-sparsity. We conclude that $(ab,-\beta)\notin E(G_0)$. 
If  $(ac,\gamma)\in E(G_0)$ then $(ac,\gamma)\notin E(H_{ac})$, and hence $a$ has degree at most $1$ in $H_{ac}$, a contradiction.
Thus $(ac,\gamma)\notin E(G_0)$.

If $v$ is not admissible then there exists either a balanced subgraph $H_{ab}'$ of $G_0-v$ containing $a,b$ with $f(H_{ab}')=2$ and all paths from $a$ to $b$ have gain $-\beta$ or a balanced subgraph $H_{ac}'$ of $G_0-v$ containing $a,c$ with $f(H_{ac})=2$ and all paths from $a$ to $c$ having gain $\gamma$.
If $H_{ab}'$ exists then  $H=H_{ab}'\cup H_{ac}$ is a proper subgraph of $G_0$ and hence, by Lemma  \ref{lem:4regno0}, $f(H)\geq 1$.
Thus, by Lemma \ref{lem:unionint}, $H$ is balanced and $f(H)=2$. Note that the edge $(bc,\beta\gamma)$  is not in $H$, for otherwise we may assume, by switching, that all gains in $H$ are $1$, and hence $\beta=\gamma$. As above, this contradicts $(2,2)$-sparsity of $G$.
Now, adjoining the edge $(bc,\beta\gamma)$ to $H$ results in a subgraph of $G_0-v$ with $f(H)=1$ which contains $\{a,b,c\}$. This is a contradiction.
If $H_{ac}'$ exists then $H=H_{ac}\cup H_{ac}'$ is balanced. However, $H$ contains an unbalanced cycle obtained by concatenating a path from $a$ to $c$ in $H_{ac}$ and   a path from $a$ to $c$ in $H_{ac}'$. This is again a contradiction.
We conclude that $v$ is admissible since we may apply a reverse H3b move.

Now suppose $(a)$ does not hold.
Consider a reduction at $v$ adding a loop at $a$ and the edge $(bc,\beta\gamma)$. 
If this move is not admissible then there is a balanced subgraph $H_{bc}$ of $G_0-v$ containing $b,c$ with $f(H_{bc})=2$ in which all paths from $b$ to $c$ have gain $\beta\gamma$. 
By switching, we may assume without loss of generality that all gains of $H_{bc}$ equal  1.
In this case $\beta=\gamma$ since all paths in $H_{bc}$ from $b$ to $c$ have gain $1$. We may assume that $H_{bc}$ does not contain $a$ (otherwise adjoining $v$ to $H_{bc}$ results in a balanced subgraph of $(G_0,\psi)$ which violates $(2,2)$-sparsity). 

Suppose $ab,ac\in E(G_0)$ with arbitrary gains.
The subgraph $H$ obtained from $H_{bc}$ by adjoining the vertices $a,v$ together with the edges $(av,1),(av,-1)$ and all edges between $\{a,v\}$ and $H_{bc}$ has $f(H)=0$.
Hence $H=G_0$. 
If $ab$ and $ac$ have the same gain then note that $H'=H-(va,-\beta)$ is balanced and satisfies $f(H')=1$. This is a contradiction. 
If $ab$ and $ac$ have different gains then, without loss of generality, we can say $(ab,1)$ and $(ac,-1)$ are in $G_0$. Since $H=G_0$ and all gains on $H_{bc}$ are $1$, there is no path in $G_0-a$ between $b$ and $c$ with gain $-1$. It follows that there is an admissible reduction at $a$ adding a loop at $v$ and the edge $(bc,-1)$. 
This is a contradiction since no neighbour of $v$ is admissible.

Without loss of generality, we may now assume there is no edge between $a$ and $b$ in $G_0$. Consider the reduction at $v$ adding $(ab,1)$ and $(ac,-1)$. Suppose there is a balanced subgraph $H_{ab}$ of $G_0-v$ containing $a,b$ with $f(H_{ab})=2$.
Note that $c\notin V(H_{ab})$ since otherwise adjoining $v$ to $H_{ab}$ results in a balanced subgraph which violates $(2,2)$-sparsity.
The subgraph $H$ obtained from the union of $H_{ab}$ with $H_{bc}$ and the edges $(va,\beta),(vb,\beta),(vc,\beta)$ is balanced and violates $(2,2)$-sparsity.
This is a contradiction and so $H_{ab}$ does not exist.
Similarly, there is no balanced subgraph $H_{ac}$ of $G_0-v$ containing $a,c$ with $f(H_{ac})=2$. 
Hence if the reduction at $v$ adding $(ab,1)$ and $(ac,-1)$ is not admissible then $(ac,-1)\in E(G_0)$.   
By a similar argument, if the reduction at $v$ adding $(ab,-1)$ and $(ac,1)$ is not admissible then $(ac,1)\in E(G_0)$.
This is a contradiction since $G_0$ does not contain both $(ac,1)$ and $(ac,-1)$, for otherwise $a$ would be admissible.
\end{proof}

\begin{lemma}\label{lem:admblocks}
Let $(G_0,\psi)$ be a $(2,2,0)$-gain-tight gain graph. 
Suppose  $v\in V_0$ has degree 4 with exactly four neighbours $a,b,c,d$.
Let $(va,\alpha),(vb,\beta),(vc,\gamma)$ and $(vd,\delta)$ be the edges incident to $v$.
Then a reduction at $v$ adding $(ab,\alpha\beta),(cd,\gamma\delta)$ is non-admissible if and only if one of the following conditions holds:
\begin{enumerate}[(i)]
\item there is a subgraph $H_{ab}$ of $G_0$ containing $a,b$ but not $v$ with $f(H_{ab})=0$ or there is a subgraph $H_{cd}$ of $G_0$ containing $c,d$ but not $v$ with $f(H_{cd})=0$;
\item there is a balanced subgraph $H_{ab}$ of $(G_0,\psi)$ containing $a,b$ but not $v$ with $f(H_{ab})=2$ and every path from $a$ to $b$ in $H_{ab}$ has gain $\alpha\beta$ or there is a balanced subgraph $H_{cd}$ of $(G_0,\psi)$ containing $c,d$ but not $v$ with $f(H_{cd})=2$ and every path from $c$ to $d$ in $H_{cd}$ has gain $\gamma\delta$;
\item there is a balanced subgraph $H$ of $(G_0,\psi)$ containing $N(v)$ but not $v$ with $f(H)=3$ where every path from $a$ to $b$ in $H$ has gain $\alpha\beta$ and every path from $c$ to $d$ in $H$ has gain $\gamma\delta$; 
\item one of the edges $(ab,\alpha\beta),(cd,\gamma\delta)$ already exists in $(G_0,\psi)$.
\end{enumerate}
\end{lemma}

\begin{proof}
If any one of these conditions holds then it is clear that $v$ is not admissible. Conversely, suppose that the result of a reduction at $v$ adding $(ab,\alpha\beta),(cd,\gamma\delta)$ is the gain graph $(G_0',\psi')$. If the reduction is non-admissible then $(G_0',\psi')$ is not $(2,2,0)$-gain-tight. It follows that either the covering graph of $(G_0',\psi')$ is not simple, there is a subgraph which violates $(2,0)$-sparsity or there is a subgraph which violates balanced $(2,2)$-sparsity. In the first case $(iv)$ holds. In the second case, suppose $H'$ is a subgraph of $(G_0',\psi')$ which violates $(2,0)$-sparsity. Either $H'$ received one edge in the reduction move and $(i)$ holds or $H'$ received two edges and there is a subgraph $H$ of $(G_0,\psi)$ containing $N(v)$ but not $v$ with $f(H)\leq 1$. However $f(H\cup v)<0$, contradicting $(2,2,0)$-gain-sparsity.

Hence it remains to consider the last case. There is a subgraph $H'$ of $(G_0',\psi')$ which violates balanced $(2,2)$-sparsity. Therefore $H'$ contains $(ab,\alpha\beta)$ or $(cd,\gamma\delta)$ or both. By symmetry we may suppose $(ab,\alpha\beta)\in H'$. If $H'$ does not contain $(cd,\gamma\delta)$ then this implies that there is a balanced subgraph $H$ of $(G_0,\psi)$ containing $a,b$ but not $v$ with $f(H)=2$ and every path from $a$ to $b$ in $H$ has gain $\alpha\beta$, giving $(ii)$. Similarly if $H'$ contains $(ab,\alpha\beta)$ and $(cd,\gamma\delta)$ then there is a balanced subgraph $H$ of $(G_0,\psi)$ containing $N(v)$ but not $v$ with $f(H)=3$ where every path from $a$ to $b$ in $H$ has gain $\alpha\beta$ and every path from $c$ to $d$ in $H$ has gain $\gamma\delta$, giving $(iii)$.
\end{proof}

\begin{lemma}\label{lem:4threeii}
Let $(G_0,\psi)$ be a $(2,2,0)$-gain-tight gain graph which is connected and $4$-regular. 
Suppose  $v\in V_0$ has exactly four neighbours $a,b,c,d$. Then $v$ is admissible if and only if 
the following conditions hold,
\begin{enumerate}[(i)]
\item $v$ is not contained in a balanced subgraph isomorphic to $K_4$, and,
\item $v$ is not contained in a balanced subgraph isomorphic $K_{1,1,3}$ with the property that $v$ has degree $4$ in this subgraph.
\end{enumerate}
\end{lemma}

\begin{proof}
If (i) or (ii) hold then $v$ is clearly not admissible. For the converse, let $(va,\alpha),(vb,\beta),(vc,\gamma)$ and $(vd,\delta)$ be the edges incident to $v$.

We first show that if condition (ii) holds in Lemma \ref{lem:admblocks}, then  $c$ and $d$ are not in $H_{ab}$. Suppose $N(v)\subset V(H_{ab})$. Since $H_{ab}$ is balanced, we may assume, by switching, that all gains on $H_{ab}$ are $1$. Since all paths from $a$ to $b$  in $H_{ab}$ have gain $\alpha\beta$, it follows that  $\alpha=\beta$. If at least 3 edges incident to $v$ all have gain 1 or all have gain $-1$, then we contradict $(2,2)$-sparsity. Thus, without loss of generality, we may assume that $(va,1),(vb,1),(vc,-1),(vd,-1) \in E(G_0)$. Since $G_0$ is 4-regular, Lemma \ref{lem:4regno0} implies that $H_{ab}$ together with the edges incident to $v$ is all of $G_0$. In particular $H_{ab}$ was an induced subgraph of $G_0$. Thus it is easy to see that the reduction move at $v$ adding $(ad,-1),(bc,-1)$ is admissible. 

So without loss of generality, we may suppose that $d\notin V(H_{ab})$. We claim that $c\notin V(H_{ab})$. Suppose to the contrary that $c\in V(H_{ab})$. By the same argument as above, we may assume that $(va,1),(vb,1),(vc,-1) \in E(G_0)$. Notice that if we add the edges  $(va,1),(vb,1),(vc,-1)$ to $H_{ab}$, then we obtain a subgraph $H$ with $f(H)=1$. Thus, $H_{ab}$ is an induced subgraph of $G$, for otherwise the fact that $d\notin V(H_{ab})$ would give a contradiction, by Lemma \ref{lem:4regno0}. Consider the reduction move at $v$ adding $(bc,-1)$ and $(ad,\delta)$. If $v$ is not admissible then either the edge $(ad,\delta)$ is in $G_0$, or there is a balanced subgraph $H_{ad}$ of $G_0-v$  containing $a,d$ with $f(H_{ad})=2$ in which all paths from $a$ to $d$ have gain $\delta$.
Note first that if there is such a graph $H_{ad}$ then $H=H_{ab}\cup H_{ad}$ is balanced with $f(H)=2$ and $N(v)\subset V(H)$.  So we can repeat the argument used to show that $d\notin V(H_{ab})$ to obtain a contradiction. So suppose  $(ad,\delta)$ is in $G_0$. By switching at the vertex $d$, we may assume that $\delta=1$. We claim that the reduction move at $v$ adding $(ac,-1),(bd,1)$ is now admissible. If $(bd,1)$ is present in $G_0$, then $H_{ab}$ together with the 5 edges $(va,1),(vb,1),(vd,1),(ad,1),(bd,1)$ violates $(2,2)$-sparsity. So suppose there is a balanced  subgraph $H_{bd}$ of $G_0-v$  containing $b,d$ with $f(H_{bd})=2$ in which all paths from $b$ to $d$ have gain $1$. But then $H=H_{ab}\cup H_{bd}$ is balanced with $f(H)=2$ and $N(v)\subset V(H)$, and hence the same argument from above may be applied once again to obtain a contradiction. Thus, as claimed,  $c\notin V(H_{ab})$.

Now, suppose that conditions (i) and (ii) in the statement of this lemma do not hold. Then condition (iv) in Lemma \ref{lem:admblocks} cannot hold for all possible pairs of edges, so we may assume that the edges $(ab,\alpha\beta),(cd,\gamma\delta)$ do not exist in $(G_0,\psi)$. Since $G_0$ is $4$-regular, Lemma \ref{lem:4regno0} implies that condition (i) in Lemma \ref{lem:admblocks} does not hold. Suppose $H$ is a balanced subgraph of $(G_0,\psi)$ that satisfies condition (iii) in Lemma \ref{lem:admblocks}. By switching, we may assume that all gains on $H$ are $1$.  By the assumptions in condition  (iii) it follows that $\alpha\beta=\gamma\delta=1$, contradicting $(2,2)$-sparsity of $G_0$ if $\alpha=\beta=\gamma=\delta$. So suppose without loss of generality that $\alpha=\beta=1$ and $\gamma=\delta=-1$. Note that at most one of the edges $(ac,-1), (ad,-1), (bc,-1), (bd,-1)$ can be present in $G_0$, for otherwise the graph obtained from $H$ by adding two of those edges together with the four edges incident to $v$ violates $(2,0)$-sparsity. Hence $(ad,-1), (bc,-1)\notin E(G_0)$ or $(ac,-1), (bd,-1)\notin E(G_0)$. Without loss of generality we assume that $(ad,-1), (bc,-1)\notin E(G_0)$ and consider the reduction operation on $v$ that adds $(ad,-1)$ and $(bc,-1)$. 

Suppose there exists a balanced subgraph $H'$ of $(G_0,\psi)$ that satisfies condition (iii) in Lemma \ref{lem:admblocks} for the pair $(ad,-1)$, $(bc,-1)$, that is,  all paths from $a$ to $d$ and all paths from $b$ to $c$ in $H'$ have gain $-1$. Then $H\cap H'$ is disconnected (since there cannot be any path from $a$ to $d$ or from $b$ to $c$ in  $H\cap H'$). It follows that $f(H\cap H')\geq 2+2= 4$ and hence $f(H\cup H')\leq 3+3-4=2$. Thus, $f(H\cup H')= 2$, for otherwise the graph obtained from $H\cup H'$ by adding the four edges incident to $v$ would violate $(2,0)$-sparsity. This says that $H\cap H'$ has exactly two components $C_1$ and $C_2$, with either $a,b\in V(C_1)$ and $c,d\in V(C_2)$ or $a,c\in V(C_1)$ and $b,d\in V(C_2)$.

 Suppose first that $a,b\in V(C_1)$ and $c,d\in V(C_2)$. Then every path in $H'$ from $a$ to $b$ and from $c$ to $d$ has gain $1$. Since every path from $a$ to $d$ and from $b$ to $c$ has gain $-1$ in $H'$, and $H'$ is balanced, it follows that every path from $a$ to $c$ and from $b$ to $d$ must have gain $-1$. Therefore, $H'$ together with the four edges incident to $v$ is balanced and violates $(2,2)$-sparsity. 
 
Suppose next that $a,c\in V(C_1)$ and $b,d\in V(C_2)$. Then every path in $H'$ from $a$ to $c$ and from $b$ to $d$ has gain $1$. This implies that $(ac,-1)$ and $(bd,-1)$ cannot exist in $G_0$. (If one of those edges did exist in $G_0$, then it cannot be an edge of $H\cup H'$, and hence $H\cup H'$ together with this edge and the four edges incident to $v$ would violate $(2,0)$-sparsity.) We now consider the reduction operation on $v$ that adds $(ac,-1)$ and $(bd,-1)$. Suppose there exists a balanced subgraph $H''$ of $(G_0,\psi)$ that satisfies condition (iii) in Lemma \ref{lem:admblocks} for the pair $(ac,-1)$, $(bd,-1)$, that is,  all paths from $a$ to $c$ and all paths from $b$ to $d$ in $H''$ have gain $-1$. Note  that $(H\cup H')\cap H''$ is disconnected since there cannot be any path from $a$ to $c$ or from $b$ to $d$ in $(H\cup H')\cap H''$. Therefore, $f((H\cup H')\cap H'')\geq 2+2=4$ and hence $f((H\cup H')\cup H'')\leq 2+3-4=1$. So the subgraph obtained from $(H\cup H')\cup H''$ by adding the four edges incident to $v$ violates $(2,0)$-sparsity. So if the reduction operation on $v$ that adds $(ac,-1)$ and $(bd,-1)$ is not admissible, then condition (ii) in Lemma \ref{lem:admblocks} holds for $(ac,-1)$, $(bd,-1)$. As we have shown in the beginning of this proof, we may assume that there exists a balanced subgraph $H_{ac}$ of $(G_0,\psi)$ containing $a,c$ but not $b,d,v$ with $f(H_{ac})=2$ and every path from $a$ to $c$ in $H_{ac}$ has gain $-1$. Then $H\cap H_{ac}$ is disconnected (since there cannot be a path from $a$ to $c$ in  $H\cap H_{ac}$) and hence $f(H\cup H_{ac})\leq 3+2-4=1$. It follows that $H\cup H_{ac}$ together with the four edges incident to $v$ violates the $(2,0)$-sparsity of $G_0$.

Therefore, if  $v$  is not admissible, then condition (ii) in Lemma \ref{lem:admblocks} holds for $(ad,-1)$, $(bc,-1)$. As we have shown in the beginning of this proof, we may assume that there exists a balanced subgraph $H_{bc}$ of $(G_0,\psi)$ containing $b,c$ but not $a,d,v$ with $f(H_{bc})=2$ and every path from $b$ to $c$ in $H_{bc}$ has gain $-1$. But then, by the same argument as in the paragraph above, we have $f(H\cap H_{bc})\leq 3+2-4=1$, which contradicts the $(2,0)$-sparsity of $G_0$.

So if $v$ is not admissible, then condition (ii) in Lemma \ref{lem:admblocks} holds for $(ab,\alpha\beta),(cd,\gamma\delta)$. By using the argument in the  beginning of this proof again, we may assume that there exists a balanced subgraph $H_{ab}$ of $(G_0,\psi)$ containing $a,b$ but not $c,d,v$ with $f(H_{ab})=2$ and every path from $a$ to $b$ in $H_{ab}$ has gain $\alpha\beta$.
Consider the reduction operation on $v$ that adds $(ad,\alpha\delta)$ and $(bc,\beta\gamma)$.

If this move is not admissible then one of the conditions (ii), (iii) or (iv) in Lemma \ref{lem:admblocks} holds for $(ad,\alpha\delta),(bc,\beta\gamma)$. Suppose first that (iv) fails, that is, $(ad,\alpha\delta)$ and $(bc,\beta\gamma)$ are not edges of $(G_0,\psi)$. Suppose there exists a balanced subgraph $H$ of $(G_0,\psi)$ that satisfies condition (iii) for the pair $(ad,\alpha\delta),(bc,\beta\gamma)$, that is,  all paths from $a$ to $d$ in $H$ have gain $\alpha\delta$ and all paths from $b$ to $c$ in $H$ have gain $\beta\gamma$. If $H\cap H_{ab}$ is connected, then every path from $a$ to $b$ in $H$ has gain $\alpha\beta$. But since $H$ is balanced, this implies that the subgraph $H'$ of $G_0$ consisting of $H$ and the four edges incident to $v$ is balanced and satisfies $f(H')=1$, contradicting $(2,2)$-sparsity. Thus $H\cap H_{ab}$ is disconnected, and hence $f(H\cap H_{ab})\leq 3+2-4=1$. But since 
$H\cap H_{ab}$ contains all four neighbours of $v$, this contradicts $(2,0)$-sparsity. So we may assume that condition (ii) in Lemma \ref{lem:admblocks} holds for $(ab,\alpha\beta),(cd,\gamma\delta)$. Without loss of generality we may assume that there exists a balanced subgraph $H_{ad}$ of $(G_0,\psi)$ containing $a,d$ but not $b,c,v$ with $f(H_{ad})=2$ and every path from $a$ to $d$ in $H_{ad}$ has gain $\alpha\delta$. If $H_{ad}\cap H_{ab}$ is connected, then $H_{ad}\cup H_{ab}$ is balanced with $f(H_{ad}\cup H_{ab})=2$, and the subgraph $H'$ of $(G_0,\psi)$ consisting of $H_{ad}\cup H_{ab}$ and the three edges joining $v$ with $a,b$ and $d$ is also balanced with $f(H')=1$, contradicting $(2,2)$-sparsity. Thus $H_{ad}\cap H_{ab}$ is disconnected, and hence $f(H_{ad}\cup H_{ab})\leq 2+2-4=0$. But since 
$H_{ad}\cup H_{ab}$ contains three neighbours of $v$, this contradicts $(2,0)$-sparsity.

So without loss of generality we may assume that $(bc,\beta\gamma)$ is an edge of $(G_0,\psi)$. Then the edge $(bd,\beta\delta)$ cannot exist in $G_0$ (for otherwise $b$ would have degree at most 1 in $H_{ab}$ by $4$-regularity of $G_0$, a contradiction). The edge $(ac,\alpha\gamma)$ can also not exist in $G_0$, for otherwise the graph consisting of $H_{ab}$ and the edges $(va,\alpha), (vb,\beta), (vc,\gamma), (ac,\alpha\gamma), (bc,\beta\gamma)$ would be balanced and would violate $(2,2)$-sparsity. So consider the reduction operation on $v$ that adds $(ac,\alpha\gamma)$ and $(bd,\beta\delta)$. 
If there exists a balanced subgraph $H$ of $(G_0,\psi)$ that satisfies condition (iii) for the pair $(ac,\alpha\gamma),(bd,\beta\delta)$, then, by the same argument as in the paragraph above, $H\cap H_{ab}$ is disconnected, and hence $f(H\cap H_{ab})\leq 3+2-4=1$. But since $H\cap H_{ab}$ contains all four neighbours of $v$, this contradicts $(2,0)$-sparsity. So we may assume that condition (ii) in Lemma \ref{lem:admblocks} holds for $(ac,\alpha\gamma),(bd,\beta\delta)$. Without loss of generality we may assume that there exists a balanced subgraph $H_{ac}$ of $(G_0,\psi)$ containing $a,c$ but not $b,d,v$ with $f(H_{ac})=2$ and every path from $a$ to $c$ in $H_{ac}$ has gain $\alpha\gamma$. As in the paragraph above, we see that $H_{ac}\cap H_{ab}$ is disconnected, and hence $f(H_{ac}\cup H_{ab})\leq 2+2-4=0$. But since 
$H_{ac}\cup H_{ab}$ contains three neighbours of $v$, this contradicts $(2,0)$-sparsity.
  \end{proof}

\subsection{Graph contractions}

We now consider the existence of suitable triangles or $K_4$'s in order to apply the reverse vertex splitting move or the reverse vertex-to-$K_4$ move. After giving conditions on when they can be applied we will also prove a couple of technical lemmas needed in the next section.

\begin{lemma}\label{lem:k4contract}
Let $(G_0,\psi)$ be a $(2,2,0)$-gain-tight gain graph. Suppose $(G_0,\psi)$ contains a balanced subgraph $K$ isomorphic to $K_4$ which induces at most one additional edge. 
Then a reverse vertex-to-$K_4$ move at $K$ 
is admissible unless there is a vertex $x$ and edges $(xa,\alpha), (xb,\alpha)$ for some $a,b\in V(K)$.
\end{lemma}

\begin{proof}
Let $K'$ denote the contraction of the graph $K^*$ induced by $K$. (So $K'$ is either a single vertex or a single vertex with a loop). Then $f(K^*)=f(K')$ and hence the lemma follows from a simple counting argument.
\end{proof}

\begin{lemma}\label{lem:k3contract}
Let $(G_0,\psi)$ be a $(2,2,0)$-gain-tight gain graph. Suppose $(G_0,\psi)$ contains a  subgraph $K$ isomorphic to $K_3$ with $V(K)=\{a,b,c\}$ and with all three edges of $K$ having gain $1$. Suppose the edge $(ab,-1)$ does not exist in $G_0$. Then a reverse vertex-splitting at $K$ contracting the edge $(ab,1)$ is non-admissible  if and only if one of the following conditions holds: 
\begin{enumerate}
\item[(i)] there is a subgraph $H_0$ of $G_0$ containing $a,b$ and the edge $(ab,1)$, but not $c$, with $f(H_0)=0$;
\item[(ii)] there is a balanced copy of $K_3$ containing $a,b$ and some vertex $d\neq c$; 
\item[(iii)] there is a balanced subgraph $H_0$ of $G_0$ containing $a,b$ and the edge $(ab,1)$ with $f(H_0)=2$, and if $H_0$ contains $c$ then it does not contain the edges $(ca,1)$ and $(cb,1)$; 
\item[(iv)] there are loops incident to both $a$ and $b$, or both edges $(ac,-1)$ and $(bc,-1)$ exist.
\end{enumerate}
\end{lemma}

\begin{proof} If any one of these conditions holds then it is clear that $v$ is not admissible. Conversely, suppose the reverse vertex-splitting at $K$ contracting the edge $(ab,1)$ is non-admissible. Let $(G_0',\psi')$ be the gain graph  resulting from this reverse vertex-splitting move and let $\overline a$ be the vertex of $(G_0',\psi')$ corresponding to the vertex pair $a$ and $b$ in $(G_0,\psi)$. 
 Then either  the covering graph of  $(G_0',\psi')$ is not simple, or there is a subgraph $H_0'$ of  $(G_0',\psi')$  with $f(H_0')<0$, or there is a balanced subgraph of $(G_0',\psi')$ with $f(H_0')<2$. In the first case (ii) or (iv) holds. In the second case, $H_0'$ clearly contains $\overline{a}$, but it cannot contain $c$, for otherwise $f(H_0)=f(H_0')<0$ if $(c\overline{a},1)\in E(H_0')$ or $f(H_0\cup \{(ca,1)\})=f(H_0')<0$ if $(c\overline{a},1)\notin E(H_0')$, where $H_0$ is the subgraph of $(G_0,\psi)$ that is obtained from $H_0'$ by the vertex splitting move at $\overline{a}$. Thus, $H_0$ contains $a,b$ and the edge $(ab,1)$, but not $c$, and satisfies $f(H_0)=0$.
 In the third case, the balanced $H_0'$ again clearly contains $\overline{a}$. If it contains $c$ and $(c\overline{a},1)$ then  $f(H_0)=f(H_0')<2$, a contradiction.  So if $H_0'$  contains  $c$ then it does not contain $(c\overline{a},1)$. Thus, $H_0$ contains $a,b$ and the edge $(ab,1)$, and if $H_0$ contains $c$ then it does not contain the edges $ca$ and $cb$. Moreover,  $H_0$ is balanced with $f(H_0)=2$.
\end{proof}
 
We now follow the approach in \cite{NO14}. Define a \emph{triangle sequence} $T_1,T_2,\dots,T_n$ where $T_1$ is a triangle on vertices $a,b,c$ and $T_{i+1}$ is formed from $T_i$ by adding a vertex of degree 2 adjacent to two vertices $x,y$ of $T_i$ such that $xy\in E(T_i)$ and $x,y$ are in exactly one triangle in $T_i$. A triangle sequence is \emph{balanced} if each $T_i$ (or equivalently just $T_n$) is balanced. A \emph{maximal balanced triangle sequence} is a balanced triangle sequence that cannot be extended to a larger balanced triangle sequence.
A \emph{chord} of $T_n$ is an edge in the subgraph $G_0[V(T_n)]$ of $G_0$ induced by $V(T_n)$   which is not in $T_n$.

The following lemma is easy to deduce from the definitions.

\begin{lemma}\label{lem:chord}
Let $(G_0,\psi)$ be $(2,2,0)$-gain-tight gain graph and let $T_n$ be a balanced triangle sequence whose edges all have gain $1$.
Then $f(T_n)=3$ and $T_n$ has at most 3 chords. Moreover, if three chords exist, at least two of them have gain $-1$.
\end{lemma}

In the next lemma we show that there exists an edge for which conditions (ii) and (iv) of Lemma \ref{lem:k3contract} do not hold, provided that the maximal balanced triangle sequence has enough vertices.

\begin{lemma}\label{lem:bound}
Let $(G_0,\psi)$ be $(2,2,0)$-gain-tight and let $T_n$ be a maximal balanced triangle sequence. Suppose $|V(T_n)|\geq 6$. Then there exists an edge of $T_n$ with the properties that it does not have a parallel edge, its end vertices are not both incident to a loop, and it is contained in exactly one balanced triangle in $(G_0,\psi)$.
\end{lemma}

\begin{proof}
Since $T_n$ is balanced we may assume that the gain on every edge of $T_n$ is 1.
Let $s_1,s_2,\dots, s_r$ be the edges of $T_n$ contained in exactly one triangle in $T_n$. 

We first show that the $s_i$ form a simple  cycle spanning $V(T_n)$. This can be verified by induction. It clearly holds for $n=1$. Suppose it holds for all $m<n$ and consider $T_n$. Let $T_n$ be formed from $T_{n-1}$ by adding a triangle on $a,b,c$ such that $a,b\in V(T_{n-1})$ and $c\notin V(T_{n-1})$. Then by induction there is a simple spanning cycle $C$ in $T_{n-1}$ with $ab\in C$. We construct the simple spanning cycle for $T_n$ by removing the edge $ab$ from $C$ and adding the edges $ac,cb$ to $C$.

By Lemma~\ref{lem:chord}, $T_n$ has at most $3$ chords, and  at most one them has gain 1. Let $C$ be the simple cycle spanning $V(T_n)$ consisting of the edges $s_1,s_2,\ldots, s_r$. 
Since $V(C)$ has at least 6 vertices, $C$ has at least $6$ edges. We say that an edge of $C$ is \emph{blocked} if it has a parallel edge, or its endvertices are both incident to a loop, or it is contained in more than one balanced triangle in $(G_0,\psi)$. We claim that $C$ contains at least one edge that is not blocked.

Note that a chord with gain $1$ can block at most three edges of $C$.
A chord with gain $-1$ can block at most one edge of $C$, since it cannot create a balanced triangle on its own, but it could be parallel to an edge of $C$. Further, two  non-loop chords with gain $-1$ only create a balanced triangle if they are of the form $(cd,-1)$ and $(ce,-1)$ for some $c,d,e\in V(T_n)$. In this case at most one edge of $C$ is contained in a triangle with these two chords. Thus, two non-loop chords with gain $-1$ can block at most two edges of $C$. Similarly, three non-loop chords with gain $-1$ can block at most three edges of $C$. Finally, two chords that are loops can block at most one edge of $C$, and three chords that are loops can block at most two edges of $C$.

So if at least two of the three chords are loops, then there are at most 4 blocked edges in $C$. So suppose exactly one of the chords is a loop. Since at most one of the remaining two chords can have gain $1$, it follows that  there are again at most 4 blocked edges in $C$. Finally, if none of the chords is a loop, then there are at most 5 blocked edges in $C$.
Thus there exists an edge of $C$ that is not blocked, as claimed.
\end{proof}

\subsection{The inductive construction}

We can now put together our results to prove the desired characterisation of $(2,2,0)$-gain-tight gain graphs.

\begin{theorem}\label{thm:recurse}
Let $(G_0,\psi)$ be a gain graph. Then $(G_0,\psi)$ is $(2,2,0)$-gain-tight if and only if $(G_0,\psi)$ can be generated from vertex disjoint copies of graphs in $\B$ by applying H1, H2, H3, vertex-to-$K_4$ and vertex splitting moves.
\end{theorem}

\begin{proof}
The easy direction is given by Lemma \ref{lem:moves}. For the converse, observe first that if $(G_0,\psi)$ is disconnected then every connected component of $(G_0,\psi)$ is $(2,2,0)$-gain-tight. So we will prove that an arbitrary \emph{connected} $(2,2,0)$-gain-tight gain graph  has an admissible reverse move which results in another (possibly disconnected) $(2,2,0)$-gain-tight gain graph with fewer vertices, or is one of the base graphs in $\B$. The theorem then follows by induction on $|V_0|$.

So from now on we will assume that $(G_0,\psi)$ is connected. Observe that $(G_0,\psi)$ is  either 4-regular or contains a vertex of degree $2$ or $3$. We consider the following cases.\\

\textbf{Case 1. $G_0$ contains a vertex with two incident edges or a degree 3 vertex with exactly two neighbours, which is not contained in a subgraph isomorphic to $R$.}\\

Lemmas \ref{lem:deg2} and \ref{lem:3threeii} show that there exists an admissible reverse move.\\

We can now assume that Case 1 does not hold. In the following, we let $\hat{K_4}$ be the $(2,2,0)$-gain-tight gain graph consisting of a balanced $K_4$ as well as one additional vertex $x$ and the four edges $(xa,1),(xa,-1), (xb,1), (xb,-1)$ where $a$ and $b$ are distinct vertices of the $K_4$.\\

\textbf{Case 2.a. $G_0$ contains a degree 3 vertex $v$ with exactly 3 neighbours and $v$ is not contained in a subgraph isomorphic to $K_4^+$ or $\hat{K_4}$.}\\

Lemma \ref{lem:3threei} implies $v$ is admissible or $v$ is contained in a balanced subgraph $K$ isomorphic to $K_4$.  
By Lemma \ref{switch} we may assume every edge of $K$ has gain $1$.
Note that $K$ is an induced subgraph of $G_0$. If $K$ is not admissible to contract then Lemma \ref{lem:k4contract} implies there is a vertex $x$ and edges $(xa,\alpha), (xb,\alpha)$ for $a,b \in V(K)$. 
We may apply Lemma \ref{switch} again at $x$ to make $\alpha=1$. 
Now we have a balanced $K_3$ on $x,a,b$. We will denote a balanced $K_3$ on vertices $r,s,t$ by  $K_3(r,s,t)$.

Consider a maximal balanced triangle sequence $T_1=K_3(v,a,c), T_2=T_1\cup K_3(a,b,c), T_3=T_2\cup K_3(a,b,x), \dots, T_n$ (where $c$ is the final vertex of $K$).
By Lemma \ref{switch} we may assume every edge of $T_n$ has gain $1$.
By Lemma \ref{lem:k3contract} an edge $(rs,1)$ in a $K_3(r,s,t)$ in $T_n$, with $(rs,-1)$ not an edge of $G_0$, is non-admissible in $G_0$ for a reverse vertex-splitting move if and only if condition (i), (ii), (iii) or (iv) holds. 

We claim that there exists an edge  in $T_n$ for which (ii) and (iv) does not hold. This follows from  Lemma \ref{lem:bound} if $|V(T_n)|\geq 6$. If $|V(T_n)|=5$, then $T_n$ is uniquely determined, and since $K$ is an induced subgraph of $G_0$, and $v$ is not contained in a subgraph isomorphic to  $\hat{K_4}$, the existence of such an edge in $T_n$ can easily be verified by inspection.

So let $(rs,1)$ be an edge of $T_n$ for which (ii) and (iv) do not hold. Let $t$ be the third vertex of the unique balanced triangle containing $(rs,1)$.
If there is a subgraph $H_0$ satisfying (i) then we 
contradict the $(2,0)$-sparsity of $G_0$ as follows. Since $t$ has two neighbours (namely $r$ and $s$) in $H_0$, clearly any other neighbour of $t$ in $T_n$ is not in $H_0$. Note that $T_n$ contains a balanced triangle of the form $K_3(r,t,w)$ or of the form $K_3(s,t,w')$ (or both). Suppose that $K_3(r,t,w)$ exists in $T_n$. 
Then we may repeat the above argument for  $w$  and  the subgraph $H_0'$ which is obtained from $H_0$ by adding $t$ and the edges $(tr,1)$ and $(ts,1)$ to see that any neighbour of $w$ in $T_n$ (other than $r$ and $t$) is not in $H_0$. By iterating this argument we conclude that none of the vertices of $V(T_n) \setminus \{r,s\}$ are in $H_0$. It now follows  that $H_0\cup T_n \cup {(vb,1)}$ violates $(2,0)$-sparsity. Thus, (i) also does not hold for $(rs,1)$.

Finally, we show that there is no  balanced subgraph $H_0$ satisfying (iii) for the edge $(rs,1)$. Suppose to the contrary that there does exist such a subgraph $H_0$.
If $t\notin V(H_0)$ then we may apply (a balanced version of) the argument above to show that none of the vertices of $V(T_n) \setminus \{r,s\}$ are in $H_0$. It now follows  that $H_0\cup T_n \cup {(vb,1)}$ violates $(2,2)$-sparsity.
So we may assume that $t\in V(H_0)$. Then, by condition (iii), the edges $(tr,1),(ts,1)$ are not in $H_0$. 

The maximal balanced triangle sequence $T_n$ contains a balanced triangle of the form $K_3(r,t,w)$ or of the form $K_3(s,t,w')$ (or both). As above, we assume that $K_3(r,t,w)$ exists.
Suppose first that $w\in V(H_0)$.
Then the edges $(wt,1)$ and $(wr,1)$ are also both in $H_0$. (If not, say $(wt,1)\notin E(H_0)$, then the graph  obtained from $H_0$ by adding  the edges $(tr,1), (ts,1), (wt,1)$ violates $(2,0)$-sparsity.) Since the edges $(wt,1)$ and $(wr,1)$ form a path of gain $1$ from $r$ to $t$ in $H_0$, and $H_0$ is balanced, every path from $r$ to $t$ in $H_0$ has gain $1$. Thus, if we add the edge $(tr,1)$ to $H_0$, then we obtain a balanced graph that violates the $(2,2)$-sparsity of $G_0$, a contradiction.

 So we may assume that $w\notin V(H_0)$. Let $H_0'$ be the graph obtained from $H_0$ by adding the edges $(tr,1)$ and $(ts,1)$. Then $f(H_0')=0$. Since $w$ has two neighbours (namely $r$ and $t$) in $H_0'$, clearly any other neighbour of $w$ in $T_n$ is not in $H_0'$. Using again the same iteration argument from above, we conclude that none of the vertices of $V(T_n) \setminus \{r,s,t\}$ are in $H_0'$. It now follows  that $H_0'\cup T_n \cup {(vb,1)}$ violates $(2,0)$-sparsity. 

Thus, in Case 2a there exists an admissible reverse move.\\

\textbf{Case 2.b.  $G_0$ contains a degree 3 vertex, and every degree 3 vertex in $G_0$ is contained in an induced subgraph isomorphic to 
$R$, $K_4^+$, $K_4^{++}$ or $\hat{K_4}$.}\\

Every vertex of degree 3 is contained in an induced subgraph $W_i$ isomorphic to $K_4^+$ or in an induced subgraph $Z_j$ isomorphic to either $R$, $\hat{K_4}$, or a graph $K_4^{++}$ from Figures~\ref{fig:gain_cov_graphs} (d)-(g). Since $f(K_4^+)=1$ and $f(R)=f(\hat{K_4})=f(K_4^{++})=0$ we say the former are type 1 and the latter are type 0.
Let $W_1,\dots, W_{r'}$ be all such type 1 induced subgraphs and $Z_1,\dots,Z_{s'}$ be all such type 0 induced subgraphs. 

Note that for all $1\leq i \leq s'$, any proper non-empty subgraph $H$ of $Z_i$ has $f(H)\geq 1$.
Now, for any pair of subgraphs $Z_i,Z_j$ we have that $Z_i$ and $Z_j$ are necessarily vertex disjoint. If not then $f(Z_i\cap Z_j)>0$ and hence $f(Z_i\cup Z_j)<0$, which would contradict the $(2,2,0)$-sparsity of $G_0$. 
Next, for all $1\leq j \leq r'$, any proper non-empty subgraph $Y$ of $W_j$ is either a loop or has $f(Y)\geq 2$. 
Thus, for any pair of subgraphs $Z_i,W_j$ we have that either the intersection is a loop,
 or 
$Z_i$ and $W_j$ are vertex disjoint. If not then $f(Z_i\cap W_j)\geq 2$ and hence $f(Z_i\cup W_j)=1-f(Z_i\cap W_j)<0$, contradicting $(2,2,0)$-sparsity. 
Lastly, for any pair of subgraphs $W_i,W_j$ with non-empty intersection, we must have $f(W_i\cap W_j)\in \{1,2\}$. This implies that $W_i\cap W_j$ is either  empty, a loop, or has $f(W_i\cap W_j)=2$. Moreover in the case when $f(W_i\cap W_j)=2$, then $W_i\cap W_j$ 
is either a double edge or a copy of $K_1$  (since the case when $W_i\cap W_j$ is a copy of $K_4$ would mean either $i=j$ or would contradict the fact that $W_i$ and $W_j$ are induced subgraphs of $G_0$). 

\begin{figure}[ht]
\begin{center}
\begin{tikzpicture}[scale=0.6]

\filldraw (0,0) circle (3pt);
\filldraw (0,2) circle (3pt);
\filldraw (2,0) circle (3pt);
\filldraw (2,2) circle (3pt);

\filldraw (-2,2) circle (3pt);
\filldraw (-2,4) circle (3pt);
\filldraw (0,4) circle (3pt);

\draw[black]
(0,4) -- (-2,4) -- (-2,2) -- (0,2) -- (0,4) -- (-2,2);

\draw[black]
(2,0) -- (0,0) -- (0,2) -- (2,2) -- (2,0) -- (0,2) -- (-2,4);

\draw[black]
(0,0) -- (2,2);

\draw[dotted] plot[smooth, tension=1] coordinates{(-2,2)(-1.4,.8)(0,0)};

\draw[thick] plot[smooth, tension=1] coordinates{(0,2) (.2,2.5)(.5,2.3) (0,2)};
\end{tikzpicture}
\end{center}
\caption{The graph consisting of some $k\geq 2$ copies of $K_4$ which all intersect in a single vertex and that vertex is incident with one loop. This family of graphs gives the only additional isomorphism classes that can occur in the modified list of $W_i$'s. Gain labels omitted.}
\label{fig:new0}
\end{figure}
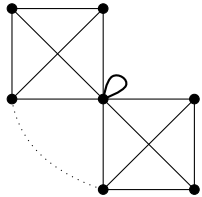

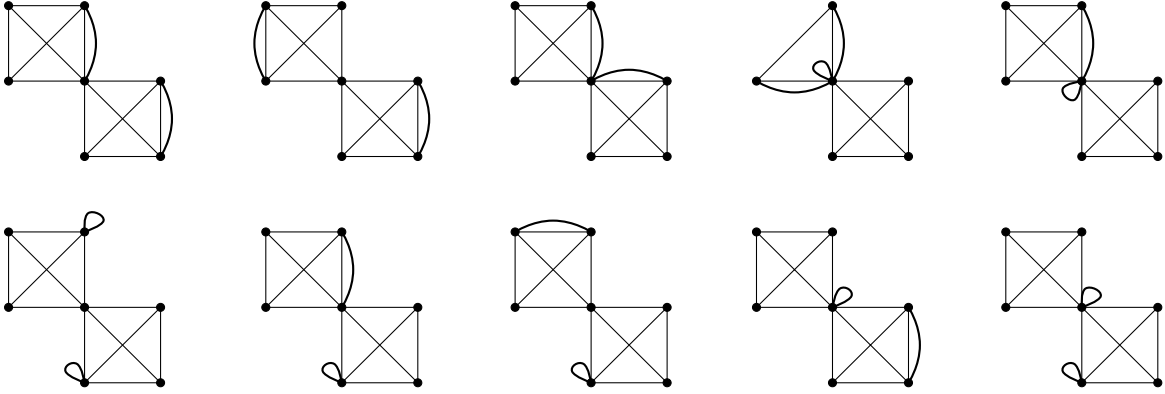
\begin{figure}[ht]
\begin{center}
\begin{tikzpicture}[scale=0.5]

\filldraw (0,0) circle (3pt);
\filldraw (0,2) circle (3pt);
\filldraw (2,0) circle (3pt);
\filldraw (2,2) circle (3pt);

\filldraw (-2,2) circle (3pt);
\filldraw (-2,4) circle (3pt);
\filldraw (0,4) circle (3pt);

\draw[black]
(0,4) -- (-2,4) -- (-2,2) -- (0,2) -- (0,4) -- (-2,2);

\draw[black]
(2,0) -- (0,0) -- (0,2) -- (2,2) -- (2,0) -- (0,2) -- (-2,4);

\draw[black]
(0,0) -- (2,2);

\draw[thick] plot[smooth, tension=1] coordinates{(0,4) (.3,3)(0,2)};
\draw[thick] plot[smooth, tension=1] coordinates{(2,2) (2.3,1)(2,0)};

\filldraw (0,-6) circle (3pt);
\filldraw (0,-4) circle (3pt);
\filldraw (2,-6) circle (3pt);
\filldraw (2,-4) circle (3pt);

\filldraw (-2,-4) circle (3pt);
\filldraw (-2,-2) circle (3pt);
\filldraw (0,-2) circle (3pt);

\draw[black]
(0,-2) -- (-2,-2) -- (-2,-4) -- (0,-4) -- (0,-2) -- (-2,-4);

\draw[black]
(2,-6) -- (0,-6) -- (0,-4) -- (2,-4) -- (2,-6) -- (0,-4) -- (-2,-2);

\draw[black]
(0,-6) -- (2,-4);

\draw[thick] plot[smooth, tension=1] coordinates{(0,-2) (.1,-1.5)(.5,-1.7) (0,-2)};
\draw[thick] plot[smooth, tension=1] coordinates{(0,-6) (-.2,-5.5)(-.5,-5.7) (0,-6)};

\end{tikzpicture}
\hspace{0.8cm}
\begin{tikzpicture}[scale=0.5]

\filldraw (0,0) circle (3pt);
\filldraw (0,2) circle (3pt);
\filldraw (2,0) circle (3pt);
\filldraw (2,2) circle (3pt);

\filldraw (-2,2) circle (3pt);
\filldraw (-2,4) circle (3pt);
\filldraw (0,4) circle (3pt);

\draw[black]
(0,4) -- (-2,4) -- (-2,2) -- (0,2) -- (0,4) -- (-2,2);

\draw[black]
(2,0) -- (0,0) -- (0,2) -- (2,2) -- (2,0) -- (0,2) -- (-2,4);

\draw[black]
(0,0) -- (2,2);

\draw[thick] plot[smooth, tension=1] coordinates{(-2,4) (-2.3,3)(-2,2)};
\draw[thick] plot[smooth, tension=1] coordinates{(2,2) (2.3,1)(2,0)};

\filldraw (0,-6) circle (3pt);
\filldraw (0,-4) circle (3pt);
\filldraw (2,-6) circle (3pt);
\filldraw (2,-4) circle (3pt);

\filldraw (-2,-4) circle (3pt);
\filldraw (-2,-2) circle (3pt);
\filldraw (0,-2) circle (3pt);

\draw[black]
(0,-2) -- (-2,-2) -- (-2,-4) -- (0,-4) -- (0,-2) -- (-2,-4);

\draw[black]
(2,-6) -- (0,-6) -- (0,-4) -- (2,-4) -- (2,-6) -- (0,-4) -- (-2,-2);

\draw[black]
(0,-6) -- (2,-4);

\draw[thick] plot[smooth, tension=1] coordinates{(0,-2) (.3,-3)(0,-4)};

\draw[thick] plot[smooth, tension=1] coordinates{(0,-6) (-.2,-5.5)(-.5,-5.7) (0,-6)};


\end{tikzpicture}
\hspace{0.8cm}
\begin{tikzpicture}[scale=0.5]

\filldraw (0,0) circle (3pt);
\filldraw (0,2) circle (3pt);
\filldraw (2,0) circle (3pt);
\filldraw (2,2) circle (3pt);

\filldraw (-2,2) circle (3pt);
\filldraw (-2,4) circle (3pt);
\filldraw (0,4) circle (3pt);

\draw[black]
(0,4) -- (-2,4) -- (-2,2) -- (0,2) -- (0,4) -- (-2,2);

\draw[black]
(2,0) -- (0,0) -- (0,2) -- (2,2) -- (2,0) -- (0,2) -- (-2,4);

\draw[black]
(0,0) -- (2,2);

\draw[thick] plot[smooth, tension=1] coordinates{(0,4) (.3,3)(0,2)};
\draw[thick] plot[smooth, tension=1] coordinates{(0,2) (1,2.3)(2,2)};

\filldraw (0,-6) circle (3pt);
\filldraw (0,-4) circle (3pt);
\filldraw (2,-6) circle (3pt);
\filldraw (2,-4) circle (3pt);

\filldraw (-2,-4) circle (3pt);
\filldraw (-2,-2) circle (3pt);
\filldraw (0,-2) circle (3pt);

\draw[black]
(0,-2) -- (-2,-2) -- (-2,-4) -- (0,-4) -- (0,-2) -- (-2,-4);

\draw[black]
(2,-6) -- (0,-6) -- (0,-4) -- (2,-4) -- (2,-6) -- (0,-4) -- (-2,-2);

\draw[black]
(0,-6) -- (2,-4);

\draw[thick] plot[smooth, tension=1] coordinates{(-2,-2) (-1,-1.7)(0,-2)};

\draw[thick] plot[smooth, tension=1] coordinates{(0,-6) (-.2,-5.5)(-.5,-5.7) (0,-6)};

\end{tikzpicture}
\hspace{0.8cm}
\begin{tikzpicture}[scale=0.5]

\filldraw (0,0) circle (3pt);
\filldraw (0,2) circle (3pt);
\filldraw (2,0) circle (3pt);
\filldraw (2,2) circle (3pt);

\filldraw (-2,2) circle (3pt);
\filldraw (0,4) circle (3pt);

\draw[black]
(-2,2) -- (0,2) -- (0,4) -- (-2,2);

\draw[black]
(2,0) -- (0,0) -- (0,2) -- (2,2) -- (2,0) -- (0,2);

\draw[black]
(0,0) -- (2,2);

\draw[thick] plot[smooth, tension=1] coordinates{(0,2) (-.2,2.5)(-.5,2.3) (0,2)};
\draw[thick] plot[smooth, tension=1] coordinates{(0,2) (-1,1.7) (-2,2)};
\draw[thick] plot[smooth, tension=1] coordinates{(0,2) (.3,3) (0,4)};

\filldraw (-2,-2) circle (3pt);
\filldraw (0,-2) circle (3pt);
\filldraw (-2,-4) circle (3pt);
\filldraw (2,-6) circle (3pt);
\filldraw (0,-6) circle (3pt);
\filldraw (0,-4) circle (3pt);
\filldraw (2,-4) circle (3pt);

\draw[black]
(-2,-2) -- (0,-2) -- (-2,-4) -- (-2,-2) -- (0,-4) -- (2,-4) -- (0,-6) -- (2,-6) -- (0,-4) -- (-2,-4);

\draw[black]
(0,-2) -- (0,-4) -- (0,-6);

\draw[black]
(2,-4) -- (2,-6);

\draw[thick] plot[smooth, tension=1] coordinates{(2,-4) (2.3,-5) (2,-6)};
\draw[thick] plot[smooth, tension=1] coordinates{(0,-4) (.2,-3.5)(.5,-3.7) (0,-4)};

\end{tikzpicture}
\hspace{0.8cm}
\begin{tikzpicture}[scale=0.5]

\filldraw (0,0) circle (3pt);
\filldraw (0,2) circle (3pt);
\filldraw (2,0) circle (3pt);
\filldraw (2,2) circle (3pt);

\filldraw (-2,2) circle (3pt);
\filldraw (-2,4) circle (3pt);
\filldraw (0,4) circle (3pt);

\draw[black]
(0,4) -- (-2,4) -- (-2,2) -- (0,2) -- (0,4) -- (-2,2);

\draw[black]
(2,0) -- (0,0) -- (0,2) -- (2,2) -- (2,0) -- (0,2) -- (-2,4);

\draw[black]
(0,0) -- (2,2);

\draw[thick] plot[smooth, tension=1] coordinates{(0,4) (.3,3)(0,2)};
\draw[thick] plot[smooth, tension=1] coordinates{(0,2)(-.2,1.5)(-.5,1.8)(0,2)};

\filldraw (0,-6) circle (3pt);
\filldraw (0,-4) circle (3pt);
\filldraw (2,-6) circle (3pt);
\filldraw (2,-4) circle (3pt);

\filldraw (-2,-4) circle (3pt);
\filldraw (-2,-2) circle (3pt);
\filldraw (0,-2) circle (3pt);

\draw[black]
(0,-2) -- (-2,-2) -- (-2,-4) -- (0,-4) -- (0,-2) -- (-2,-4);

\draw[black]
(2,-6) -- (0,-6) -- (0,-4) -- (2,-4) -- (2,-6) -- (0,-4) -- (-2,-2);

\draw[black]
(0,-6) -- (2,-4);

\draw[thick] plot[smooth, tension=1] coordinates{(0,-4) (.1,-3.5)(.5,-3.7) (0,-4)};
\draw[thick] plot[smooth, tension=1] coordinates{(0,-6) (-.2,-5.5)(-.5,-5.7) (0,-6)};

\end{tikzpicture}
\end{center}
\caption{Additional isomorphism classes that can occur in the first step of the modified list of $Z_j$'s. Gain labels omitted.}
\label{fig:new}
\end{figure}

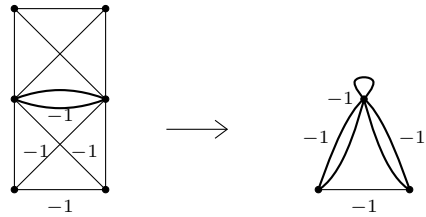
\begin{figure}[ht]
\begin{center}
\begin{tikzpicture}[scale=0.4]

\filldraw (0,0) circle (3pt);
\filldraw (3,0) circle (3pt);
\filldraw (0,3) circle (3pt);
\filldraw (3,3) circle (3pt);
\filldraw (0,-3) circle (3pt);
\filldraw (3,-3) circle (3pt);

\filldraw (1.5,-3) circle (0pt)node[anchor=north]{\tiny $-1$};
\filldraw (1.5,0) circle (0pt)node[anchor=north]{\tiny $-1$};
\filldraw (.7,-1.2) circle (0pt)node[anchor=north]{\tiny $-1$};
\filldraw (2.3,-1.2) circle (0pt)node[anchor=north]{\tiny $-1$};

\draw[black]
(3,-3) -- (0,0) -- (3,3);

\draw[black]
(0,-3) -- (3,0) -- (3,3) -- (0,3) -- (3,0) -- (3,-3) -- (0,-3) -- (0,0) -- (0,3);

\draw[thick] plot[smooth, tension=1] coordinates{(0,0) (1.5,.3)(3,0)};
\draw[thick] plot[smooth, tension=1] coordinates{(0,0) (1.5,-.3)(3,0)};

\draw[black]
(5,-1) -- (7,-1);

\draw[black]
(6.6,-1.3) -- (7,-1) -- (6.6,-.7);

\filldraw (11.5,0) circle (3pt);
\filldraw (10,-3) circle (3pt);
\filldraw (13,-3) circle (3pt);

\draw[black]
(10,-3) -- (13,-3);

\draw[thick] plot[smooth, tension=1] coordinates{(10,-3) (10.8,-1)(11.5,0)};
\draw[thick] plot[smooth, tension=1] coordinates{(10,-3) (10.8,-2)(11.5,0)};

\draw[thick] plot[smooth, tension=1] coordinates{(13,-3) (12.2,-1)(11.5,0)};
\draw[thick] plot[smooth, tension=1] coordinates{(13,-3) (12.2,-2)(11.5,0)};

\draw[thick] plot[smooth, tension=1] coordinates{(11.5,0)(11.2,.6)(11.8,.6)(11.5,0)};

\filldraw (11.5,0) circle (0pt)node[anchor=east]{\tiny $-1$};
\filldraw (11.5,-3) circle (0pt)node[anchor=north]{\tiny $-1$};
\filldraw (10.7,-1.3) circle (0pt)node[anchor=east]{\tiny $-1$};
\filldraw (12.3,-1.3) circle (0pt)node[anchor=west]{\tiny $-1$};

\end{tikzpicture}
\end{center}
\caption{An additional isomorphism class that can occur in the first step of the modified list of $Z_j$'s and the result, $R$, of a reverse vertex-to-$K_4$ move applied to this graph. Omitted gain labels are equal to $1$.}
\label{fig:newextra}
\end{figure}

We next modify our lists $W_1,\dots, W_r$ and $Z_1,\dots,Z_s$ until any pair $W_i,W_j$ and any pair $Z_i,W_j$ are vertex disjoint. This means that whenever $W_i$ and $W_j$ intersect in a loop, then we discard them and add the union of $W_i$ and $W_j$ as a new $W_\ell$, and whenever $W_i$ and $W_j$ intersect in a double edge or a copy of $K_1$, then we also discard them and add the union of $W_i$ and $W_j$ as a new $Z_m$. Moreover, whenever $Z_i$ and $W_j$ intersect in a loop then we discard them and add the union of $Z_i$ and $W_j$ as a new $Z_k$. This process is then iterated until the final list of all $W_i$'s and $Z_i$'s has the property that any two elements in the list are vertex disjoint. (The additional $W_i$'s at any step of the process and the additional $Z_j$'s from step 1 of this process are as depicted in Figures \ref{fig:new0} and Figures \ref{fig:new} and \ref{fig:newextra}.)

Let $U$ and $F$ be the sets of vertices and edges of $G_0$ which are in none of the $W_i$ and in none of the $Z_j$. Associate with $G_0$ an auxiliary (multi)graph $G_0^\ast$ which has a vertex for each $W_i$, a vertex for each $Z_i$ and a vertex for $U$ and has an edge corresponding to each edge of $G_0$ of the form $x_ix_j$, where $x_i,x_j$ are taken from distinct elements of $V(G_0^\ast)=\{W_1,\dots,W_r,Z_1,\dots, Z_s,U\}$. Also define $G_0^-$ to be the  simple graph which is obtained from  $G_0^\ast$ by removing any parallel edges.

The connectivity of $G_0$ implies that $G_0^-$ is connected. 
Suppose $|V(G_0^-)|=1$. Then $G_0$ is a copy of some $Z_j$. If $G_0=\hat{K_4}$ then we may apply a reverse H3d move. 
If $G_0$ is a union of two $K_4^+$ graphs which intersect in a double edge then we may apply a reverse vertex-to-$K_4$ move (see Figure \ref{fig:newextra}). If $G_0$ is not a base graph then by the iteration process above, it follows in all remaining cases that $G_0$ contains a cut-vertex that separates an induced subgraph isomorphic to $K_4^+$ from the rest of the graph. It is now clear that we may use Lemma \ref{lem:k4contract} to apply a reverse vertex-to-$K_4$ move to a $K_4^+$ subgraph of $G_0$. 
Thus, we may suppose that $|V(G_0^-)|>1$. 

Suppose $r=0$. 
Let $G_0[U]$ denote the subgraph of $G_0$ induced by $U$.
Note that since $f(Z_i)=0$ for each $i$, no two $Z_i$ can be adjacent (by the $(2,2,0)$-gain-sparsity of $G_0$). Since $G_0^-$ is connected, it follows that $G_0^-$ is the graph $K_{1,s}$ where $s\geq 1$.
Moreover, we have \[f(G_0[U]) - d(U,V_0-U)=\sum_{i=1}^s f(Z_i) +f(G_0[U]) - d(U,V_0-U) = f(G_0) = 0.\] 
Since every vertex in $U$ has degree at least $4$ in $G_0$, by Lemma \ref{l:claim} we have, 
\[d(U,V_0-U)\geq 2f(G_0[U])=2d(U,V_0-U).\] Thus $d(U,V_0-U)=0$, a contradiction. 

Now suppose $r>0$. 
Recall that each $W_i$ is a $K_4^+$ or is of the form illustrated in Figure \ref{fig:new0}.
Hence, if any $W_i$ is not incident to two parallel edges in $G_0^\ast$ then we may contract a copy of $K_4^+$ to a loop by Lemma \ref{lem:k4contract}. So we suppose that every $W_i$ is incident to two parallel edges in $G_0^\ast$.

We calculate
\begin{equation}\label{eqn:2ci}
f(G_0)=\sum_{i=1}^rf(W_i) + \sum_{j=1}^sf(Z_j)+2|U|-|F|,
\end{equation}
which implies that $|F|=2|U|+r.$ 

Suppose first that  every $W_i$ and every $Z_j$ is incident to at least two edges in $F$. Since each vertex in $U$ has degree at least $4$, there are at least $4|U|+2(r+s)$ edge/vertex incidences in $F$. This implies $|F|\geq 2|U|+r+s$, 
and hence $s=0$. By a similar counting argument, if some $W_i$ is incident to more than two edges in $F$ then there are at least $4|U|+2(r-1)+3$ edge/vertex incidences in $F$.
This implies $|F|> 2|U|+r$, a contradiction.
Thus each $W_i$ has degree exactly 2 in $G_0^\ast$. This implies that either $G_0$ is the disjoint union of $W_1$ and $W_2$ with two edges between them, or, $G_0^-$ is the graph $K_{1,r}$. 
In the former case, there is an admissible reverse vertex-to-$K_4$ move which contracts a copy of $K_4^+$ to a loop. So suppose  $G_0^-$ is $K_{1,r}$. In this case, every $v\in U$ has degree exactly 4.  
We may assume that every $W_i$ is a copy of $K_4^+$ and that for every $W_i$ there exists a vertex in $U$ that is joined to two vertices in $W_i$ by edges with identical gains. (Otherwise, there is an admissible reverse vertex-to-$K_4$ move.)

Let $v\in U$ be adjacent to two vertices in some $W_i$. Then there is balanced copy of $K_3$ containing $u$ and two vertices $a,b$ of $W_i$. We may assume the gains on this copy of $K_3$ are all 1. We now apply Lemma \ref{lem:k3contract} to show that there is an  admissible reverse vertex-splitting move that contracts either $(va,1)$ or $(vb,1)$. Since $W_i$ is a copy of $K_4^+$ at most one of $a,b$ can have a loop. We suppose there is no loop on $a$ and consider the contraction of $(va,1)$. It  follows from Lemma \ref{lem:k3contract} that if $(va,1)$ is non-admissible then there is either a subgraph $H_0$ containing $v,a$ and the edge $(va,1)$ but not $b$ with $f(H_0)=0$ or a balanced subgraph $H_0$ containing $v,a$ and the edge $(va,1)$ with $f(H_0)=2$ and if $H_0$ contains $b$ then it does not contain the edges $(ba,1)$ and $(vb,1)$. In both cases it is easy to deduce from the structure of $G_0$ that such an $H_0$ cannot exist.

So we suppose that there is a $Z_j$ which is incident to only one edge in $F$.
If a $W_i$ is joined to another $W_k$ by two edges, then we discard them and add the union of $W_i$ and $W_k$ together with the two extra edges as a new $Z_m$. Similarly, if a $W_i$ is joined to a $Z_k$ by one edge, then we discard them and add the union of $W_i$ and $Z_k$ together with the extra edge as a new $Z_\ell$. 
Then the graph $G_0^-$ corresponding to this new underlying structure of $G_0$ contains the graph $K_{1,t}$ as a spanning subgraph, where the vertices in the partite set of size $t=b+c$ correspond to the graphs  $W_1,\dots, W_b$ and $Z_1,\dots,Z_c$, and the vertex in the other partite set corresponds to $U$. If $G_0^-$ is not equal to $K_{1,t}$, then the additional edges must join vertices corresponding to the $W_i$, and these edges represent single edges among pairs of $W_i$ in $G_0^\ast$. We call this set of edges $A$.  Note that in $G_0^\ast$, for each $W_i$ there are two  parallel edges joining $W_i$ to $U$, and for each $Z_j$ there is  one edge joining $Z_j$ to $U$. There may be additional edges joining a $W_i$ with $U$ or a $Z_j$ with $U$, and we denote this set of additional edges by $B$. Hence we have  $f(G_0[U])= b+c+d$, where $d=|A|+|B|$. Thus, by Lemma \ref{l:claim}, we have $d(U,V_0-U)\geq 2(b+c+d)$. This is a contradiction since there are only exactly $2b+c+d-|A|$ edges incident to $U$. 

Thus, in Case 2b there exists an admissible reverse move.

Henceforth we may assume that $G_0$ is 4-regular. First let us deal with two special possible subgraphs of $G_0$.\\

\textbf{Case 3.a. There exists a balanced subgraph isomorphic to either a copy of $K_{1,1,3}$, or, to a copy of $K_4$ which neither induces additional edges nor is contained in a copy of $K_5-e$.}\\

Suppose that $G_0$ contains an induced balanced copy of $K_4$. By switching we may assume that all edges of the $K_4$ have gain $1$.  By Lemma \ref{lem:k4contract} and the assumptions of this case, if a reverse vertex-to-$K_4$ move is not admissible then there is a vertex $s$ adjacent to exactly two of the $K_4$ vertices such that $K_4\cup s$ is balanced.
Consider the vertex $s$ in $K_4\cup s$. Since $G_0$ is 4-regular  and $K_4$ is not contained in $K_5-e$, $s$ has either a neighbour not in the $K_4$ with two parallel edges to $s$, or two neighbours not in the $K_4$, or a loop on $s$. (These three possibilities are illustrated in Figure \ref{fig:n3a}(a).) In the first case we can use Lemma~\ref{lem:4threei} to conclude that $s$ is admissible and in the second case we can use Lemma~\ref{lem:4threeii} to conclude that $s$ is admissible. In the third case we may use Lemma \ref{lem:k3contract}. Let $r,t$ be the neighbours of $s$ in $G_0$. By switching, we may assume that $(sr,1)$ and $(st,1)$ are edges in $G_0$. Consider the contraction of the edge $(sr,1)$. Since $G_0$ is 4-regular, conditions (ii) and (iv) of Lemma~\ref{lem:k3contract} evidently fail. Condition (i) fails by Lemma \ref{lem:4regno0}. Finally, since $G_0$ is $4$-regular and any balanced subgraph $H_0$ with $f(H_0)=2$ cannot have a vertex of degree $1$, condition (iii) also fails.

\begin{figure}[ht]
\begin{center}
\begin{tikzpicture}[scale=0.7]

\filldraw (0,0) circle (3pt);
\filldraw (0,2) circle (3pt);
\filldraw (2,0) circle (3pt);
\filldraw (2,2) circle (3pt);
\filldraw (1,-1.5) circle (3pt) node[anchor=east]{$s$};
\filldraw (1,-3) circle (3pt);

\draw[black]
(0,2) -- (2,0) -- (1,-1.5) -- (0,0) -- (2,2) -- (2,0) -- (0,0) -- (0,2) -- (2,2);

\draw[thick] plot[smooth, tension=1] coordinates{(1,-3) (.7,-2.25) (1,-1.5)};
\draw[thick] plot[smooth, tension=1] coordinates{(1,-3) (1.3,-2.25) (1,-1.5)};

\filldraw (1.5,-4.5) circle (0pt);
\end{tikzpicture}
\hspace{1.1cm}
\begin{tikzpicture}[scale=0.7]

\filldraw (0,0) circle (3pt);
\filldraw (0,2) circle (3pt);
\filldraw (2,0) circle (3pt);
\filldraw (2,2) circle (3pt);
\filldraw (1,-1.5) circle (3pt) node[anchor=east]{$s$};

\filldraw (0,-3) circle (3pt);
\filldraw (2,-3) circle (3pt);

\draw[black]
(0,2) -- (2,0) -- (1,-1.5) -- (0,0) -- (2,2) -- (2,0) -- (0,0) -- (0,2) -- (2,2);

\draw[black]
(0,-3) -- (1,-1.5) -- (2,-3);

\filldraw (1.5,-4) circle (0pt)node[anchor=east]{(a)};

\end{tikzpicture}
\hspace{1.1cm}
\begin{tikzpicture}[scale=0.7]

\filldraw (0,0) circle (3pt);
\filldraw (0,2) circle (3pt);
\filldraw (2,0) circle (3pt);
\filldraw (2,2) circle (3pt);
\filldraw (1,-1) circle (3pt) node[anchor=east]{$s$};

\draw[black]
(0,2) -- (2,0) -- (1,-1) -- (0,0) -- (2,2) -- (2,0) -- (0,0) -- (0,2) -- (2,2);

\draw[thick] plot[smooth, tension=1] coordinates{(1,-1) (.8,-1.3)(1.2,-1.3) (1,-1)};

\filldraw (1.5,-4.5) circle (0pt);

\end{tikzpicture}
\hspace{1.1cm}
\begin{tikzpicture}[scale=0.7]

\filldraw (0,2) circle (3pt);
\filldraw (2,2) circle (3pt);

\filldraw (-1,0) circle (3pt);
\filldraw (1,0) circle (3pt);
\filldraw (3,0) circle (3pt);

\filldraw (-2,-2) circle (3pt);
\filldraw (0,-2) circle (3pt);
\filldraw (2,-2) circle (3pt);

\draw[black]
(0,2) -- (3,0) -- (2,2) -- (1,0) -- (0,2) -- (2,2) -- (-1,0) -- (0,2);

\draw[black]
(0,-2) -- (1,0) -- (2,-2);

\draw[thick] plot[smooth, tension=1] coordinates{(-1,0)(-1.8,-1)(-2,-2)};
\draw[thick] plot[smooth, tension=1] coordinates{(-1,0)(-1.2,-1)(-2,-2)};

\draw[thick] plot[smooth, tension=1] coordinates{(3,0)(2.8,-.3)(3.2,-.3)(3,0)};

\filldraw (1.5,-4) circle (0pt)node[anchor=east]{(b)};
\end{tikzpicture}
\end{center}
\caption{(a) The possibilities in Case 3.a when $G_0$ contains an induced balanced copy of $K_4$ and (b) the possibilities when $G_0$ contains an induced balanced copy of $K_{1,1,3}$.}
\label{fig:n3a}
\end{figure}
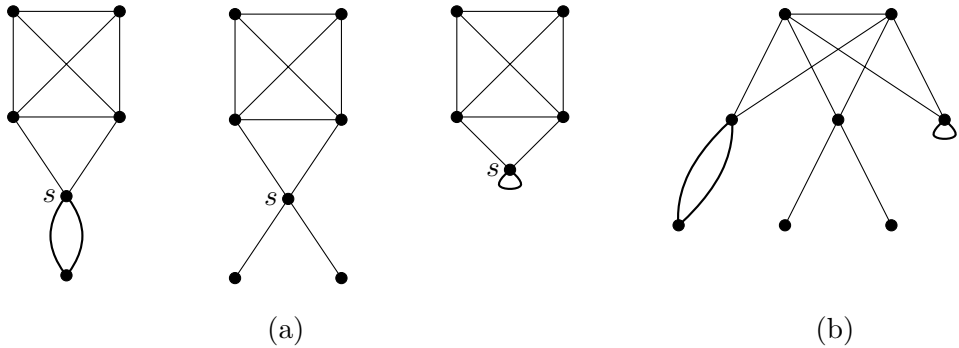

Suppose then that 
$G_0$ contains a  balanced copy of $K_{1,1,3}$ and does not contain an induced balanced copy of $K_4$. By switching, we may assume that all edges of $K_{1,1,3}$ have gain $1$. Then each degree 2 vertex in the $K_{1,1,3}$ is either incident to a double edge joining it to a third vertex, 
or to two single edges joining it to a third and fourth vertex, 
or to a loop. (These three possibilities are illustrated in Figure \ref{fig:n3a}(b).)  If there is a vertex of the first or the second type, then we may apply Lemma~\ref{lem:4threei} or
Lemma~\ref{lem:4threeii}, respectively, to show that there exists an admissible reverse move.  So suppose
each of the degree 2 vertices of $K_{1,1,3}$ is incident to a loop.
 Since $K_{1,1,3}$ with three loops is 4-regular, it must be equal to $G_0$. We show there is a contraction using Lemma \ref{lem:k3contract}. Let $r$ and $s$ be the vertices corresponding to the two partite sets of $K_{1,1,3}$ of size 1, and let $t$ be a vertex incident to a loop. We will contract the edge $(rt,1)$. Lemma \ref{lem:4regno0} implies that (i) fails and again (ii) and (iv) clearly fail. To see that (iii) fails note simply that every subgraph $H$ with no loops satisfies $f(H)\geq 3$.\\

\textbf{Case 3.b. There exists a vertex not contained in a balanced subgraph isomorphic to  $K_4$ or to $K_{1,1,3}$.}\\

We consider each of the possibilities for the neighbourhood of a given vertex $v$ in $G_0$ satisfying the hypotheses of this case. 

Suppose first that $v$ has exactly one neighbour. Since $G_0\neq 2K_2^1$ and $G_0$ is connected, Lemma \ref{lem:4one} implies that $v$ is admissible. Suppose next that $v$ has either exactly three or exactly four neighbours.
Then Lemmas \ref{lem:4threei} and \ref{lem:4threeii} imply that $v$ is admissible. 

So we suppose that $v$ has exactly two neighbours. Suppose first that there is no loop on $v$. Let $N(v)=\{x,y\}$. Lemma \ref{lem:4twob} implies that if $v$ is not admissible then there is a loop at $x$ or at $y$, say at $x$. Since $G_0$ is 4-regular we see that $x$ satisfies the condition of Lemma~\ref{lem:4one} and hence there is an admissible reverse move.

So suppose that $v$ has exactly two neighbours $x,y$ and a loop on $v$. By switching, we may assume that $(vx,1)$ and $(vy,1)$ are the non-loop edges incident to $v$. We consider a possible H2d-reduction at $v$.
If every vertex of $G_0$ is incident with a loop and has two neighbours then it is easy to see that $G_0$ is precisely a cycle with one loop on each vertex. If the cycle has length 3 then $G_0$ is a base graph. If the cycle has length at least 4 then it is easy to see that there is no balanced subgraph $H_0$ of $G_0-v$ containing $x,y$ with $f(H_0)=2$ such that all walks from $x$ to $y$ have gain $1$. 
Hence $v$ is admissible.

So we may suppose there is some vertex in $G_0$ which is not incident to a loop. By 4-regularity of $G_0$ and by relabelling if necessary, we may suppose that $x$ is such a vertex. Since $x$ is adjacent to $v$ with exactly one edge and $x$ does not have a loop, $x$ may have either 3 or 4 neighbours.

 If $x$ has 3 neighbours then 4-regularity implies that $x$ is not contained in a $K_4$. Hence Lemma \ref{lem:4threei} implies that $x$, or one of its neighbours, is admissible. 
 So we suppose that $x$ has 4 distinct neighbours.
If the edge $(xy,1)$ exists, then we may contract the edge $vx$ using Lemma \ref{lem:k3contract}. (The structure of $G_0$ makes all four conditions in that lemma easy to rule out.) So suppose $(xy,1)$ does not exist. If $(xy,-1)$ exists then we claim that $x$ is admissible.  If not, then it follows from Lemma~\ref{lem:4threeii} that $x$ must be in a balanced $K_4$. This $K_4$ must consist of the vertices $x,y,a,b$. By switching we may assume that the three edges joining $x$ with $a$, $b$ and $y$ all have gain $-1$ and that the remaining edges of the $K_4$ have gain $1$. Note that every path  from $x$ to $y$ within the $K_4$ has gain $-1$. This implies that there is no balanced subgraph $H_{xy}$ of $G_0-v$ containing $x$ and $y$ with $f(H_{xy})=2$ such that each path from $x$ to $y$ in $H_{xy}$ has gain $1$. Since there is clearly also no subgraph $H_0$ of $G_0-v$ containing $x$ and $y$ with $f(H_0)=0$, this implies that there is a H2d-reduction at $v$, contradicting our assumption that $v$ is non-admissible. 

So we may suppose that $x$ is not adjacent to $y$. By Lemma~\ref{lem:4threeii} $x$ is either in a balanced $K_{1,1,3}$, which gives an immediate contradiction, or $x$ is in a balanced $K_4$. In the latter case, since we are not in Case 3.a, either  the $K_4$ induces exactly one additional edge  and we can apply a reverse vertex-to-$K_4$ move by Lemma \ref{lem:k4contract} (since the additional edge cannot be a loop, by $4$-regularity of $G_0$), or the $K_4$ is contained in a copy of $K_5-e$. 

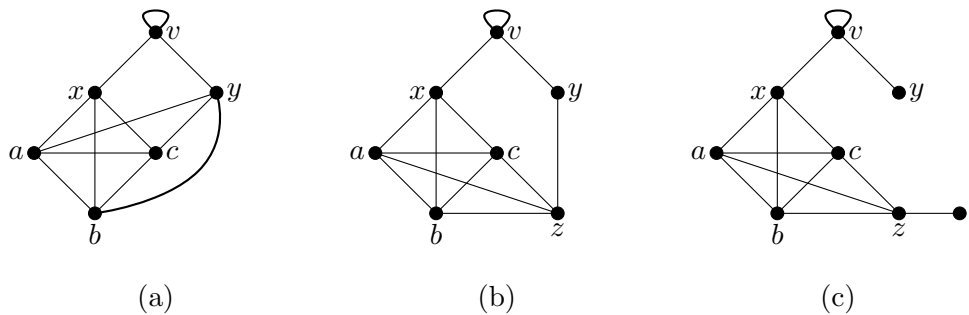
\begin{figure}[ht]
\begin{center}
\begin{tikzpicture}[scale=0.8]

\filldraw (0,6) circle (3pt)node[anchor=west]{$v$};
\filldraw (-1,5) circle (3pt)node[anchor=east]{$x$};
\filldraw (1,5) circle (3pt)node[anchor=west]{$y$};

\filldraw (-2,4) circle (3pt)node[anchor=east]{$a$};
\filldraw (-1,3) circle (3pt)node[anchor=north]{$b$};
\filldraw (0,4) circle (3pt)node[anchor=west]{$c$};

\draw[black]
(1,5) -- (0,6) -- (-1,5) -- (-2,4) -- (-1,3) -- (-1,5) -- (0,4) -- (-1,3);

\draw[black]
(1,5) -- (-2,4) -- (0,4) -- (1,5);

\draw[thick] plot[smooth, tension=1] coordinates{(-1,3)(.7,3.7)(1,5)};

\draw[thick] plot[smooth, tension=1] coordinates{(0,6)(-.2,6.3)(.2,6.3)(0,6)};

\filldraw (0,2) circle (0pt)node[anchor=north]{(a)};

\end{tikzpicture}
\hspace{0.9cm}
\begin{tikzpicture}[scale=0.8]

\filldraw (0,6) circle (3pt)node[anchor=west]{$v$};
\filldraw (-1,5) circle (3pt)node[anchor=east]{$x$};
\filldraw (1,5) circle (3pt)node[anchor=west]{$y$};

\filldraw (-2,4) circle (3pt)node[anchor=east]{$a$};
\filldraw (-1,3) circle (3pt)node[anchor=north]{$b$};
\filldraw (0,4) circle (3pt)node[anchor=west]{$c$};

\filldraw (1,3) circle (3pt)node[anchor=north]{$z$};

\draw[black]
(-1,3) -- (1,3) -- (1,5) -- (0,6) -- (-1,5) -- (-2,4) -- (-1,3) -- (-1,5) -- (0,4) -- (-1,3);

\draw[black]
(1,3) -- (-2,4) -- (0,4) -- (1,3);

\draw[thick] plot[smooth, tension=1] coordinates{(0,6)(-.2,6.3)(.2,6.3)(0,6)};

\filldraw (0,2) circle (0pt)node[anchor=north]{(b)};
\end{tikzpicture}
\hspace{0.9cm}
\begin{tikzpicture}[scale=0.8]

\filldraw (0,6) circle (3pt)node[anchor=west]{$v$};
\filldraw (-1,5) circle (3pt)node[anchor=east]{$x$};
\filldraw (1,5) circle (3pt)node[anchor=west]{$y$};

\filldraw (-2,4) circle (3pt)node[anchor=east]{$a$};
\filldraw (-1,3) circle (3pt)node[anchor=north]{$b$};
\filldraw (0,4) circle (3pt)node[anchor=west]{$c$};

\filldraw (1,3) circle (3pt)node[anchor=north]{$z$};
\filldraw (2,3) circle (3pt)node[anchor=west]{};

\draw[black]
(1,5) -- (0,6) -- (-1,5) -- (-2,4) -- (-1,3) -- (-1,5) -- (0,4) -- (-1,3) -- (1,3);

\draw[black]
(1,3) -- (-2,4) -- (0,4) -- (1,3) -- (2,3);

\draw[thick] plot[smooth, tension=1] coordinates{(0,6)(-.2,6.3)(.2,6.3)(0,6)};

\filldraw (0,2) circle (0pt)node[anchor=north]{(c)};
\end{tikzpicture}
\end{center}
\caption{Possible structures of $G$ when $v$ has two neighbours and a loop and one of its neighbours has 4 distinct neighbours.}
\label{fig:new1}
\end{figure}

Let the vertices of the $K_4$ be $a,b,c,x$ and suppose all edges of the $K_4$ have gain $1$. If the final vertex of the $K_5-e$ is $y$ then $V(G_0)=\{v,x,y,a,b,c\}$ (see Figure \ref{fig:new1}(a)) and we claim that $y$ is admissible.
To see this, note that exactly one or exactly two of the edges joining $y$ with $a,b,c$ have gain $-1$, for otherwise $K_4$ together with the three edges incident with $y$ would violate $(2,2)$-sparsity. Hence in both cases, $y$ is admissible by Lemma \ref{lem:4threeii}.
 
 So we may assume that the final vertex of the $K_5-e$ is $z\neq y$. Suppose first that $z$ is adjacent to $y$ (see Figure \ref{fig:new1}(b)). Then we claim that $y$ is still admissible. 
Clearly $y$ is adjacent to $v$ and $z$. If $y$ is  also incident to a loop, then $V(G_0)=\{v,x,y,a,b,c,z\}$ and it is easy to see that a H2d-reduction is possible at $y$. If $y$ is incident to two parallel edges joining it with a vertex $u\neq v,z$ then $y$, or one of its neighbours, is admissible  by Lemma~\ref{lem:4threei} and the $4$-regularity of $G_0$. If $y$ is incident to two single edges joining it with vertices $u$ and $w$ that are distinct from $v$ and $z$, then $y$ is admissible  by Lemma~\ref{lem:4threeii} and the $4$-regularity of $G_0$. So we may suppose that $z$ is not adjacent to $y$ (see Figure \ref{fig:new1}(c)). Then $z$ is admissible by 
Lemma~\ref{lem:4threeii} and the 4-regularity of $G_0$.\\

\textbf{Case 3.c. Every vertex is contained in a balanced subgraph isomorphic to $K_4$ plus either one or two additional edges or $K_5-e$.}\\

If $G_0$ contains a $K_4^{++}$ as a subgraph then $G_0=K_4^{++}$, by $4$-regularity of $G_0$. In this case $G_0$ is the base graph depicted in Figure \ref{fig:gain_cov_graphs}(h). So we may assume that every vertex of $G_0$ is contained in a balanced subgraph isomorphic to $K_4^+$ (which cannot have a loop) or $K_5-e$. Since $G_0$ is 4-regular any pair of copies of $K_4^+$ or $K_5-e$ are vertex disjoint, and each copy has exactly two edges incident to it. If $G_0$ contains a copy of $K_4^+$ then we can apply a reverse vertex-to-$K_4$ move by Lemma \ref{lem:k4contract}. (Note that since every vertex is in a $K_4^+$ or a $K_5-e$, and $G_0$ is 4-regular, there cannot exist a vertex outside of the $K_4^+$ that is adjacent to two of the vertices of the $K_4^+$.) 
Hence we may suppose that $G_0$ contains no copies of $K_4^+$.
Thus we may assume that we have a copy of $K_5-e$. Suppose $e=xy$ and note that if  $x$ (resp. $y$) is not admissible, then by Lemma \ref{lem:4threeii}  $x$ (resp. $y$) must be contained in a balanced $K_4$. However, if $x$ and $y$ are both contained in balanced $K_4$'s, then the $K_5-e$ is the union of two balanced subgraphs and is hence balanced by Lemma \ref{lem:balancedunion}, contradicting $(2,2)$-sparsity.
\end{proof}

\section{Application to $\C_2$-symmetric frameworks in the $\ell^1$ and $\ell^\infty$-plane}
\label{Sect:Gridlike}
Let $\|\cdot\|_\P$ be a norm on $\bR^2$ with the property that the closed unit ball $\P=\{x\in \bR^2: \|x\|_\P\leq 1\}$ is a quadrilateral (eg. the $\ell^1$ or $\ell^\infty$ norms).
We refer to bar-joint frameworks in this context as {\em grid-like}.
In this section, the results of the previous sections are combined to obtain geometric  and combinatorial characterisations of $\chi$-symmetric isostaticity and infinitesimal rigidity for $\C_2$-symmetric grid-like frameworks. 

\subsection{Framework colourings}
\label{Sect:FrameworkColours}
Let $(G,p)$ be a well-positioned grid-like bar-joint framework  and let $F\in\{\pm F_1,\pm F_2\}$ be one of the four facets of the quadrilateral $\P$. An edge $vw\in E$ is said to have {\em framework colour} $F$ (equivalently, $-F$) if either $p_v-p_w$ or $p_w-p_v$ lies in the cone $\{x\in\bR^2:\frac{x}{\|x\|_\P}\in F\}$. Recall that, since $(G,p)$ is well-positioned, each edge of $G$ has exactly one framework colour (see \cite{kit-pow}). Denote by $G_{F}$ the {\em monochrome subgraph} of $G$ spanned by edges with framework colour $F$.   

For each facet $F$ there exists a unique extreme point $\hat{F}$ of the polar set $\P^{\triangle}=\{y\in\bR^2:x\cdot y\leq 1,\,\,\,\forall\, x\in \P\}$ such that $F=\{x\in\P:x\cdot \hat{F}=1\}$. Define a linear functional $\varphi_F:\bR^2\to\bR$ by setting $\varphi(x)=x\cdot \hat{F}$, for all $x\in \bR^2$.
If $(G,p)$ is well-positioned and $vw\in G_F$ then it can be shown (see \cite{kit-pow}) that the linear functional $\varphi_{v,w}$ described in Lemma \ref{lem:differential} satisfies $\varphi_{v,w}= \varphi_F$. 

If $\G=(G,p,\theta,\tau)$ is a $\C_2$-symmetric grid-like bar-joint framework, then each edge $e\in E$ shares the same framework colour as its image $-e$. By assigning this common framework colour to the edge orbit $[e]=\{e,-e\}$ we induce a framework colouring on the edges of the quotient graph $G_0$. Denote by $G_{F,0}$ the {\em monochrome subgraph} of $G_0$ spanned by edges $[e]$ with framework colour $F$. 

\begin{example}
Consider the $\ell^\infty$ plane. The unit ball  $\P=\{x\in\bR^2:\|x\|_\infty\leq 1\}$ has four facets: $F_1=\{(x_1,x_2)\in\P:x_1=1\}$, $F_2=\{(x_1,x_2)\in\P:x_2=1\}$ and their negatives. The polar set of $\P$ is the $\ell^1$ unit ball $\P^\triangle=\{x\in\bR^2:\|x\|_1\leq 1\}$, and the extreme points of the polar set are $\hat{F}_1=(1,0)$, $\hat{F}_2=(0,1)$ and their negatives.
Figure \ref{fig:gain_cov_graphs} illustrates several examples of framework colourings for $\C_2$-symmetric bar-joint frameworks in the $\ell^\infty$-plane together with the induced framework colourings on their $\bZ_2$-gain graphs.
\end{example}

A {\em map graph} is a graph in which every connected component contains exactly one cycle. An \emph{unbalanced map graph} is a $\mathbb{Z}_2$-gain graph $(H,\psi)$ such that $H$ is a map graph, the covering graph is simple and every cycle is unbalanced.

\begin{theorem}
\label{thm:C2gridgeom}
Let $\G=(G,p,\theta,\tau)$ be a well-positioned and $\C_2$-symmetric bar-joint framework in $(\bR^2,\|\cdot\|_\P)$.
\begin{enumerate}[(A)]
\item
The following statements are equivalent.
\begin{enumerate}[(i)]
\item $\G$ is $\chi_0$-symmetrically isostatic.
\item $G_{F_1,0}$ and $G_{F_2,0}$ are edge-disjoint spanning unbalanced map graphs in $G_0$.
\end{enumerate}

\item
The following statements are equivalent.
\begin{enumerate}[(i)]
\item $\G$ is  $\chi_1$-symmetrically isostatic.
\item $G_{F_1,0}$ and $G_{F_2,0}$ are edge-disjoint spanning trees in $G_0$.
\end{enumerate}

\item
The following statements are equivalent.
\begin{enumerate}[(i)]
\item $\G$ is infinitesimally rigid.
\item $G_{F_1,0}$ and $G_{F_2,0}$ both contain connected spanning unbalanced map graphs.
\end{enumerate}
\end{enumerate}
\end{theorem}

\begin{proof}
$(A)$
$(i)\Rightarrow (ii)$
Suppose there exists a vertex $[v_0]\in V_0\setminus V(G_{F_1,0})$. 
Let $\tilde{v}_0$ be the representative vertex for $[v_0]$ in $G$.
Choose a non-zero vector $x\in \ker\varphi_{F_2}$ and for all $v\in V(G)$ define,
\[u_{v} = \left\{ \begin{array}{ll}
x & \mbox{ if }  v=\tilde{v}_0, \\
-x & \mbox{ if }  v=-\tilde{v}_0, \\
0 & \mbox{ otherwise. } \\
\end{array}\right.\]
Then $u$ is a non-trivial $\chi_0$-symmetric infinitesimal flex for $(G,p)$.
This is a contradiction since $(G,p)$ is $\chi_0$-symmetrically isostatic. Thus every vertex of $G_0$ must be incident to an edge of $G_{F_1,0}$.  By a similar argument every vertex of $G_0$ must  be incident to an edge of $G_{F_2,0}$.

Suppose $G_{F_1,0}$ has a connected component $H_0$ which is a balanced subgraph of $G_0$.
Then, by Lemma \ref{switch}, we may assume that each edge of $H_0$ has gain $1$. Thus if $H$ is the covering graph for $H_0$, then there is no edge $vw\in E(H)$ with $v\in \tilde{V}_0$ and $w\notin \tilde{V}_0$. (Recall Section~\ref{subsec:gaingraphs} for the definition of $\tilde{V}_0$). Choose a non-zero vector $x\in \ker\varphi_{F_2}$ and for all $v\in V(G)$ define,
\[u_{v} = \left\{ \begin{array}{ll}
x & \mbox{ if } [v]\in V(H_0) \mbox{ and } v\in\tilde{V}_0, \\
-x & \mbox{ if } [v]\in V(H_0) \mbox{ and } v\notin\tilde{V}_0, \\
0 & \mbox{ otherwise. }
\end{array}\right.
\]
Then $u$ is a non-trivial $\chi_0$-symmetric infinitesimal flex for $(G,p)$. This is a contradiction and so every connected component of $G_{F_1,0}$ must be an unbalanced subgraph of $G_0$.
Similarly, each connected component of $G_{F_2,0}$ is an unbalanced  subgraph of $G_0$.

By Corollary \ref{cor:neccounts},  we have $|E_0|=2|V_0|$.
Note that each connected component of $G_{F_1,0}$ must contain a cycle (since it is unbalanced) and so if $G_{F_1,0}$ has $n$ connected components, $H_1,H_2,\ldots, H_n$ say, then $|E(H_j)|\geq|V(H_j)|$ for each $j$ and,
\[|E(G_{F_1,0})|=\sum_{j=1}^n |E(H_j)|\geq\sum_{j=1}^n|V(H_j)|= |V_0|.\]
Similarly, $|E(G_{F_2,0})|\geq |V_0|$.
Now $|E(G_{F_1,0})|+|E(G_{F_2,0})|=|E_0|=2|V_0|$ and so
 $|E(G_{F_1,0})|= |V_0|=|E(G_{F_2,0})|$.
It follows that $|E(H_j)|=|V(H_j)|$  
for each $j$ and so the connected components of $G_{F_1,0}$  each contain exactly one cycle. By the same argument, the connected components of $G_{F_2,0}$ each contain exactly one cycle. Thus $G_{F_1,0}$ and $G_{F_2,0}$ are both unbalanced spanning mapping graphs in $G_0$.

$(ii)\Rightarrow (i)$
Suppose $(ii)$ holds and let $u$ be a $\chi_0$-symmetric infinitesimal flex of $(G,p)$. 
Then $u_{-v}=-u_v$ for all $v\in V$.
Let $v_0\in V$ and let $H_0^1$ and $H_0^2$ be the connected components of $G_{F_1,0}$ and $G_{F_2,0}$ respectively which contain $[v_0]\in V_0$. 
Since $H_0^i$ contains a unique unbalanced cycle, there exists a path in $G_{F_i}$ from $v_0$ to $-v_0$.
It follows that $u_{v_0} - u_{-v_0}\in \cap_{i=1,2}\ker \varphi_{F_i}=\{0\}$ and so $u_{v_0} = u_{-v_0} = -u_{v_0}$.
Thus $u_{v_0}=0$. Applying this argument to all $v\in V$, we have $u=0$ and so $(G,p,\theta,\tau)$ is $\chi_0$-symmetrically infinitesimally rigid. 
Note that $|E_0|=2|V_0|$ and so $\G$ is also $\chi_0$-symmetrically isostatic.

$(B)$
$(i)\Rightarrow (ii)$
Suppose  there exists a vertex $[v_0]\in V_0 \setminus V(G_{F_1,0})$.
Choose a non-zero vector $x\in \ker\varphi_{F_2}$. For all $v\in V$ define,
\[u_{v} = \left\{ \begin{array}{ll}
x & \mbox{ if } [v]=[v_0],\\
0 & \mbox{ otherwise. } \\
\end{array}\right.
\]
Then $u$ is a non-trivial  $\chi_1$-symmetric infinitesimal flex for $(G,p)$.
This is a contradiction and so $G_{F_1,0}$ is a spanning subgraph 
of $G_0$.
Similarly, $G_{F_2,0}$ is a spanning subgraph of $G_0$.

Suppose $G_{F_1,0}$ is not connected, 
and let $H_0$ be a connected component of $G_{F_1,0}$. Choose a non-zero vector $x\in \ker\varphi_{F_2}$ and for all 
$v\in V$ define, 
\[u_{v} = \left\{ \begin{array}{ll}
x & \mbox{ if } [v]\in V(H_0) ,\\
0 & \mbox{ otherwise. }
\end{array}\right.
\]
Then $u$ is a non-trivial $\chi_1$-symmetric infinitesimal flex for $(G,p)$, which is a contradiction. Thus $G_{F_1,0}$ is a connected  spanning  subgraph of $G_0$.
Similarly, $G_{F_2,0}$ is a connected  spanning  subgraph of $G_0$.
By Corollary \ref{cor:neccounts}, we have  $|E_0|=2|V_0|-2$.
Note that $|E(G_{F_1,0})|\geq |V_0|-1$ and 
$|E(G_{F_2,0})|\geq |V_0|-1$ and so $G_{F_1,0}$ and $G_{F_2,0}$ are both spanning trees in $G_0$.

$(ii)\Rightarrow (i)$
Suppose $(ii)$ holds and let $u$ be a $\chi_1$-symmetric infinitesimal flex for $\G$. 
Then $u_{-v}=u_v$ for all $v\in V$.
Fix $v,w\in V$.
Since $G_{F_1,0}$ is a spanning tree in $G_0$, there exists a path in $G_{F_1,0}$ from $[v]$ to $[w]$. Thus  there either exists a path $P$ in $G_{F_1}$ from $v$ to $w$ or there exists a path $P$ in $G_{F_1}$ from $v$ to $-w$. In the former case it follows directly that $u_v-u_w\in \ker\varphi_{F_1}$ while in the latter case it follows that $u_v-u_w=u_v-u_{-w}\in\ker\varphi_{F_1}$.
Similarly, $u_v-u_w\in \ker\varphi_{F_2}$ and so $u_v=u_w$ for all $v,w\in V$.
Thus $u$ is a trivial infinitesimal flex and so $\G$ is $\chi_1$-symmetrically infinitesimally rigid. Since $|E_0|=2|V_0|-2$, $\G$ is also $\chi_1$-symmetrically isostatic.

$(C)$
$(i)\Rightarrow (ii)$
If $(G,p)$ is infinitesimally rigid then it is both $\chi_0$ and $\chi_1$-symmetrically infinitesimally rigid. By removing edge orbits from $G$ we arrive at a spanning subgraph $A$ such that $(A,p)$ is $\chi_0$-symmetrically isostatic.
By (A),
$A_{F_1,0}$ and $A_{F_2,0}$ are unbalanced spanning map graphs in $G_0$ and so each contains an unbalanced cycle. Similarly, $(G,p)$ contains a spanning subgraph $B$ such that $(B,p)$ is $\chi_1$-symmetrically isostatic. By (B), $B_{F_1,0}$ and $B_{F_2,0}$ are spanning trees in $G_0$. Since $B_{F_i,0}$ is a spanning tree for $i=1,2$, there exists a set of edges in $B_{F_i,0}$ which, when added to $A_{F_i,0}$, form a connected unbalanced spanning map graph $H_{F_i,0}$. This gives the result.

$(ii)\Rightarrow (i)$
Suppose $(ii)$ holds. Then $(G_0,\psi)$ contains a spanning subgraph $H_0$ such that the induced monochrome subgraphs $H_{F_1,0}$ and $H_{F_2,0}$ are edge-disjoint connected unbalanced spanning map graphs. Let $H$ be the covering graph for $H_0$.
By (A), the $\C_2$-symmetric subframework $(H,p)$ is $\chi_0$-symmetrically infinitesimally rigid.
Similarly, note that $H_{F_1,0}$ and $H_{F_2,0}$ both contain  spanning trees in $H_0$ and so by (B), $(H,p)$ is $\chi_1$-symmetrically infinitesimally rigid.
It follows that $(H,p)$, and hence also $(G,p)$, is infinitesimally rigid. 
\end{proof}

\subsection{Existence of rigid grid-like placements with half-turn symmetry}
\label{Sect:Existence}

Recall from Corollary~\ref{cor:neccounts} that if $\G=(G,p,\theta,\tau)$ is a well-positioned, $\mathcal{C}_2$-symmetric and $\chi_0$-symmetrically isostatic bar-joint framework in $(\bR^2, \|\cdot \|_\P)$, where $\P$ is a quadrilateral, then the gain graph $(G_0,\psi)$ for $(G,\theta)$ is $(2,2,0)$-gain-tight. By Theorem~\ref{thm:recurse}, $(G_0,\psi)$ is $(2,2,0)$-gain-tight if it can be generated from vertex-disjoint copies of graphs in $\B$ by applying H1, H2, H3, vertex-to-$K_4$ and vertex splitting moves. We now show that if there exists such a recursive construction sequence, then there exists a half-turn symmetric realisation of $G$ that is   well-positioned and $\chi_0$-symmetrically isostatic in $(\bR^2,\|\cdot\|_\P)$. Overall, this yields the following main combinatorial result for $\chi_0$-symmetrically isostatic frameworks with half-turn symmetry in $(\bR^2, \|\cdot \|_\P)$.

\begin{theorem}  \label{thm:geom}
Let $\|\cdot\|_\P$ be a norm on $\bR^2$ for which $\P$ is a quadrilateral, and let $(G,\theta)$ be a $\mathbb{Z}_2$-symmetric graph. Further, let $(G_0,\psi)$ be  the gain  graph for $(G,\theta)$.
The following are equivalent.
\begin{enumerate}
\item[(i)] There exists a representation $\tau:\mathbb{Z}_2\to \Isom(\mathbb{R}^2)$ 
 and a realisation $p$ such that $\G=(G,p,\theta,\tau)$ is   well-positioned, $\C_2$-symmetric and $\chi_0$-symmetrically isostatic in $(\bR^2,\|\cdot\|_\P)$;
\item[(ii)] $(G_0,\psi)$  is $(2,2,0)$-gain tight;
\item[(iii)] $(G_0,\psi)$ can be constructed from disjoint copies of base graphs in Figure~\ref{fig:gain_cov_graphs} by a sequence of H1a,b,c moves, H2a,b,c,d,e moves, H3a,b,c,d moves, vertex-to-$K_4$ moves, and vertex splitting moves.
\end{enumerate}
\end{theorem}

To show that (iii) implies (i), we rely on Theorem~\ref{thm:C2gridgeom}$(A)$. We split the proof into a number of geometric lemmas. In these lemmas, we will use the notation 
of Section~\ref{subsec:gaingraphs} and write $[v]$ and $[e]$ for a vertex and an edge of the gain graph $(G_0,\psi)$ for a $\mathbb{Z}_2$-symmetric graph $(G,\theta)$, respectively. 
Moreover, we let $\tilde{V}_0=\{\tilde{v}_1,\ldots,\tilde{v}_n\}$ be a choice of representatives for the vertex orbits of $(G,\theta)$.

\begin{lemma} \label{lemma:h1} Let $(G_0,\psi)$ and $(G'_0,\psi')$ be the gain graphs of the $\mathbb{Z}_2$-symmetric graphs $(G,\theta)$ and $(G',\theta')$, respectively and suppose that   $(G_0,\psi)$ is obtained from  $(G'_0,\psi')$ by a H1a, H1b or H1c move. If for $(G'_0,\psi')$ there exists a representation $\tau:\mathbb{Z}_2\to \Isom(\mathbb{R}^2)$ 
 and a realisation $p'$ such that  $\G=(G',p',\theta',\tau)$ is   well-positioned, $\C_2$-symmetric and $\chi_0$-symmetrically isostatic in $(\bR^2,\|\cdot\|_\P)$, then the same is true for $(G_0,\psi)$.
\end{lemma}
\begin{proof} By Theorem~\ref{thm:C2gridgeom}$(A)$, there exists a well-positioned $\C_2$-symmetric realisation $p'$ of  $(G',\theta')$   in $(\bR^2,\|\cdot\|_\P)$   so that the induced monochrome subgraphs $G'_{F_1,0}$ and $G'_{F_2,0}$ of $(G'_0,\psi')$ are both spanning unbalanced map graphs. By Theorem~\ref{thm:C2gridgeom}$(A)$, it now suffices to show that the vertex of $G_0\setminus G'_0$ can be placed in such a way that the corresponding framework $(G,p,\theta, \tau)$ is $\C_2$-symmetric and well-positioned, and the induced monochrome subgraphs $G_{F_1,0}$ and $G_{F_2,0}$ are both spanning unbalanced map graphs in $(G_0,\psi)$.

We fix two points $x_1$ and $x_2$ in the relative interiors of $F_1$ and $F_2$ respectively.

Suppose first that $(G_0,\psi)$ is obtained from $(G'_0,\psi')$ by a H1a move, where $[v]\in G_0\setminus G'_0$ is adjacent to the vertices $[v_1]$ and $[v_2]$ of $G'_0$ with respective gains $\gamma_1$ and $\gamma_2$. 
Set $p_w=p_w'$ for all vertices $w$ of $G$ with $[w]\neq[v]$.
Let $a\in \mathbb{R}^2$ be the point of intersection of the lines $L_1=\{\tau(\gamma_1)p_{\tv_1}+tx_1:t\in \mathbb{R}\}$ and $L_2=\{\tau(\gamma_2)p_{\tv_2}+tx_2:t\in \mathbb{R}\}$ and let $B(a,r)$ be an open ball with centre $a$ and radius $r>0$. Choose $p_{\tv}$ to be any point in $B(a,r)$ which is distinct from $\{p_w:w\in V(G')\}$ and which is not fixed by $\tau(-1)$. Set $p_{-\tv}=\tau(-1)p_{\tv}$.
Then $(G,p,\theta, \tau)$ is a $\C_2$-symmetric bar-joint framework  and, by  applying a small perturbation to $p_{\tilde{v}}$ if necessary, we may assume that $(G,p,\theta, \tau)$ is well-positioned. 
If $r$ is sufficiently small then the induced framework colours for $[v][v_1]$ and $[v][v_2]$ are $[F_1]$ and $[F_2]$ respectively. 
Thus, the induced monochrome subgraphs of $(G_0,\psi)$ are $G_{F_1,0}=G'_{F_1,0}\cup \{[v][v_1]\}$ and $G_{F_2,0}=G'_{F_2,0}\cup \{[v][v_2]\}$.  Clearly, $G_{F_1,0}$  and $G_{F_2,0}$ are spanning unbalanced map graphs of $(G_0,\psi)$.

If $(G_0,\psi)$ is obtained from $(G'_0,\psi)$ by a H1b move, then the proof is completely analogous to the proof above.

Suppose $(G_0,\psi)$ is obtained from $(G'_0,\psi')$ by a H1c move, where $[v]\in G_0\setminus G'_0$ is incident to the unbalanced loop $[e]$ and  adjacent to the vertex $[z]$ of $(G'_0,\psi')$ with gain $\gamma$. Let $a\in \mathbb{R}^2$ be the point of intersection of the lines $L_1=\{\tau(\gamma)p_{\tz}+tx_2:t\in \mathbb{R}\}$ and $L_2=\{tx_1:t\in \mathbb{R}\}$ and let $B(a,r)$ be an open ball with centre $a$ and radius $r>0$. Choose $p_{\tv}$ to be any point in $B(a,r)$ which is distinct from $\{p_w:w\in V(G')\}$ and which is not fixed by $\tau(-1)$. Set $p_{-\tv}=\tau(-1)p_{\tv}$. Then, by  applying a small perturbation to $p_{\tilde{v}}$ if necessary,  $(G,p,\theta, \tau)$ is well-positioned and $\C_2$-symmetric, and the induced monochrome subgraphs of $(G_0,\psi)$ are $G_{F_1,0}=G'_{F_1,0}\cup \{[e]\}$ and $G_{F_2,0}=G'_{F_2,0}\cup \{[v][z]\}$.  Clearly, $G_{F_1,0}$ and  $G_{F_2,0}$ are unbalanced spanning map graphs of $(G_0,\psi)$.
\end{proof}


\begin{lemma} \label{lemma:h2} Let $(G_0,\psi)$ and $(G'_0,\psi')$ be the gain graphs of the $\mathbb{Z}_2$-symmetric graphs $(G,\theta)$ and $(G',\theta')$, respectively and suppose that   $(G_0,\psi)$ is obtained from  $(G'_0,\psi')$ by a H2a, H2b, H2c, H2d, or H2e move. If for $(G'_0,\psi')$ there exists a representation $\tau:\mathbb{Z}_2\to \Isom(\mathbb{R}^2)$ and a realisation $p'$ such that $\G=(G',p',\theta',\tau)$ is   well-positioned, $\C_2$-symmetric and $\chi_0$-symmetrically isostatic in $(\bR^2,\|\cdot\|_\P)$, then the same is true for $(G_0,\psi)$.
\end{lemma}
\begin{proof} By Theorem~\ref{thm:C2gridgeom}$(A)$, there exists a well-positioned $\C_2$-symmetric realisation $p'$ of $(G',\theta')$  in $(\bR^2,\|\cdot\|_\P)$  so that the induced monochrome subgraphs $G'_{F_1,0}$ and $G'_{F_2,0}$ of $(G'_0,\psi')$ are both spanning unbalanced map graphs. By Theorem~\ref{thm:C2gridgeom}$(A)$, it now suffices to show that the vertex of $G_0\setminus G'_0$ can be placed in such a way that the corresponding framework $(G,p,\theta, \tau)$ is $\C_2$-symmetric and well-positioned, and the induced monochrome subgraphs $G_{F_1,0}$ and $G_{F_2,0}$ are both spanning unbalanced map graphs in $(G_0,\psi)$.

We fix two points $x_1$ and $x_2$ in the relative interiors of $F_1$ and $F_2$ respectively.

Suppose first that $(G_0,\psi)$ is obtained from $(G'_0,\psi')$ by a H2a move where $[v]\in G_0\setminus G'_0$ subdivides the edge $[e]=[v_1][v_2]$ into the edges $[e_1]=[v][v_1]$ and $[e_2]=[v][v_2]$ with respective gains $\gamma_1$ and $\gamma_2$, and $[v]$ is also incident to the edge $[e_3]$ with end-vertex $[z]$ and gain $\gamma_3$. Without loss of generality we may assume that $[e]\in G'_{F_1,0}$. Let $a\in \bR^2$ be the point of intersection of the line $L_1$ which passes through the points $\tau(\gamma_1)p_{\tv_1}$ and $\tau(\gamma_2)p_{\tv_2}$, and the line $L_2=\{\tau(\gamma_3)p_{\tz}+tx_2:t\in\bR\}$.
Let $B(a,r)$ be the open ball with centre $a$ and radius $r>0$ and choose $p_{\tv}$ to be a point in $B(a,r)$ which is distinct from $\{p_w:w\in G'\}$ and which is not fixed by $\tau(-1)$.
Set $p_{-\tv}=\tau(-1)p_{\tv}$. Then $(G,p,\theta,\tau)$ is $\C_2$-symmetric and, by  applying a small perturbation to $p_{\tilde{v}}$ if necessary, we may assume it is well-positioned. If $r$ is sufficiently small then  $[e_1]$ and $[e_2]$ have induced framework colour $[F_1]$ and $[e_3]$ has framework colour $[F_2]$. The induced monochrome subgraphs of $(G_0,\psi)$ are $G_{F_1,0}=(G'_{F_1,0}\backslash\{[e]\})\cup \{[e_1],[e_2]\}$ and $G_{F_2,0}=G'_{F_2,0}\cup \{[e_3]\}$. Clearly, $G_{F_1,0}$ and $G_{F_2,0}$ are spanning unbalanced map graphs of $(G_0,\psi)$.

The cases where $(G_0,\psi)$ is obtained from $(G'_0,\psi')$ by a H2b or a H2c move may be proved completely  analogously to the case above for the H2a move.

Next, we suppose that $(G_0,\psi)$ is obtained from $(G'_0,\psi')$ by a H2d move where $[v]\in G_0\setminus G'_0$ subdivides the edge $[e]=[v_1][v_2]$ into the edges $[e_1]=[v][v_1]$ and $[e_2]=[v][v_2]$ with respective gains $\gamma_1$ and $\gamma_2$, and $[v]$ is also incident to the unbalanced loop $[e_3]$. Without loss of generality we may assume that $[e]\in G'_{F_1,0}$. Let $a\in \mathbb{R}^2$ be the point of intersection of the line $L_1$ which passes through the points $\tau(\gamma_1)p_{\tv_1}$ and $\tau(\gamma_2)p_{\tv_2}$ and the line $L_2=\{tx_2:t\in \mathbb{R}\}$, and let $B(a,r)$ be an open ball with centre $a$ and radius $r>0$. (Note that $a$ could possibly be the centre of the rotation $\tau(-1)$, i.e. the origin.) Choose $p_{\tv}$ to be a point in $B(a,r)$ which is distinct from $\{p_w:w\in G'\}$ and which is not fixed by $\tau(-1)$.
Set $p_{-\tv}=\tau(-1)p_{\tv}$. Then $(G,p,\theta,\tau)$ is $\C_2$-symmetric and if $r$ is sufficiently small then  $[e_1]$ and $[e_2]$ have induced framework colour $[F_1]$. Moreover, by  applying a perturbation   to $p_{\tilde{v}}$ within $B(a,r)$ if necessary, we may assume  that $[e_3]$ has framework colour $[F_2]$ and that $(G,p,\theta,\tau)$ is well-positioned. The induced monochrome subgraphs of $(G_0,\psi)$ are $G_{F_1,0}=(G'_{F_1,0}\backslash\{[e]\})\cup \{[e_1],[e_2]\}$ and $G_{F_2,0}=G'_{F_2,0}\cup \{[e_3]\}$. Clearly, $G_{F_1,0}$ and $G_{F_2,0}$ are spanning unbalanced map graphs of $(G_0,\psi)$.

The case where $(G_0,\psi)$ is obtained from $(G'_0,\psi')$ by a H2e move may be proved completely  analogously to the case above for the H2d move. Here $B(a,r)$ is in fact centred at the origin.
\end{proof}


\begin{lemma} \label{lemma:h3} Let $(G_0,\psi)$ and $(G'_0,\psi')$ be the gain graphs of the $\mathbb{Z}_2$-symmetric graphs $(G,\theta)$ and $(G',\theta')$, respectively and suppose that   $(G_0,\psi)$ is obtained from  $(G'_0,\psi')$ by a H3a, H3b, H3c, or H3d move. If for $(G'_0,\psi')$ there exists a representation $\tau:\mathbb{Z}_2\to \Isom(\mathbb{R}^2)$ and a realisation $p'$ such that $\G=(G',p',\theta',\tau)$ is   well-positioned, $\C_2$-symmetric and $\chi_0$-symmetrically isostatic in $(\bR^2,\|\cdot\|_\P)$, then the same is true for $(G_0,\psi)$.
\end{lemma}

\begin{proof} By Theorem~\ref{thm:C2gridgeom}$(A)$, there exists a well-positioned $\C_2$-symmetric realisation $p'$ of  $(G',\theta')$   in $(\bR^2,\|\cdot\|_\P)$   so that the induced monochrome subgraphs $G'_{F_1,0}$ and $G'_{F_2,0}$ of $(G'_0,\psi')$ are both spanning unbalanced map graphs. By Theorem~\ref{thm:C2gridgeom}$(A)$, it now suffices to show that the vertex of $G_0\setminus G'_0$ can be placed in such a way that the corresponding framework $(G,p,\theta, \tau)$ is $\C_2$-symmetric and well-positioned, and the induced monochrome subgraphs $G_{F_1,0}$ and $G_{F_2,0}$ are both spanning unbalanced map graphs in $(G_0,\psi)$.

First we suppose that $(G_0,\psi)$ is obtained from $(G'_0,\psi')$ by a H3a move where $[v]\in G_0\setminus G'_0$ subdivides the edge $[e]=[v_1][v_2]$ into the edges $[e_1]=[v][v_1]$ and $[e_2]=[v][v_2]$, and the edge  $[f]=[v_3][v_4]$  into the edges $[f_1]=[v][v_3]$ and $[f_2]=[v][v_4]$. By switching $[v_1]$ and $[v_3]$ if necessary, we may assume without loss of generality that $[e]$ and $[f]$ have both gain $1$. The edges $[e_1],[e_2],[f_1],[f_2]$ will then also be assigned gain $1$. (The proof for the case where they are all assigned gain $-1$ is analogous.) We distinguish two cases.

Case A: $[e]$ and $[f]$ belong to different induced monochrome subgraphs of $G'_0$, say $[e]\in G'_{F_1,0}$ and $[f]\in G'_{F_2,0}$.  Let $a\in \mathbb{R}^2$ be the point of intersection of the line $L_1$ which passes through the points $p_{\tv_1}$ and $p_{\tv_2}$, and the line $L_2$ which passes through the points $p_{\tv_3}$ and $p_{\tv_4}$. Let $B(a,r)$ be an open ball with centre $a$ and radius $r>0$, and choose $p_{\tv}$ to be a point in $B(a,r)$ which is distinct from $\{p_w:w\in G'\}$ and which is not fixed by $\tau(-1)$. Set $p_{-\tv}=\tau(-1)p_{\tv}$. Then $(G,p,\theta,\tau)$ is $\C_2$-symmetric and, by  applying a small perturbation to $p_{\tilde{v}}$ if necessary, we may assume it is well-positioned. If $r$ is sufficiently small then  $[e_1]$ and $[e_2]$ have induced framework colour $[F_1]$, and $[f_1]$ and $[f_2]$ have framework colour $[F_2]$. The induced monochrome subgraphs of $(G_0,\psi)$ are $G_{F_1,0}=(G'_{F_1,0}\backslash\{[e]\})\cup \{[e_1],[e_2]\}$ and $G_{F_2,0}=(G'_{F_2,0}\backslash\{[f]\})\cup \{[f_1],[f_2]\}$. Clearly, $G_{F_1,0}$ and $G_{F_2,0}$ are spanning unbalanced map graphs of $(G_0,\psi)$.

Case B: $[e]$ and $[f]$ belong to the same induced monochrome subgraph of $G'_0$, say $[e],[f]\in G'_{F_1,0}$. We need the following claim.

\begin{claim}\label{claimh3a} Let $p_1,p_2,p_3,p_4$ be four distinct points in $\mathbb{R}^2$ such that the line segments $p_1p_2$ and $p_3p_4$ both have framework colour $[F_1]$. Let $i\in\{1,2,3,4\}$. Then there exists an open set $N$ in $\mathbb{R}^2$ such that for every point $p_v\in N$, the line segment $p_vp_i$ has framework colour $[F_2]$ and the three line segments $p_vp_j$ with $j\in\{1,2,3,4\}, j\neq i$, have framework colour $[F_1]$.
\end{claim}
\begin{proof} Without loss of generality we may assume that $p_1$ lies to the left of $p_2$, and $p_3$ lies to the left of $p_4$
(see Figure~\ref{fig:h3aclaim} for an illustration). Moreover we may assume  that $i=4$.

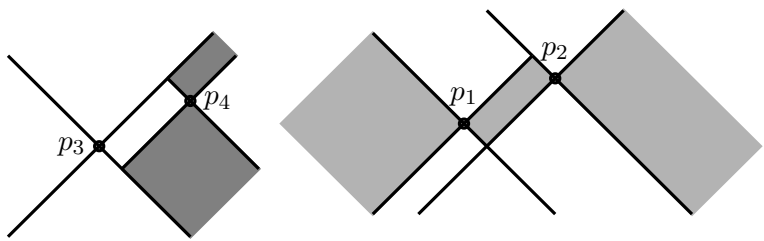
\begin{figure}[htp]
\begin{center}
\begin{tikzpicture}[very thick,scale=0.6]
\tikzstyle{every node}=[circle, draw=black, fill=white, inner sep=0pt, minimum width=5pt];

\filldraw[rectangle, gray](1.5,1.5)--(2,1)--(3,2)--(2.5,2.5)--(1.5,1.5);
\filldraw[rectangle, gray](0.5,-0.5)--(2,1)--(3.5,-0.5)--(2,-2)--(0.5,-0.5);

\filldraw[rectangle, white!70!black](8,0.5)--(6,2.5)--(4,0.5)--(6,-1.5)--(8,0.5);
\filldraw[rectangle, white!70!black](8,0.5)--(8.5,0)--(10,1.5)--(9.5,2)--(8,0.5);
\filldraw[rectangle, white!70!black](10,1.5)--(11.5,3)--(14.5,0)--(13,-1.5)--(10,1.5);

\node [draw=white, fill=white] (a) at (-0.6,0)  {$p_3$};
\node [draw=white, fill=white] (a) at (2.6,1)  {$p_4$};
\node [draw=white, fill=white] (a) at (8,1.1)  {$p_1$};
\node [draw=white, fill=white] (a) at (10,2.1)  {$p_2$};

\draw(0,0) circle(0.1cm);
\draw(2,1) circle(0.1cm);
\draw(8,0.5) circle(0.1cm);
\draw(10,1.5) circle(0.1cm);

\draw(-2,-2)--(2.5,2.5);
\draw(2,-2)--(-2,2);

\draw(2,1)--(0.5,-0.5);
\draw(2,1)--(1.5,1.5);

\draw(2,1)--(3,2);
\draw(2,1)--(3.5,-0.5);

\draw(8,0.5)--(6,2.5);
\draw(10,1.5)--(11.5,3);

\draw(8,0.5)--(8.5,0);
\draw(8.5,0)--(10,-1.5);

\draw(8.5,0)--(10,1.5);
\draw(9.5,2)--(10,1.5);

\draw(9.5,2)--(6,-1.5);
\draw(8.5,0)--(7,-1.5);

\draw(10,1.5)--(13,-1.5);
\draw(9.5,2)--(8.5,3);

\end{tikzpicture}
\end{center}
\vspace{-0.2cm}
\caption{Illustration of the proof of Claim~\ref{claimh3a}.}
\label{fig:h3aclaim}
\end{figure}

We need to find an open set $N$ which lies within the two shaded areas in Figure~\ref{fig:h3aclaim}. Note that the shaded area on the left hand side of Figure~\ref{fig:h3aclaim} is connected, and unbounded from below and above. The shaded area on the right hand side of Figure~\ref{fig:h3aclaim} is also connected, and unbounded from the left and right. Since $p_1,p_2,p_3,p_4$ are distinct points, the shaded areas will  always have a nontrivial intersection, and we may choose $N$ within that intersection.
\end{proof}

Suppose first that $[e]$ and $[f]$ lie on a common (unbalanced) cycle in $G'_{F_1,0}$. Without loss of generality, we may assume that this cycle consists of a path $P_1$ from $[v_2]$ to $[v_3]$ with an odd number of edges with gain $-1$, and a path $P_2$ from $[v_1]$ to $[v_4]$ with an even number of edges with gain $-1$. Then we choose $p_{\tv}$ to be a point  which is distinct from $\{p_w:w\in G'\}$, not fixed by $\tau(-1)$, and such that $(G,p,\theta,\tau)$ is well-positioned and in $(G_0,\psi)$ the edges $[e_1], [e_2], [f_1]$ have framework colour $[F_1]$, and $[f_2]$ has framework colour $[F_2]$. Such a position for $p_{\tv}$ exists by Claim~\ref{claimh3a}. The induced monochrome subgraphs of $(G_0,\psi)$ are $G_{F_1,0}=(G'_{F_1,0}\backslash\{[e],[f]\})\cup \{[e_1],[e_2],[f_1]\}$ and $G_{F_2,0}=(G'_{F_2,0})\cup \{[f_2]\}$. Clearly, $G_{F_2,0}$ is a spanning unbalanced map graph of $(G_0,\psi)$. As for $G_{F_1,0}$, note that the removal of $[e]$ and $[f]$ from $G'_{F_1,0}$ breaks the connected component of $G'_{F_1,0}$ containing $[e]$ and $[f]$ into the two disjoint trees. By adding the vertex $[v]$ and the edges $[e_1],[e_2],[f_1]$, these two trees are reconnected and a single unbalanced cycle (consisting of $P_1$, $[e_2]$ and $[f_1]$) is created in this connected component of $G_{F_1,0}$.

If $[e]$ and $[f]$ do not lie on a common cycle in $G'_{F_1,0}$, but they are still in the same connected component $K'$ of $G'_{F_1,0}$, then we may proceed as above. However, if either $[e]$ or $[f]$, say $[e]$, lies on the unique cycle $C'$ in $K'$, and without loss of generality there exists a path in $K'$ from a vertex in $C'$ to $[v_4]$ that does not include $[v_3]$, then we need to choose $p_{\tv}$ so that $[e_1]$ and $[e_2]$ both have framework colour $[F_1]$, and $[f_1]$ and $[f_2]$ have respective framework colours $[F_1]$ and $[F_2]$. This guarantees that the unbalanced cycle $C=C'\backslash\{[e]\}\cup\{[e_1],[e_2]\}$ in the corresponding component of $G_{F_1,0}$ is unique.

If $[e]$ and $[f]$ lie in different connected components $K'$ and $K''$ of $G'_{F_1,0}$, then we may again proceed as above. However, care needs to be taken in the case where either $[e]$ or $[f]$, say $[e]$, lies on the unique cycle $C'$ in $K'$, and $[f]$ does not lie on the cycle of $K''$. In this case we  choose $p_{\tv}$ so that $[e_1]$ and $[e_2]$ have framework colours $[F_1]$ and $[F_2]$, and  $[f_1]$ and $[f_2]$ both have framework colour $[F_1]$, so that $G_{F_1,0}$ will not have a connected component with two cycles. Similarly, if neither $[e]$ nor $[f]$ lie on the cycle in their respective connected components, then we need to choose $p_{\tv}$ so that the three new edges with framework colour $[F_1]$ do not give rise to a connected component of  $G_{F_1,0}$ that has two cycles.

Next, we suppose that $(G_0,\psi)$ is obtained from $(G'_0,\psi')$ by a H3b move where $[v]\in G_0\setminus G'_0$, and the H3b move deletes the edges $[e]=[v_1][v_2]$ and $[f]=[v_1][v_3]$ and adds the edges 
$[e_1]=[v][v_1]$ and $[e_1']=[v][v_1]$, and the edges  $[e_2]=[v][v_2]$ and $[e_3]=[v][v_3]$. By switching $[v_2]$ and $[v_3]$ if necessary, we may assume that $[e]$ and $[f]$ have gain $1$. The edges $[e_1]$ and $[e_1']$ are assigned the gains $1$ and $-1$, respectively, and the edges $[e_2],[e_3]$ are assigned the gains $1$ and $-1$, respectively. We distinguish two cases.

Case A: $[e]$ and $[f]$ belong to different induced monochrome subgraphs of $G'_0$, say $[e]\in G'_{F_1,0}$ and $[f]\in G'_{F_2,0}$. Let $a\in \bR^2$ be the point of intersection of the line $L_1$ which passes through the points $p_{\tv_1}$ and $p_{\tv_2}$, and the line $L_2$ which passes through the points $\tau(-1)p_{\tv_1}$ and $\tau(-1)p_{\tv_3}$. Let $B(a,r)$ be the open ball with centre $a$ and radius $r>0$ and choose $p_{\tv}$ to be a point in $B(a,r)$ which is distinct from $\{p_w:w\in G'\}$ and which is not fixed by $\tau(-1)$.
Set $p_{-\tv}=\tau(-1)p_{\tv}$. Then $(G,p,\theta, \tau)$ is $\C_2$-symmetric and, by  applying a small perturbation to $p_{\tilde{v}}$ if necessary, we may assume it is well-positioned. If $r$ is sufficiently small then  $[e_1]$ and $[e_2]$ have induced framework colour $[F_1]$, and $[e_1']$ and $[e_3]$ have framework colour $[F_2]$. The induced monochrome subgraphs of $(G_0,\psi)$ are $G_{F_1,0}=(G'_{F_1,0}\backslash\{[e]\})\cup \{[e_1],[e_2]\}$ and $G_{F_2,0}=(G'_{F_2,0}\backslash\{[f]\})\cup \{[e_1'],[e_3]\}$. Clearly, $G_{F_1,0}$ and $G_{F_2,0}$ are spanning unbalanced map graphs of $(G_0,\psi)$.

Case B: $[e]$ and $[f]$ belong to the same induced monochrome subgraph of $G'_0$, say $[e],[f]\in G'_{F_1,0}$.

 Suppose first that $[e]$ and $[f]$ lie on a common cycle in  $G'_{F_1,0}$. Then we may apply Claim~\ref{claimh3a} to the points $p_{\tv_1}, \tau(-1)p_{\tv_1},p_{\tv_2},\tau(-1)p_{\tv_3}$ to find a position for $p_{\tv}$ so that it is distinct from $\{p_w:w\in G'\}$, not fixed by $\tau(-1)$, and such that $(G,p,\theta,\tau)$ is well-positioned and in $(G_0,\psi)$ the edges $[e_1],[e_1']$ and $[e_2]$ have framework colour $[F_1]$, and  $[e_3]$ has framework colour $[F_2]$. The induced monochrome subgraphs of $(G_0,\psi)$ are $G_{F_1,0}=(G'_{F_1,0}\backslash\{[e],[f]\})\cup \{[e_1],[e_1'],[e_2]\}$ and $G_{F_2,0}=G'_{F_2,0}\cup \{[e_3]\}$. Clearly, $G_{F_1,0}$ and $G_{F_2,0}$ are spanning unbalanced map graphs of $(G_0,\psi)$.

Suppose next that $[e]$ and $[f]$ do not lie on a common cycle in  $G'_{F_1,0}$. If either $[e]$ or $[f]$, say $[e]$, lies on a cycle in $G'_{F_1,0}$, then we proceed as above, but we need to choose $p_{\tv}$ so that the edges $[e_1],[e_1']$ and $[e_3]$ have framework colour $[F_1]$, and  $[e_2]$ has framework colour $[F_2]$ to guarantee that $G_{F_1,0}$ is a spanning unbalanced map graph of $(G_0,\psi)$. If neither $[e]$ nor $[f]$ lie on a cycle in $G'_{F_1,0}$, then we need to distinguish two cases. Let $C$ be the cycle in the connected component of $G'_{F_1,0}$ containing the edges $[e],[f]$. If there exists a path in $G'_{F_1,0}$ from a vertex in $C$ to $[v_2]$ or $[v_3]$ that does not include $[v_1]$, then we may again proceed as above. Otherwise we choose a position for $p_{\tv}$ so that it is distinct from $\{p_w:w\in G'\}$, not fixed by $\tau(-1)$, and such that $(G,p,\theta,\tau)$ is well-positioned and in $(G_0,\psi)$ the edges $[e_1]$ and $[e_1']$ have respective framework colours $[F_1]$ and $[F_2]$, and  $[e_2]$ and $[e_3]$ both have framework colour $[F_1]$.

Suppose next that $(G_0,\psi)$ is obtained from $(G'_0,\psi')$ by a H3c move where $[v]\in G_0\setminus G'_0$, and the H3c move deletes the unbalanced loop $[e]=[v_1][v_1]$ and the edge $[f]=[v_2][v_3]$ and adds the edges $[e_1]=[v][v_1]$ and $[e_1']=[v][v_1]$ with respective gains $\gamma_1=1\neq -1= \gamma_1'$, and the edges  $[e_2]=[v][v_2]$ and $[e_3]=[v][v_3]$ with respective gains $\gamma_2$ and $\gamma_3$.  This case is completely analogous to the H3a case.
If $[e]$ and $[f]$ belong to different induced monochrome subgraphs of $G'_0$, say $[e]\in G'_{F_1,0}$ and $[f]\in G'_{F_2,0}$, then we may choose $p_{\tv}$ so that  $[e_1]$ and $[e_1']$ have induced framework colour $[F_1]$, and $[e_2]$ and $[e_3]$ have induced framework colour $[F_2]$. If $[e]$ and $[f]$ belong to the same induced monochrome subgraph of $G'_0$, say $[e],[f]\in G'_{F_1,0}$, then Claim~\ref{claimh3a} applies, and we may choose $p_{\tv}$ so that  $[e_1]$ and $[e_1']$ have induced framework colour $[F_1]$, and $[e_2]$ and $[e_3]$ have respective framework colours $[F_1]$ and $[F_2]$, so that $G_{F_1,0}$ and $G_{F_2,0}$ are spanning unbalanced map graphs of $(G_0,\psi)$. 

If $(G_0,\psi)$ is obtained from $(G'_0,\psi')$ by a H3d move, then we may again proceed analogously to the  H3a (or H3c) case. Note that if the loops $[e]$ and $[f]$ that are deleted in the H3d move are both in the same induced monochrome subgraph of $G'_0$, say $[e],[f]\in G'_{F_1,0}$, then they must lie in separate connected components of $G'_{F_1,0}$ (since they are both unbalanced cycles). So their removal results in two disjoint trees. With the addition of the vertex $[v]$ and the edges $[e_1],[e_1'],[e_2]$, these two trees are connected and a single (unbalanced) cycle is created in this connected component of $G_{F_1,0}$. 
\end{proof}


\begin{lemma} \label{lemma:furthergeomoves} Let $(G_0,\psi)$ and $(G'_0,\psi')$ be the gain graphs of the $\mathbb{Z}_2$-symmetric graphs $(G,\theta)$ and $(G',\theta')$, respectively and suppose that   $(G_0,\psi)$ is obtained from  $(G'_0,\psi')$ by a vertex-to-$K_4$ or vertex splitting move. If for $(G'_0,\psi')$ there exists a representation $\tau:\mathbb{Z}_2\to \Isom(\mathbb{R}^2)$ and a realisation $p'$ such that $\G=(G',p',\theta',\tau)$ is   well-positioned, $\C_2$-symmetric and $\chi_0$-symmetrically isostatic in $(\bR^2,\|\cdot\|_\P)$, then the same is true for $(G_0,\psi)$.
\end{lemma}
\begin{proof} By Theorem~\ref{thm:C2gridgeom}$(A)$, there exists a well-positioned $\C_2$-symmetric realisation $p'$ of  $(G',\theta')$   in $(\bR^2,\|\cdot\|_\P)$   so that the induced monochrome subgraphs $G'_{F_1,0}$ and $G'_{F_2,0}$ of $(G'_0,\psi')$ are both spanning unbalanced map graphs. By Theorem~\ref{thm:C2gridgeom}$(A)$, it now suffices to show that the vertex (or vertices) of $G_0\setminus G'_0$ can be placed in such a way that the corresponding framework $(G,p,\theta, \tau)$ is $\C_2$-symmetric and well-positioned, and the induced monochrome subgraphs $G_{F_1,0}$ and $G_{F_2,0}$ are both spanning unbalanced map graphs in $(G_0,\psi)$.

We fix two points $x_1$ and $x_2$ in the relative interiors of $F_1$ and $F_2$ respectively.

First we suppose that $(G_0,\psi)$ is obtained from $(G'_0,\psi')$ by a  vertex-to-$K_4$-move, where the vertex $[v]$ of  $(G'_0,\psi')$ (which may be incident to an unbalanced loop $[e]$) is replaced by a copy of $K_4$ with a trivial gain labelling (and $[e]$ is replaced by the edge $[f]$ with gain $-1$). Suppose without loss of generality that the loop $[e]$ (if present) has framework colour $[F_2]$. As shown in Figure~\ref{fig:gain_cov_graphs}, $K_4$ has a well-positioned  placement in $(\bR^2,\|\cdot\|_\P)$ where the two monochrome subgraphs are both trees. Moreover, we may scale this realisation  so that all of the vertices of the $K_4$ lie in a ball of arbitrarily small radius.  
Let $B(p_{\tv},r)$ be the open  ball with centre $p_{\tv}$ and radius $r>0$. Choose a placement of the representative vertices of the new $K_4$ to lie within $B(p_{\tv},r)$ such that the vertices are distinct from $\{p_w:w\in V(G')\backslash\{\tv\}\}$, none of the vertex placements are fixed by $\tau(-1)$ and the resulting placement of the new $K_4$ is such that the monochrome subgraphs are both trees. If $r$ is sufficiently small then the edge  $[f]$ (if present) has the induced framework colour $[F_2]$.
It can be assumed that the corresponding $\C_2$-symmetric placement of $G$ is well-positioned. Moreover, the induced monochrome subgraphs $G_{F_1,0}$ and $G_{F_2,0}$ of $G_0$  are clearly spanning unbalanced  map graphs  of $(G_0,\psi)$. 

Finally, we suppose that $(G_0,\psi)$ is obtained from $(G'_0,\psi')$ by a vertex split, where the vertex $[v]$ of  $(G'_0,\psi')$ (which  is replaced by the vertices $[v_0]$ and $[v_1]$) is incident to the edge $[v][u]$ with trivial gain and the edges $[v][u_i]$, $i=1,\ldots, t$, in $G'_0$. Without loss of generality we may assume that $[v][u]\in G'_{F_1,0}$. If we choose $p_{\tv_0}=p_{\tv}$ and $p_{\tv_1}$ to be a point on the line $L=\{p_{\tv}+tx_2:t\in \bR\}$ which is sufficiently close to $p_{\tv}$, then  the induced framework colour for $[v_0][v_1]$ is $[F_2]$ and the induced framework colour for $[v_0][u]$ and $[v_1][u]$ is $[F_1]$. (Again we may assume the framework is well-positioned). Moreover, all other edges of $(G'_0,\psi')$ which have been replaced by new edges in $(G_0,\psi)$ clearly retain their induced framework colouring if $p_{\tv_1}$   is chosen sufficiently close to $p_{\tv}$.   It is now easy to see that for such a  placement of $\tv_0$ and $\tv_1$, both  $G_{F_1,0}$ and $G_{F_2,0}$ are spanning unbalanced  map graphs  of $(G_0,\psi)$. 
\end{proof}


We are now ready to prove Theorem~\ref{thm:geom}.

\begin{proof}
As mentioned earlier, (i) $\Rightarrow$ (ii) follows from Corollary~\ref{cor:neccounts}, and (ii) $\Rightarrow$ (iii) follows from Theorem~\ref{thm:recurse}.

(iii) $\Rightarrow$ (i): We employ induction on the number of vertices of $G_0$. By Theorem~\ref{thm:C2gridgeom}$(A)$, for each of the  base gain graphs  there exists  a representation $\tau:\mathbb{Z}_2\to \Isom(\mathbb{R}^2,\|\cdot\|_\infty)$ and a realisation $p$ such that  $\G=(G,p,\theta, \tau)$ is  well-positioned, $\C_2$-symmetric and $\chi_0$-symmetrically isostatic in $(\bR^2,\|\cdot\|_\infty)$, as indicated in Figure~\ref{fig:gain_cov_graphs}. (The two induced spanning map graphs  $G_{F_1,0}$ and $G_{F_2,0}$ are shown in gray and black colour, respectively.) Since $(\bR^2,\|\cdot\|_\P)$ is isometrically isomorphic to $(\bR^2,\|\cdot\|_\infty)$, there also exists a well-positioned, $\C_2$-symmetric and $\chi_0$-symmetrically isostatic realisation for each of the base graphs in $(\bR^2,\|\cdot\|_\P)$.

Let $n\geq 5$ and suppose (i) holds for all gain   graphs satisfying (iii) with at most $n-1$ vertices. Let $(G_0,\psi)$ have $n$ vertices, and let $(G'_0,\psi')$ be the penultimate graph in the construction sequence of $(G_0,\psi)$. By the induction hypothesis, there exists a realisation $p'$ of the covering graph $G'$ of $(G'_0,\psi')$  in $(\bR^2,\|\cdot\|_\P)$  so that $(G',p',\theta',\tau)$ is   well-positioned, $\C_2$-symmetric and $\chi_0$-symmetrically isostatic in $(\bR^2,\|\cdot\|_\P)$.

If $(G_0,\psi)$ is obtained from $(G'_0,\psi')$ by a H1a, H1b, or H1c move, then the result follows from Lemma~\ref{lemma:h1}. If $(G_0,\psi)$ is obtained from $(G'_0,\psi')$ by a H2a, H2b, H2c, H2d or H2e move, then the result follows from Lemma~\ref{lemma:h2}. If $(G_0,\psi)$ is obtained from $(G'_0,\psi')$ by a H3a, H3b, H3c, or H3d move, then the result follows from Lemma~\ref{lemma:h3}. Finally, if $(G_0,\psi)$ is obtained from $(G'_0,\psi')$ by a vertex-to-$K_4$ or vertex splitting  
move, then the result follows from Lemma~\ref{lemma:furthergeomoves}.

\end{proof}

Next we establish the $\chi_1$-symmetric counterpart to Theorem~\ref{thm:geom}. The proof of this result is much simpler than the proof of Theorem~\ref{thm:geom} since the characterisation of $(2,2,2)$-gain-tight gain graphs in terms of a recursive construction sequence is significantly less complex than the one for $(2,2,0)$-gain-tight gain graphs.

\begin{theorem}  \label{thm:geom2}
Let $\|\cdot\|_\P$ be a norm on $\bR^2$ for which $\P$ is a quadrilateral, and let $(G,\theta)$ be a $\mathbb{Z}_2$-symmetric graph.  Further, let $(G_0,\psi)$ be  the gain graph for $(G,\theta)$.
The following are equivalent.
\begin{enumerate}
\item[(i)] There exists a representation $\tau:\mathbb{Z}_2\to \Isom(\mathbb{R}^2)$ and a realisation $p$ such that $\G=(G,p,\theta,\tau)$ is   well-positioned, $\C_2$-symmetric and $\chi_1$-symmetrically isostatic in $(\bR^2,\|\cdot\|_\P)$;
\item[(ii)] $(G_0,\psi)$  is $(2,2,2)$-gain tight;
\item[(iii)] $(G_0,\psi)$ can be constructed from $K_1$ by a sequence of H1a,b moves, H2a,b moves, vertex-to-$K_4$ moves, and vertex splitting moves.
\end{enumerate}
\end{theorem}

\begin{proof} (i) $\Rightarrow$ (ii): This follows again from Corollary~\ref{cor:neccounts}.

 (ii) $\Rightarrow$ (iii): The proof proceeds by induction. Since $(G_0,\psi)$ is $(2,2,2)$-gain-tight, it has no loops. If $G_0$ has a vertex $[v]$ of degree $2$, then it is clearly admissible (via an inverse H1a or H1b move). So suppose $G_0$ has no degree 2 vertices. By $(2,2,2)$-gain-tightness,  $G_0$ has a vertex $[v]$ of degree 3 with at least two neighbours. If $[v]$ has exactly two neighbours, then it is admissible (via an inverse H2b move, see also Lemma~\ref{lem:3threeii}). Thus we may assume that every degree 3 vertex $[v]$ of $G_0$ has exactly three neighbours. It is easy to see that $[v]$ is admissible (via an inverse H2a move) unless it is contained in a balanced copy of $K_4$. (See also Lemma~\ref{lem:3threei}).  The vertices of this $K_4$ cannot induce any additional edges since $(G_0,\psi)$ is $(2,2,2)$-gain-tight. Denote this copy of $K_4$ as $K$. We may apply an inverse vertex-to-$K_4$ move unless there is a vertex $[x]\notin K$ and edges $[x][a]$ and $[x][b]$ with equal gains, where $[a],[b]\in K$. By switching, we may assume that both gains are $1$. We may now apply an inverse vertex splitting move, contracting either $[x][a]$ and $[x][b]$, unless there exist vertices $[y]$ and $[z]$ that are distinct from the vertices of $K$ and $[x]$ so that $[y][x]$ and $[y][a]$ are edges in $(G_0,\psi)$ with the same gain, and $[z][x]$ and $[z][b]$ are edges in $(G_0,\psi)$ with the same gain. By switching, we may again assume that the gains of these edges are all $1$. We continue in this fashion, thereby constructing an increasing chain of subgraphs of $(G_0,\psi)$ which are all $(2,2,2)$-gain tight and whose edges have all gain $1$. (Note that at each step a new vertex is introduced for otherwise $(2,2,2)$-gain-sparsity is violated.) This sequence terminates after finitely many steps at which point there will be an admissible inverse vertex splitting move.

(iii) $\Rightarrow$ (i): Using Theorem~\ref{thm:C2gridgeom}$(B)$, this result may be proved completely analogously to Theorem~\ref{thm:geom} (iii) $\Rightarrow$ (i).
\end{proof}

\section{Concluding remarks}

One may be tempted to try to combine Theorems \ref{thm:geom} and \ref{thm:geom2} to combinatorially characterise infinitesimal rigidity for half-turn symmetric frameworks. However this seems to be non-trivial. In particular, given a gain graph which contains a spanning $(2,2,2)$-gain-tight subgraph and a spanning $(2,2,0)$-gain-tight subgraph it is not clear that a placement exists that preserves both the colourings needed to apply Theorem~\ref{thm:C2gridgeom}$(C)$.

It is also natural to try to extend Theorems~\ref{thm:geom} and \ref{thm:geom2} to higher-order groups, such as the cyclic group $\C_4$ generated by a 4-fold rotation in the $\ell^1$- or $\ell^\infty$-plane. In this case, Corollary~\ref{cor:neccounts} provides necessary gain-sparsity conditions for $\chi$-symmetric infinitesimal rigidity. However, we are currently lacking analogues of Theorems~\ref{thm:C2gridgeom}$(A)$ and \ref{thm:C2gridgeom}$(B)$ to prove the sufficiency of these counts.

There is a second form of vertex splitting, known as the vertex-to-4-cycle move \cite{LMW,NOP14}, which one may use instead of vertex splitting to give analogous inductive constructions to Theorem \ref{thm:recurse} and Theorem \ref{thm:geom2} (ii) $\Leftrightarrow$ (iii). In fact, in the case of $(2,2,0)$-gain-tight gain graphs, this alternative gives a non-trivial simplification to the proof, replacing the maximal balanced triangle sequence considerations with a direct counting argument. However in both the symmetric and anti-symmetric contexts this construction operation does not seem to be amenable to finding appropriate rigid placements.

In \cite{NS} symmetric rigidity is considered for frameworks in Euclidean space that are restricted to move on a fixed surface. In particular the matroidal classes of $(2,2,2), (2,2,1)$ and $(2,2,0)$-gain-tight gain graphs are the relevant sparsity types for frameworks restricted to an infinite circular cylinder. Hence our recursive construction of $(2,2,0)$-gain-tight gain graphs may be useful in establishing an analogue of Theorem \ref{thm:geom} for  the appropriate symmetry group, that is for half-turn symmetric frameworks on the cylinder with rotation axis perpendicular to the axis of the cylinder.


\vspace{0.2cm}

\end{document}